\documentclass[letterpaper, 11pt]{amsart}
\pdfoutput=1
\usepackage{amsmath,amssymb,amsthm,pinlabel,tikz,hyperref,mathrsfs,color, thmtools}
\usepackage{versions}
\usepackage{fullpage}
\DeclareSymbolFontAlphabet{\amsmathbb}{AMSb}%
\usepackage{verbatim}
\usepackage{float}
\usepackage{caption}
\usepackage{subcaption}
\usepackage{enumitem}
\usepackage{tikz-cd}
\newcommand{\nc}{\newcommand}
\nc{\dmo}{\DeclareMathOperator}
\dmo{\ra}{\rightarrow}
\dmo{\Prob}{\mathbb{P}}
\dmo{\E}{\mathbb{E}}
\dmo{\N}{\mathbb{N}}
\dmo{\Z}{\mathbb{Z}}
\dmo{\Q}{\mathbb{Q}}
\dmo{\R}{\mathbb{R}}
\dmo{\C}{\mathcal{C}}
\dmo{\X}{\mathcal{X}}
\dmo{\U}{\mathcal{U}}
\dmo{\T}{\mathcal{T}}
\dmo{\F}{\mathcal{F}}
\dmo{\AC}{\mathcal{AC}}
\dmo{\MIN}{\mathcal{MIN}}
\dmo{\Mod}{Mod}
\dmo{\w}{\omega}
\dmo{\PMod}{PMod}
\dmo{\PMF}{\mathcal{PMF}}
\dmo{\Mat}{Mat}
\dmo{\supp}{supp}
\dmo{\UE}{\mathcal{UE}}
\dmo{\vol}{vol}
\dmo{\B}{B}
\dmo{\PB}{PB}
\dmo{\PR}{PSL(2,\mathbb{R})}
\dmo{\GL}{GL(k, \mathbb{C})}
\dmo{\SL}{SL(2, \mathbb{Z})}
\dmo{\Isom}{Isom}
\dmo{\RP}{\mathbb{R} \mathrm{P}}
\dmo{\I}{\mathcal{I}}
\dmo{\el}{\ell_{\C}}
\dmo{\NN}{\mathcal{N}}
\dmo{\rk}{rank}
\dmo{\tr}{tr}
\dmo{\llangle}{\langle\langle}
\dmo{\rrangle}{\rangle\rangle}
\dmo{\Unif}{Unif}
\dmo{\Out}{Out}
\dmo{\diam}{\operatorname{diam}}
\dmo{\Aut}{\operatorname{Aut}}
\dmo{\sumRho}{\mathscr{B}}
\dmo{\stopping}{\vartheta}
\dmo{\diffPivot}{\mathcal{P}}
\dmo{\diffEvenPivot}{\mathcal{Q}}
\dmo{\varGam}{\Upsilon}
\dmo{\prodSeq}{\Pi}
\dmo{\NSupp}{N_{supp}}
\dmo{\KSleep}{\mathnormal{K_{sleep}}}
\dmo{\Devi}{\upsilon}
\dmo{\DeviStr}{\iota}
\dmo{\DeviDisc}{\varrho}
\dmo{\DeviUni}{\varsigma}
\dmo{\wVar}{\check{Z}}
\dmo{\muVar}{\eta}
\dmo{\sublinear}{\Delta}
\dmo{\pivotComplete}{\tilde{\mathrm{S}}}
\dmo{\pivotRW}{\mathcal{S}}
\dmo{\EvenPivot}{Q}
\dmo{\Proj}{Pr}
\dmo{\axes}{\mathbf{Y}}

\usetikzlibrary{decorations.markings, patterns}
\usetikzlibrary{decorations.pathmorphing}
\usetikzlibrary{calc, arrows, positioning}
\tikzset{->-/.style={decoration={
  markings,
  mark=at position #1 with {\arrow{>}}},postaction={decorate}}}

\nc{\nt}{\newtheorem}
\nt{theorem}{Theorem}

\newtheorem{thm}{{\bf Theorem}}[section]
\newtheorem{definition}[thm]{Definition}
\newtheorem{lem}[thm]{{\bf Lemma}}
\newtheorem{cor}[thm]{{\bf Corollary}}
\newtheorem{conv}[thm]{{\bf Convention}}

\newtheorem{prop}[thm]{{\bf Proposition}}
\newtheorem{fact}[thm]{Fact}

\newtheorem{remark}[thm]{Remark}

\newtheorem{obs}[thm]{Observation}

\title[Random walks and contracting elements I]{Random walks and contracting elements I: Deviation inequality and Limit laws}

\date{\today}
\author{Inhyeok Choi}
\email{%
        inhyeokchoi@kias.re.kr
        }
\address{%
		June E Huh Center for Mathematical Challenges, KIAS\\
		85 Hoegiro Dongdaemun-gu, Seoul 02455, Republic of Korea
}

\begin{document}

\begin{abstract}
We study random walks on metric spaces with contracting isometries. In this first article of the series, we establish sharp deviation inequalities by adapting Gou{\"e}zel's pivotal time construction. As an application, we establish the exponential bounds for deviation from below, central limit theorem, law of the iterated logarithms and the geodesic tracking of random walks on mapping class groups and CAT(0) spaces.


\noindent{\bf MSC classes:} 20F67, 30F60, 57M60, 60G50
\end{abstract}

\maketitle

\tableofcontents

\section{Introduction} \label{section:intro}

%%%SPACEHOLDER

This is the first in the series of articles concerning random walks on metric spaces with contracting elements. This series is a reformulation of the preprint \cite{choi2022limit} announced by the author, aiming for a more concise and systematic exposition.

Let $G$ be a countable group of isometries of a metric space $(X, d)$ with basepoint $o \in X$. We consider the random walk generated by a probability measure $\mu$ on $G$, which entails the product $Z_{n} = g_{1} \cdots g_{n}$ of independent random isometries $g_{i}$'s chosen with law $\mu$. We are interested in the asymptotic behavior of a random path $(Z_{n})_{n >0}$ seen by $X$, or in other words, the behavior of a random \emph{orbit path} $(Z_{n} o)_{ n >0}$ on $X$. For instance, we can ask the following questions:
\begin{itemize}
\item Does the random variable $\frac{1}{n} d(o, Z_{n} o)$ converge to a constant almost surely?
\item Does the random variable $\frac{1}{\sqrt{n}} d(o, Z_{n} o)$ converge in law to a Gaussian law?
\item How fast does $\Prob\left(an \le d(o, Z_{n} o) \le bn\right)$ decay for $0 \le a \le b$?
\end{itemize}
These questions are associated with the so-called \emph{moment conditions}. For each $p>0$ we define the $p$-th moment of $\mu$ by \[
\E_{\mu} [d(o, go)^{p}] = \int_{G} d(o, go)^{p} \, d\mu,
\]
and the exponential moment (with a parameter $K>0$) of $\mu$ by \[
\E_{\mu} \operatorname{exp}(K d(o, go)) = \int_{G} e^{K d(o, go)} \, d\mu.
\]
In the classical setting of $X = \mathbb{R}$, the previous three questions are answered when $\mu$ has finite first moment, finite second moment and finite exponential moment, respectively.

A particularly interesting examples come from isometric actions on non-positively curved spaces. This setting includes Gromov hyperbolic groups (\cite{benoist2016central}, \cite{boulanger2022large}, \cite{gouezel2022exponential}); relatively hyperbolic groups (\cite{sisto2017tracking}, \cite{qing2020sublinearly}); groups with nontrivial Floyd boundary (\cite{gekhtman2021martin}); the mapping class group of a finite-type hyperbolic surface acting on Teichm{\"u}ller space (\cite{kaimanovich1996poisson}, \cite{horbez2018clt}, \cite{dahmani2018spectral}, \cite{baik2021linear}) or the curve complex (\cite{maher2010linear}, \cite{maher2010heegaard}, \cite{maher2011random}); the outer automorphism group of a finite-rank free group acting on the Culler-Vogtmann Outer space (\cite{horbez2018clt}, \cite{dahmani2018spectral}) and the free factor complex; groups acting on CAT(0) spaces (\cite{karlsson1999ergodic}, \cite{karlsson2006lln}, \cite{fernos2018the-furstenberg-poisson}, \cite{le-bars2022random}, \cite{le-bars2022central}). 

In this paper, we propose a unified theory for random walks on the aforementioned spaces. We first study the case where $X$ possesses strongly contracting isometries (see Convention \ref{conv:strong}), and $\mu$ is non-elementary (see Subsection \ref{subsection:RW}). This condition is mild enough to cover all the aforementioned spaces (except for Outer space, which will be studied carefully in \cite{choi2022random3} due to the asymmetry issue). At the same time, this is just the right amount of restriction that leads to limit laws under optimal moment conditions.

We also present a parallel theory for metric spaces with weakly contracting isometries (see Convention \ref{conv:weak}). As a result, we obtain limit laws on hierarchically hyperbolic groups (HHGs) with optimal moment conditions. We describe the case of mapping class group for concreteness.

\begin{theorem} \label{thm:expBdMod}
Let $G$ be the mapping class group of a finite-type surface, let $d$ be a word metric on $G$, let $(Z_{n})_{n\ge0}$ be the random walk generated by a non-elementary probability measure $\mu$ on $G$, and let \[
\lambda = \lambda(\mu) := \lim_{n \rightarrow \infty} \frac{1}{n} \E [d(id, Z_{n})]
\]
be the drift of $\mu$ on $G$. Then for each $0 < L < \lambda$, the probability $\Prob (d(id, Z_{n}) \le Ln)$ decays exponentially as $n$ tends to infinity.
\end{theorem}

This is an analogue of the result of Gou{\"e}zel \cite[Theorem 1.3]{gouezel2022exponential}, who established the exponential bound for Gromov hyperbolic spaces. Note that, for every admissible probability measure $\mu$ on the mapping class group $G$, the spectral radius of $\mu$ is strictly smaller than $1$ due to the non-amenability of $G$ \cite{kesten1959full}. Combining this with the exponential growth of $G$, one can obtain $L>0$ for which $\Prob( d(id, Z_{n}) \le Ln)$ decays exponentially. Hence, the nontrivial part of Theorem \ref{thm:expBdMod} is that $L$ can be as close to $\lambda$ as we want.

We also obtain the deviation inequalities with optimal moment conditions (see Proposition \ref{prop:deviation}). Combining this with Mathieu-Sisto's theory \cite{mathieu2020deviation}, we establish the central limit theorem (CLT) and law of the iterated logarithms (LIL) on mapping class groups.

\begin{theorem}\label{thm:CLT}
Let $G$ be the mapping class group of a finite-type hyperbolic surface, let $d$ be a word metric on $G$, and let $(Z_{n})_{n\ge 0}$ be the random walk generated by a non-elementary probability measure $\mu$ on $G$ with finite second moment. Then there exists $\sigma(\mu)\ge 0$ such that $\frac{1}{\sqrt{n}} (d(id, Z_{n}) - n \lambda(\mu))$ converges in law to the Gaussian law $\mathcal{N}(0, \sigma(\mu))$ of variance $\sigma(\mu)^{2}$, and moreover, \[
\limsup_{n \rightarrow \infty}  \frac{d(id, Z_{n}) - n\lambda(\mu)}{\sqrt{2n \log \log n}}  = \sigma(\mu) \quad \textrm{almost surely.}
\]
\end{theorem}

In acylindrically hyperbolic groups, Mathieu and Sisto established CLT for random walks with finite exponential moment (\cite[Theorem 13.4]{mathieu2020deviation}). We strengthen their result by weakening the moment condition.

Lastly, we address the geodesic tracking of random paths.

\begin{theorem} \label{thm:tracking}
Let $G$ be the mapping class group of a finite-type surface, let $d$ be a word metric on $G$, and let $(Z_{n})_{n\ge 0}$ be the random walk generated by a non-elementary measure $\mu$ on $G$.

\begin{enumerate}
\item Let $p>0$ and suppose that $\mu$ has finite $p$-th moment. Then for almost every sample path $(Z_{n})_{n\ge 0}$, there exists a geodesic $\gamma$ on $G$ such that \[
\lim_{n \rightarrow \infty} \frac{1}{n^{1/p}} d(Z_{n}, \gamma) = 0.
\]
\item If $\mu$ has finite exponential moment, then there exists $K < \infty$ such that the following holds. For almost every sample path $(Z_{n})_{n\ge 0}$, there exists a geodesic $\gamma$ on $G$ such that \[
\limsup_{n \rightarrow \infty} \frac{1}{\log n} d(Z_{n}, \gamma) < K.
\]
\end{enumerate}
\end{theorem}

For finitely supported random walks, Sisto established the deviation rate $d(Z_{n}, \gamma)  = O(\sqrt{n \log n})$ (\cite[Theorem 1.2]{sisto2017tracking}). Later, Qing, Rafi and Tiozzo obtained the rate $d(Z_{n}, \gamma) = O(\log^{3g - 3 + b}(t))$, where $g$ and $b$ denote the genus and the number of punctures of the surface (\cite[Theorem C]{qing2020sublinearly}). We refine these results by suggesting the deviation rate $O(\log(t))$ for random walks with finite exponential moment. 

In full generality, the main results  hold in the setting of Convention \ref{conv:strong} and Convention \ref{conv:weak}. In particular, Theorem \ref{thm:expBdMod}, \ref{thm:CLT} and \ref{thm:tracking} apply to random walks on rank-1 CAT(0) spaces. This extends the author's previous work \cite{choi2023central} that deals with Gromov hyperbolic spaces and Teichm{\"u}ller space, and recovers several results by Le Bars \cite{le-bars2022random}, \cite{le-bars2022central}.

To obtain the main theorems, we blend the pioneering theories due to  Gou{\"e}zel \cite{gouezel2022exponential} and due to Mathieu and Sisto \cite{mathieu2020deviation}. Gou{\"e}zel's method effectively captures the alignment of the orbit path on $X$(see Subsection \ref{subsection:pivotState}), while Mathieu-Sisto's technique provides the desired limit theorems when appropriate deviation inequalities are given. Both of these theories rely on the Gromov hyperbolicity of the ambient space. Our contribution is to replace the Gromov hyperbolicity with weaker notion of hyperbolicity. In particular, we obtain large deviation principle, CLT and geodesic tracking on (possibly non-proper) CAT(0) spaces. Moreover, we generalize Mathieu-Sisto's theory by lifting the moment condition, leading to the exponential bounds for the escape to infinity and CLT for random walks without finite exponential moment.

\subsection{Context} \label{subsection:context}

Random walks on groups have often been studied via their actions on Gromov hyperbolic spaces. For instance, random walks on Teichm{\"u}ller space and Outer space have been understood by coupling them with the curve complex and the free factor complex, respectively (\cite{horbez2018clt}, \cite{dahmani2018spectral}). A similar strategy was recently pursued for proper CAT(0) spaces by Le Bars (\cite{le-bars2022random}, \cite{le-bars2022central}), building upon a new hyperbolic model for CAT(0) spaces (\cite{petyt2024hyperbolic}).

These strategies eventually depend on the following ingredients: \begin{itemize}
\item the non-atomness of the stationary measure on the Gromov boundary (\cite[Proposition 5.1]{maher2018random});
\item CLT for martingales arising from Busemann cocycles (\cite[Theorem 4.7]{benoist2016central});
\item linear progress with exponential decay (\cite[Theorem 1.2]{maher2012exp}), or
\item linear progress using the acylindricity of the action (\cite[Theorem 9.1, Proposition 9.4]{mathieu2020deviation}).
\end{itemize}
The first two items require a nice (e.g., compact) boundary structure of $X$. These boundary structures are also available in some class of non-Gromov-hyperbolic spaces (such as Teichm{\"u}ller space, Outer space and finite-dimensional CAT(0) cube complices  -- see \cite{fernos2018the-furstenberg-poisson}, \cite{fernos2018the-furstenberg-poisson}, \cite{fernos2024contact}) but are hard to come by in the general case. 

To establish the third item, Maher considered a stopping time that arises when a random path penetrates nested shadows, which relies on moment conditions: see \cite{maher2012exp} and \cite{sunderland2020linear}. For the last item, Mathieu and Sisto assumed finite exponential moment condition to couple the random paths on $G$ with the corresponding paths on $G$ in probability.

It is not straightforward to apply the aforementioned strategies to, say, random walks on (non-proper) CAT(0) spaces. Even in well-known settings such as Gromov hyperbolic groups, moment conditions are often necessary. Our goal is to lift these restrictions: we want a structure for random walks on a wide class of spaces $X$ that: \begin{itemize}
\item does not assume global Gromov hyperbolicity of $X$;
\item does not rely on any boundary structure of $X$;
\item does not assume any moment condition \emph{a priori}, and
\item effectively captures the `alignment' of a sample path on $X$.
\end{itemize}

The first goal was studied by Sisto in \cite{sisto2018contracting}. Not assuming global Gromov hyperbolicity of $X$, Sisto presented a random walk theory using strongly contracting isometries, which are found in both Gromov hyperbolic spaces and CAT(0) spaces. Note that the existence of strongly contracting isometries also has implications on the growth problem and counting problem (\cite{arzhantseva2015growth}, \cite{yang2014growth}, \cite{yang2019statistically}, \cite{yang2020genericity}, \cite{legaspi2022constricting} and \cite{coulon2022patterson}).

The second goal was pursued by Mathieu and Sisto for acylindrically hyperbolic groups in \cite{mathieu2020deviation}, establishing deviation inequalities without referring to the boundary of $X$.

The first and the second goals were also pursued by Boulanger, Mathieu, Sert and Sisto in \cite{boulanger2022large}. They discuss Corollary \ref{cor:LDPMod} for Gromov hyperbolic spaces and pointed out the versatility of Schottky sets in other spaces. For more detail, see Section \ref{section:LDP}.

All the goals except the first one were achieved in Gou{\"e}zel's recent paper \cite{gouezel2022exponential}. In \cite{gouezel2022exponential}, Gou{\"e}zel establishes Theorem \ref{thm:expBdMod} for Gromov hyperbolic spaces by recording the Schottky directions aligned along a random path. Such a recording, called the set of pivotal times, grows linearly with exponential decay. More importantly, this growth is uniform and \emph{is independent of the intermediate non-Schottky steps}. 

Our theory achieves the 4 goals in the setting of Convention \ref{conv:strong}. For this purpose, we combine Gou{\"e}zel's pivotal time construction with Sisto's theory of random walks involving strongly contracting isometries. This was indirectly pursued for Teichm{\"u}ller space in \cite{choi2023central}. Our usage of strongly contracting isometries is also hugely influenced by Yang's series of papers (\cite{yang2014growth}, \cite{yang2019statistically}, \cite{yang2020genericity}) in the context of counting problems.

Although strongly contracting isometries are found in various groups, it is not known whether the Cayley graph of a mapping class group possesses strongly contracting isometries. A related issue  arises when one considers a group $G$ that is quasi-isometric to another group $H$. Having a strongly contracting isometry is not passed through quasi-isometries: it is even not preserved under the change of finite generating set of a group \cite[Theorem 4.19]{arzhantseva2019negative}.

This is why we provide a parallel theory in the language of weakly contracting isometries. We note that having a weakly contracting \emph{infinite quasigeodesic} is stable under quasi-isometry. Strictly speaking, our setting is not stable under quasi-isometry: we consider two coarsely equivariant $G$-actions, one involving weak contraction and the other one involving strong contraction. Nevertheless, the present theory is an attempt towards QI-invariant random walk theory. We record recent breakthrough in this direction by Goldborough and Sisto \cite{goldsborough2021markov}, showing that certain QI-invariant group-theoretic property (that involves an action on a hyperbolic space) guarantees a CLT for simple random walks.

\subsection{Strategy} \label{subsection:strategy}
Morally, contracting directions constitute a tree-like structure. As a toy model, consider  \[
G = F_{2} \ast \Z^{2}= \langle a, b, c, d \, |\,  cd = dc \rangle
\] acting on its Cayley graph $X$. A geodesic $\gamma = abaaba$ in $X$ is composed of edges $e_{1} = [id, a]$, $e_{2} = [a, ab]$, $e_{3} = [ab, aba]$ and so on. The geodesicity of $\gamma$ forces the local alignment among $e_{i}$'s: $e_{i}$ projects onto $e_{i+1}$ at the beginning point of $e_{i+1}$ and $e_{i+1}$ projects onto $e_{i}$ at the ending point of $e_{i}$. Conversely, this local alignment implies that $\gamma$ is geodesic. (This is false when $e_{i}$'s are directions in a flat, e.g., $e_{1} = [id, c]$, $e_{2}  = [c, cd]$ and $e_{3} = [cd, cdc^{-1}]$.) The same conclusion holds even if we insert edges in the flats in between $e_{i}$'s. For example, consider \[
e_{1} = [c, ca], \,\,e_{2} = [cacd, cacdb^{2}], \,\,g = cacdb^{2} cd.
\]
Observe that $(id, e_{1})$, $(e_{1}, e_{2})$ and $(e_{2}, g)$ satisfy the local alignment conditions. This forces that $e_{1}$ and $e_{2}$ are subsegments of any geodesic between $id$ and $g$ even if such a geodesic is not unique due to flat parts. We will formulate this more precisely in the alignment lemma in Section \ref{section:strongAlign}.

We will then construct many independent ``tree-like'' directions. In our example, the set \[
S_{M, m} = \{(s_{1}s_{2} \cdots s_{M})^{m} : s_{i} \in \{a, b\} \}
\]
consists of $2^{M}$ directions in the free factor. We have the following property: \begin{enumerate}
\item For any $x \in X$, $d(id, [x, s^{\pm 1}])< M$ for all but at most 1 element $s \in S_{M, m}$.
\item For all $s \in S_{M, m}$, the geodesic $[s^{-1}, s]$ passes through $id$.
\end{enumerate}
This property will be captured by the notion of Schottky sets (Definition \ref{dfn:Schottky}). Note that one can increase the cardinality of $S_{M, m}$ by taking larger $M$.

Let us now consider the random walk $Z$ generated by a probability measure $\mu$ with $\mu(a), \mu(b) > 0$. Then for any $M, m>0$, each element of the Schottky set $S_{M, m}$ is admitted by $\mu^{\ast Mm}$. By decomposing $\mu^{\ast Mm}$ into a uniform measure on $S_{M, m}$ and the remainder, a random path $(Z_{n})_{n}$ can be modelled by the concatenation of some non-Schottky isometries $w_{i}$'s and Schottky isometries $s_{i}$'s, where the timing for Schottky progresses are given by a renewal process. That means, for a large $K$, a random word $Z_{n} = g_{1} \cdots g_{n}$ is of the form \[
Z_{n} = w_{0} s_{1} w_{1} \cdots s_{n/K} w_{n/K},
\]
where $s_{i}$'s are drawn from $S_{M, m}$. Now Gou{\"e}zel's construction of pivotal times provides a large $K'$ such that the following holds: among $\{1, \ldots, n/K\}$, we can pick indices $i(1) < \ldots < i(n/KK')$ at which the Schottky segment is aligned along the entire progress, i.e., $w_{0}s_{1} \cdots w_{i(k)} [id, s_{i(k)}]$'s are subsegments of $[id, Z_{n}]$ $(\ast$). Now pick $x \in X$. We have plenty of Schottky isometries available for the slot $s_{i(k)}$'s. By choosing the right choice among them (i.e., by pivoting), we can also assure that $(x, w_{0}s_{1} \cdots w_{i(k)} [id, s_{i(k)}])$ is aligned. Combined with $(\ast$), this means that we have a bound of $d(id, [x, Z_{n}])$ in terms of an initial subsegment $w_{0}s_{1} \cdots w_{i(k)}$ of the random path. All these phenomena are exponentially generic (see Lemma \ref{lem:Devi}). We subsequently obtain deviation inequalities (Proposition \ref{prop:deviation}), central limit theorem and geodesic tracking. A more involved combinatorial model for random paths leads to the large deviation principle.

In this example, the contracting property of a tree-like edge $e$ is as strong as possible: any geodesic $\gamma$ connecting the left and the right of $e$'s passes through $e$. We study two variants of such a contracting property. If we require that $\gamma$ passes through a bounded neighborhood of $e$, we say that $e$ is strongly contracting. If we require that $\gamma$ passes through a $\log(\diam(e))$-neighborhood of $e$, than we say that $e$ is weakly contracting. The argument so far also works for strongly contracting directions, up to a finite error. A more delicate argument is required for weakly contracting isometries. We will deal with these notions  in Part \ref{part:strong} and Part \ref{part:weak}, respectively.

\subsection*{Acknowledgments}
The author thanks Hyungryul Baik, Kunal Chawla, Ilya Gekhtman, Vivian He, Sang-hyun Kim, Joseph Maher, Hidetoshi Masai, Yulan Qing, Kasra Rafi,  Samuel Taylor, Giulio Tiozzo and Wenyuan Yang for helpful discussions. The author is indebted to the anonymous referee's helpful and careful comments. The author is also grateful to the American Institute of Mathematics and the organizers and the participants of the workshop ``Random walks beyond hyperbolic groups'' in April 2022 for helpful and inspiring discussions.

The author is supported by Samsung Science \& Technology Foundation (SSTF-BA1702-01 and SSTF-BA1301-51) and by a KIAS Individual Grant (SG091901) via the June E Huh Center for Mathematical Challenges at KIAS. This work constitutes part of the author's PhD thesis.

\section{Preliminaries} \label{section:prelim}

Before entering Part \ref{part:strong}, we review basic notions and lemmata. We fix a metric space $(X, d)$ and a basepoint $o \in X$. For $x, y , z \in X$, we define the \emph{Gromov product} of $x$ and $z$ with respect to $y$ by \[
(x, z)_{y} := \frac{1}{2} \big(d(x, y) + d(y, z) - d(x, z)\big).
\]

\subsection{Paths} \label{subsection:path}

Let $A$ and $B$ be subsets of $X$. \emph{$A$  is $K$-coarsely contained in $B$} if $A$ is contained in the $K$-neighborhood of $B$. $A$ and $B$ are \emph{$K$-coarsely equivalent} if $A$ is $K$-coarsely contained in $B$ and vice versa. $A$ is \emph{$K$-coarsely connected} if for every $x, y \in A$ there exists a chain $x = a_{0}, a_{1}, \ldots, a_{n} = y$ of points in $A$ such that $d(a_{i}, a_{i+1}) \le K$ for each $i$.

A \emph{path} on $X$ is a map $\gamma : I \rightarrow X$ from a $1$-coarsely connected subset $I$ of $\mathbb{R}$, called a domain, to $X$. A \emph{subdomain} $J$ of $I$ is of the form $I \cap [a, b]$ for some $a, b \in \mathbb{R}$. The restriction of $\gamma$ on $J$ is called a \emph{subpath} of $\gamma$. We denote this subpath by $\gamma|_{[a, b]}$.

For paths $\gamma : I \rightarrow X$ and $\gamma' : I' \rightarrow X$, we say that $\gamma'$ is a \emph{reparametrization of $\gamma$} when there exists a non-decreasing map $\rho : I' \rightarrow I$ such that $\gamma' = \gamma \circ \rho$. We say that two paths $\kappa : I \rightarrow X$ and $\eta : J \rightarrow X$ are \emph{$K$-fellow traveling} if there exists a reparametrization $\kappa' : J \rightarrow X$ of $\kappa$ such that $d(\kappa'(t), \eta(t)) \le K$ for every $t \in J$. In this case, note that the images of $\kappa$ and $\eta$ are within Hausdorff distance $K$ and the endpoints of $\kappa$ and $\eta$ are pairwise $K$-near. By abuse of notation, for a path $\kappa : I \rightarrow X$, $\kappa$ will often refer to the set-theoretical image $\kappa(I)$ of $\kappa$. For instance, when we say that a path $\kappa: I \rightarrow X$ is $K$-close to a point $x$, it means $d(\kappa(t), x) < K$ for some $t \in I$.

We say that $X$ is \emph{geodesic} if for each pair of points $x, y \in X$ there exists a geodesic connecting $x$ to $y$. Given two points $x, y \in X$, we denote by $[x, y]$ an arbitrary geodesic connecting $x$ to $y$. 

Let $[x, y]$ be a geodesic on $X$ and $A_{1}, \ldots, A_{N}$ be subsets of $[x, y]$. We say that $A_{1}, \ldots, A_{N}$ are \emph{in order from left to right} if $d(x, x_{1}) \le d(x, x_{2}) \le \ldots \le d(x, x_{N})$ for any choices of $x_{i} \in A_{i}$.

We will construct a path for a sequence of isometries as follows. Given a sequence $\alpha = (\phi_{1}, \ldots, \phi_{k})$ of isometries of $X$, we denote the product of its entries $\phi_{1} \cdots \phi_{k}$ by $\prodSeq(\alpha)$. Now let  \[
x_{mk + i} := \Pi(s)^{m}\phi_{1} \cdots \phi_{i} o= (\phi_{1} \cdots \phi_{k})^{m} \phi_{1} \cdots \phi_{i} o
\] for each $m \in \Z$ and $i = 0, \ldots, k-1$; see Figure \ref{fig:seqOrbit}. We let  $\Gamma^{m}(\alpha) := (x_{0}, x_{1}, \ldots, x_{mk})$ when $m \ge 0$ and $\Gamma^{m}(\alpha) := (x_{0}, x_{-1}, \ldots, x_{mk})$ when $m <0$. For $m = \pm 1$, we also use a simpler notation \[\begin{aligned}
\Gamma^{+}(s) &:= (x_{0},x_{1}, \ldots, x_{k}), \\
\Gamma^{-}(s)&:= (x_{0},x_{-1}, \ldots, x_{-k}).
\end{aligned}
\]
In other words, we write:
\[\begin{aligned}
 \Gamma^{+}(\phi_{1}, \ldots, \phi_{k}) &:= (o, \,\,\phi_{1} o, \quad\phi_{1} \phi_{2} o,\quad \quad\ldots, \,\,\phi_{1}\phi_{2} \cdots \phi_{k} o), \\
\Gamma^{-}(\phi_{1}, \ldots, \phi_{k}) &:= (o, \,\,\phi_{k}^{-1} o,\,\, \phi_{k}^{-1} \phi_{k-1}^{-1} o,\,\, \ldots, \,\,\phi_{k}^{-1} \cdots \phi_{1}^{-1} o).
\end{aligned}
\]
Given a path $\gamma = (y_{1}, \ldots, y_{N})$, we denote by $\bar{\gamma}$ its \emph{reversal}, defined by \[
\bar{\gamma} := (y_{N}, \ldots, y_{1}).
\] For example, the reversal of $\Gamma^{-}(\alpha)$ is denoted by $\bar{\Gamma}^{-}(\alpha)$, which is \[\begin{aligned}
\bar{\Gamma}^{-}(\phi_{1},\ldots, \phi_{k}) := &\,\,(x_{-k}, x_{-(k-1)}, \ldots, x_{0}) \\
=&\,\,
(\phi_{k}^{-1} \cdots \phi_{1}^{-1} o, \,\,\ldots, \,\, \phi_{k}^{-1} \phi_{k-1}^{-1} o, \,\, \phi_{k}^{-1} o, \,\, o).
\end{aligned}
\]

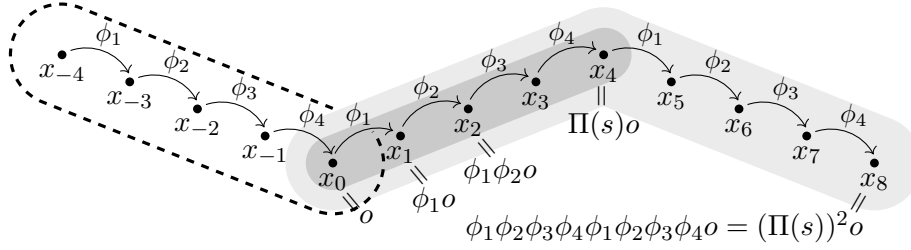
\begin{figure}
\begin{tikzpicture}[scale=0.9]
\def\c{0.7}
\def\ce{0.75}
\def\b{0.4}

\fill[black!8, rotate=21.74] (0, -\c) -- (4.32, -\c) arc (-90:90: \c) -- (0, \c) arc (90:270:\c);
\fill[black!8, shift={(8, 0)}, rotate=180-21.74] (0, -\c) -- (4.32, -\c) arc (-90:90: \c) -- (0, \c) arc (90:270:\c);

\draw[dashed, very thick,  rotate=180-21.74] (0.3, -\ce) -- (4.32, -\ce) arc (-90:90: \ce) -- (0, \ce) arc (90:240:\ce);

\fill[opacity=0.14, rotate=21.74] (0, -\b) -- (4.32, -\b) arc (-90:90: \b) -- (0, \b) arc (90:270:\b);

\fill (-4, 1.6) circle (0.065);
\foreach \i in {1, 2, 3, 4}{
\fill (1*\i - 4, 1.6 - 0.4*\i) circle (0.065);
\fill (1*\i, 0.4*\i) circle (0.065);
\fill (1*\i + 4, 1.6 - 0.4*\i) circle (0.065);
\draw[->, shift={(0.03 + \i - 1, 0.14+ 0.4*\i - 0.4)}, rotate=21.4] (0, 0) arc (130:50:0.72);
\draw[->, shift={(0.126 + \i - 1+4, 0.064- 0.4*\i + 0.4+1.6)}, rotate=-21.4] (0, 0) arc (130:50:0.72);
\draw[->, shift={(0.126 + \i - 1-4, 0.064- 0.4*\i + 0.4+1.6)}, rotate=-21.4] (0, 0) arc (130:50:0.72);
}

\foreach \i in {1, 2, 3}{
\draw[shift={(0.25 + \i - 1, -0.6 + 0.4*\i - 0.4)}] (0, 0) node[rotate=-50] {$=$};
}
\draw (0, -0.28) node {$x_{0}$};
\draw (1, 0.12) node {$x_{1}$}; 
\draw (2, 0.52) node {$x_{2}$}; 
\draw (3, 0.92) node {$x_{3}$}; 
\draw (0.39, 0.72) node {\small{$\phi_{1}$}};
\draw (0.39 + 1, 0.72 + 0.4) node {\small{$\phi_{2}$}};
\draw (0.39 + 2, 0.72 + 0.4*2) node {\small{$\phi_{3}$}};
\draw (0.39 + 3, 0.72 + 0.4*3) node {\small{$\phi_{4}$}};
\begin{scope}[shift={(-8, 0)}]
\draw (4.7, 1.9) node {\small{$\phi_{1}$}};
\draw (4.7 + 1, 1.9 - 0.4) node {\small{$\phi_{2}$}};
\draw (4.7 + 2, 1.9 - 0.4*2) node {\small{$\phi_{3}$}};
\draw (4.7 + 3, 1.9 - 0.4*3) node {\small{$\phi_{4}$}};
\end{scope}
\draw (-4, 1.32) node {$x_{-4}$}; 
\draw (-3, 0.92) node {$x_{-3}$}; 
\draw (-2, 0.52) node {$x_{-2}$}; 
\draw (-1, 0.12) node {$x_{-1}$}; 

\draw (5, 0.92) node {$x_{5}$}; 
\draw (6, 0.52) node {$x_{6}$}; 
\draw (7, 0.12) node {$x_{7}$}; 
\draw (8, -0.28) node {$x_{8}$};
\draw (4.7, 1.9) node {\small{$\phi_{1}$}};
\draw (4.7 + 1, 1.9 - 0.4) node {\small{$\phi_{2}$}};
\draw (4.7 + 2, 1.9 - 0.4*2) node {\small{$\phi_{3}$}};
\draw (4.7 + 3, 1.9 - 0.4*3) node {\small{$\phi_{4}$}};

\draw (0.49, -0.8) node {$o$};
\draw (0.49 + 1, -0.8 + 0.25) node {$\phi_{1} o$};
\draw (0.47 + 2, -0.8 + 0.68) node {$\phi_{1}\phi_{2} o$};
\draw (4, 0.4*4-0.28) node {$x_{4}$};
\draw (4, 0.4*4-0.6) node[rotate=90] {$=$};
\draw (4, 0.4*4-1.05) node {$\prodSeq(s)o$};

\draw (8 - 0.212, -0.595) node[rotate=60] {$=$};
\draw (4.9, -0.94) node {$\phi_{1}\phi_{2}\phi_{3}\phi_{4}\phi_{1}\phi_{2}\phi_{3}\phi_{4} o = (\prodSeq(s))^{2} o$};

\end{tikzpicture}
\caption{Axes associated with a sequence of isometries $s = (\phi_{1}, \phi_{2}, \phi_{3}, \phi_{4})$. Points inside the darker shadow constitute $\Gamma^{+}(s)$, and those inside the lighter shadow constitute $\Gamma^{2}(s)$. Points in the dashed region constitute $\Gamma^{-}(s)$.}
\label{fig:seqOrbit}
\end{figure}

\subsection{Strong contraction} \label{subsection:contracting}

Given a subset $A$ of $X$, we define the closest point projection $\pi_{A} : X \rightarrow 2^{A}$ onto $A$ by \[
\pi_{A}(x) := \big\{ a \in A : d(x, a) = d(x, A)\big\}.
\]
Note that $\pi_{A}(x)$ is nonempty for each $x \in X$ when $A$ is a closed and locally compact set.

\begin{definition}\label{dfn:contractingSet}
Let $K>0$. A subset $A$ of $X$ is \emph{$K$-strongly contracting} if the following holds for the closest point projection $\pi_{A}$: \[
\diam_{X}\big(\pi_{A}(x) \cup \pi_{A}(y)\big) \le K
\]
for all $x, y \in X$ that satisfy $d_{X}(x, y) \le d_{X}(x, A)$.
\end{definition}

A $K$-strongly contracting $K$-quasigeodesic is called a \emph{$K$-contracting axis}. A lemma follows:

\begin{lem}\label{lem:projBdd}
Let $A$ be a $K$-strongly contracting subset of $X$. Then the closest point projection $\pi_{A} : X \rightarrow A$ is $(1, 4K)$-coarsely Lipschitz, i.e., for each $x, y \in X$ we have\[
\diam(\pi_{A}(x) \cup \pi_{A}(y)) < d(x, y) + 4K
\]
\end{lem}

This lemma is well-known in various forms (\cite[Lemma 2.11]{arzhantseva2015growth}, \cite[Lemma 2.4]{sisto2018contracting} and \cite[Proposition 2.4(4)]{yang2019statistically}). The explicit constant $4K$ is given as a consequence of Lemma \ref{lem:BGIPHausdorff}.

\begin{lem}[{\cite[Proposition 2.2 (3)]{yang2020genericity}}]\label{lem:BGIPRestriction}
For each $K>1$ there exists a constant $K' = K'(K)$ such that any subpath of a $K$-contracting axis is a $K'$-contracting axis.
\end{lem}

\begin{lem}[{\cite[Lemma 2.15]{arzhantseva2015growth}}, {\cite[Proposition 2.2(2)]{yang2020genericity}}]\label{lem:ContractingHausdorff}
Let $A$ and $A'$ be coarsely equivalent subsets of $X$. Then $A$ is strongly contracting if and only if $A'$ is strongly contracting.
\end{lem}

\begin{definition}\label{dfn:contractingIso}
An isometry $g$ of $X$ is \emph{strongly contracting} if its orbit $\{g^{i} o\}_{i \in \Z}$ is a strongly contracting quasigeodesic.
\end{definition}

\begin{definition} \label{dfn:independence}
We say that isometries $g$ and $h$ of $X$ are \emph{independent} if for any $x \in X$ the map \[
(m, n) \mapsto d(g^{m} o, h^{n} o)
\] is proper, i.e., $\{ (m, n) : d(g^{m} o, h^{n} o) < M\}$ is bounded for each $M > 0$.
\end{definition}

The following lemma will be proved in Subsection \ref{subsection:contractingGeod}.

\begin{lem}\label{lem:indepEquiv}
Two strongly contracting isometries $g$ and $h$ of $X$ are independent if and only if $\pi_{\{ g^{i} o : i \in \Z\}}(\{h^{i} o : i \in \Z\})$ and $\pi_{\{ h^{i} o : i \in \Z\}}(\{g^{i} o : i \in \Z\})$ have finite diameters.
\end{lem}

\subsection{Weak contraction}\label{subsection:weakContracting}

This subsection only matters in Part \ref{part:weak}; readers interested in Part \ref{part:strong} may skip this subsection.

\begin{comment}
\begin{definition}\label{dfn:weakContractingSet}
Let $K>0$, let $(\tilde{X}, \tilde{d})$ be a metric space and $\tilde{A}$ be a subset of $\tilde{X}$. A map $\tilde{\pi} : \tilde{X} \rightarrow \tilde{A}$ is called a \emph{$K$-projection} if \[
\tilde{d}(\tilde{\pi}(\tilde{x}), \tilde{a}) \le K \tilde{d}(\tilde{x}, \tilde{a}) + K
\]
for any $\tilde{x} \in \tilde{X}$ and $\tilde{a} \in \tilde{A}$. Note that a $K$-projection $\tilde{\pi} : \tilde{X} \rightarrow \tilde{A}$ satisfies \begin{equation}\label{eqn:KProjection}
\tilde{d}(\tilde{x}, \tilde{\pi}(\tilde{x})) \le \inf_{\tilde{a} \in \tilde{A}} \left( \tilde{d}(\tilde{x}, \tilde{a}) + \tilde{d}(\tilde{a}, \tilde{\pi}(\tilde{x})) \right) \le (K+1) \tilde{d}(\tilde{x}, \tilde{A}) + K
\end{equation}
for all $\tilde{x} \in \tilde{X}$.

$\tilde{A}$ is said to be \emph{$K$-weakly contracting} if there exists a $K$-projection $\tilde{\pi}_{\tilde{A}} : X \rightarrow \tilde{A}$ such that \[
\diam_{\tilde{X}} \Big( \tilde{\pi}_{\tilde{A}} (\tilde{x}) \cup \tilde{\pi}_{\tilde{A}}(\tilde{y})\Big) \le K
\]
for all $\tilde{x}, \tilde{y} \in \tilde{X}$ that satisfy $\tilde{d}(\tilde{x}, \tilde{y}) \le  \tilde{d}(\tilde{x}, \tilde{A}) / K$.

An isometry $g$ of $\tilde{X}$ is \emph{weakly contracting} if the orbit $\{g^{i} \tilde{o}\}$ of $\tilde{o} \in \tilde{X}$ by $g$ is a weakly contracting quasigeodesic.
\end{definition}
\end{comment}

\begin{definition}\label{dfn:weakContractingSet}
Let $K>0$ and $A \subset X$. A \emph{$K$-projection} onto $A$ is a $K$-coarsely Lipschitz map $\pi : X \rightarrow A$ such that $d(a, \pi(a)) \le K$ for each $a \in A$. Note that  for each $x \in X$ we have\begin{equation}\label{eqn:KProjection}
\begin{aligned}
d(x, \pi(x)) &\le  \inf_{a \in A} \big[d(x, a) + d(a, \pi(a)) + d(\pi(a), \pi(x))\big] \\
& \le  \inf_{a \in A} \big[(K+1) d(x, a) + 2K\big] \le (K+1) d(x, A) + 2K
\end{aligned}
\end{equation}

A set $A$ is  \emph{$K$-weakly contracting} if there exists a $K$-projection $\pi_{A}$ such that\begin{equation}\label{eqn:weak}
\diam_{X} \Big( \pi_{A}(x) \cup \pi_{A}(y)\Big) \le K
\end{equation}
holds for all $x, y \in X$ that satisfy $d(x, y) \le d(x, A)/K$.
\end{definition}

\begin{lem}\label{lem:quickLanding}
For each $K, M>1$ there exists $K' > K$ such that the following holds. 

Let $x, y \in X$. Let $A$ be a $K$-weakly contracting set such that $d(x, A) \ge K'$ and such that $\diam\big(\pi_{A}(x)\cup \pi_{A}(y)\big) \ge K'$. Then there exists $p \in[x, y]$ such that $\diam\big(\pi_{A}(x)\cup \pi_{A}(p)\big) \le 2K'$ and such that either:\[
d(x, A) \ge M d(p, A) \quad \textrm{or} \quad M d(x, A) \le d(p, A).
\]
\end{lem}

\begin{proof}
We set $K' = K^{2}\big( M(M+7)(K+1) + 1\big)$.

Let $\eta : [0, L] \rightarrow X$ be a geodesic connecting $x$ to $y$. Note that for \[
\tau := \inf \Big\{ 0 \le t\le L : \diam\big(\pi_{A}(x) \cup \pi_{A}(\eta(t)) \big) > K'+ K\Big\}
\]
The $K$-coarse Lipschitzness of $\pi_{A}$ and Inequality \ref{eqn:weak} imply\[\begin{aligned}
\lim_{\epsilon \rightarrow 0+} \diam\big( \pi_{A}\big(\eta(\tau)\big) \cup \pi_{A}\big(\eta(\tau+\epsilon)\big) \Big) &\le K, \\
\diam\big(\pi_{A}(x) \cup \pi_{A}(\eta(\tau)) \big) \ge (K'+ K) - K &= K'.
\end{aligned}
\]
Hence, by replacing $y$ with $\eta(\tau)$, we may assume $\diam\big(\pi_{A}(x) \cup \pi_{A}(\eta(t))\big) \le K' + K$ for $t \in [0, L]$. If  $d(\eta(t), A) < \frac{1}{M} d(x, A)$ for some $t \in [0, L]$, then we are done; suppose not. We inductively take\[
t_{0} := 0, \quad
t_{i} := \min \left\{ t_{i-1} + \frac{1}{MK} d(x, A), \,\, L \right\} \quad( i> 0).
\]
The process halts at step $N$ when $t_{N}$ reaches $L$. We then have \[
d\big(\eta(t_{i-1}), \eta(t_{i})\big) = t_{i} - t_{i-1} \le \frac{1}{MK} d(x, A) \le \frac{1}{K} d\big(\eta(t_{i-1}), A\big)
\]
for each $i$. Using Inequality \ref{eqn:weak}, we deduce \[
\diam \big( \pi_{A}(x) \cup \pi_{A}(y)\big) \le \sum_{i=1}^{N} \diam \big(\pi_{A}(\eta(t_{i-1})) \cup \pi_{A}(\eta(t_{i}))\big) \le NK.
\]
Since the LHS is at least $K'$, we have $N \ge K' /K \ge 2KM(M+7)(K+1) + 1$.

Meanwhile, $t_{i} - t_{i-1} = d(x, A)/MK$ holds for $i \le N-1$. This implies \[
d(x, y) \ge t_{N-1} - t_{0} \ge (N-1) \frac{1}{MK} d(x, A),
\]
and considering the assumption $d(x, A) \ge K'\ge K$ we deduce\[\begin{aligned}
d(x, y) - \diam \big(\pi_{A}(x) \cup \pi_{A}(y)\big) &\ge (N-1) \frac{1}{MK} d(x, A) - (K'+ K)\\
&\ge MK(M+7)(K+1) \cdot \frac{1}{MK} d(x, A) - 2K'\\
& = (M+7) (K+1) d(x, A) -2K' \\
&\ge (M+1)(K+1) d(x, A)+ 4K.
\end{aligned}
\]
Now using Inequality \ref{eqn:KProjection} twice, we get\[ \begin{aligned}
d(y, A) &\ge \frac{1}{K+1} [d(y, \pi_{A}(y)) - 2K] \\
&\ge \frac{1}{K+1}[d(y, x) - d(x, \pi_{A}(x)) - \diam (\pi_{A}(x) \cup \pi_{A}(y)) - 2K] \\
&\ge \frac{1}{K+1} [d(x, y) - \diam (\pi_{A}(x) \cup \pi_{A}(y)) - (K+1) d(x, A) - 4K]\\
&\ge M d(x, A).\qedhere
\end{aligned}
\]
\end{proof}

\begin{lem}\label{lem:weakContractingConvex}
For each $K>1$ there exists $K'>0$ satisfying the following.

Let $A$ be a $K$-weakly contracting set, let $x, y \in X$, let $p$ be a point on $[x, y]$ and define \[
\begin{aligned}
D_{1} := \diam \big( \pi_{A}(x)\cup \pi_{A}(p)\big), \quad D_{2} := \diam \big(\pi_{A}(y)\cup \pi_{A}(p)\big).
\end{aligned}
\]
Then we have \begin{equation}\label{eqn:weakContractingConvex}
d(p, A) \le K'e^{-D_{1}/K'} d(x, A) + K'e^{-D_{2}/K'} d(y, A) + K'.
\end{equation}
\end{lem}

\begin{proof}
Let $M:= 2K+4$, let $K_{1} := K'(K, M)$ be as in Lemma \ref{lem:quickLanding}, and let $K':= 9 MK_{1}$.

Suppose to the contrary that Inequality \ref{eqn:weakContractingConvex} does not hold. Our goal is to find a triple $x', y', z'$ on $[x, y]$, in order from left to right, such that \[\begin{aligned}
d(y', A) &>  \max\big( Md(x', A), Md(z', A), K'\big),\\
4K_{1} & \ge \diam(\pi_{A}\{x', y', z'\}).
\end{aligned}
\]
If we find such triple, then we have \[\begin{aligned}
d(y', A) &> \frac{K+1}{2K+4} \cdot M\cdot  d(x', A) + \frac{K+1}{2K+4} \cdot M \cdot d(z', A) + \frac{2}{2K+4} \cdot K' \\
&\ge (K+1) d(x', A) + (K+1) d(z', A) + 18K_{1} \\
&> (K+1) d(x', A) + (K+1) d(z', A) + (2K_{1} + 4K).
\end{aligned}
\]
This will then lead to the contradiction \[ \begin{aligned}
d(x', z') &\le d\big(x', \pi_{A}(x')\big) + \diam \big(\pi_{A}(x')\cup \pi_{A}(z')\big) + d\big(\pi_{A}(z'), z'\big) \\
&\le \big( (K+1) d(x', A) + 2K \big) + 4K_{1} + \big( (K+1) d(z', A) + 2K \big) \\
&< 2 d(y', A) - (K+1) d(x', A) - (K+1) d(z', A) - 4K \\
&\le \left[ d(y', A) - d\big(x', \pi_{A}(x')\big) \right] +  \left[ d(y', A) - d\big(z', \pi_{A}(z')\big) \right] \\
&\le d(x', y') + d(y', z').
\end{aligned}
\]

Let $\eta : [0, L] \rightarrow X$ be the geodesic connecting $p$ to $x$ and let $t_{0} = 0$. Given $t_{i-1} \in [0, L)$, we pick $t_{i} \in [t_{i-1}, L]$ such that \begin{equation}\label{eqn:choosingRule}
\diam \big(\pi_{A}(\eta(t_{i-1}))\cup  \pi_{A}(\eta(t_{i}))\big) \le 2K_{1}, \quad d(\eta(t_{i}), A) \ge M d(\eta(t_{i-1}), A).
\end{equation} If such $t_{N}$ does not exist at step $N$, we let $t_{N} = L$ and stop. 

Recall that we are assuming \[
d(\eta(t_{0}), A) \ge K'e^{-D_{1}/K'} d(x, A) + K'e^{-D_{2}/K'} d(y, A) + K' \ge K'.
\]
Hence, $d(\eta(t_{i}), A) \ge M^{i} K' \ge K'$ for $i = 0, \ldots, N-1$. ($\ast$) Since $\eta$ is bounded, the process must stop at some $N$. We always have $t_{N} = L$ and $\eta(t_{N}) = x$. We discuss possible scenarios: 
\begin{enumerate}
\item $d\big(\pi_{A}(\eta(t_{N-1}))\cup \pi_{A}(\eta(t_{N}))\big) > 2K_{1}$. Recall Lemma \ref{lem:quickLanding}: there exists $\tau \in [t_{N-1}, t_{N}]$ such that $\diam \big(\pi_{A}(\eta(t_{N-1})\cup \eta(\tau)\big) \le 2K_{1}$ and either $d(\eta(\tau), A) \ge Md(\eta(t_{N-1}), A)$ or $d(\eta(\tau), A) \le \frac{1}{M} d(\eta(t_{N-1}), A)$. Since the first possibility is excluded, we conclude that $d(\eta(\tau), A) \le \frac{1}{M}d(\eta(t_{N-1}), A)$. There are two subcases. \begin{enumerate}[label=(\alph*)]
\item $N \ge 2$: in this case, $d(\eta(t_{N-1}), A) \ge M d(\eta(t_{N-2}), A)$ holds by our choice in Display \ref{eqn:choosingRule}. By ($\ast$), we also know that $d(\eta(t_{N-1}), A) \ge K'$. Lastly, $\pi_{A}\{\eta(t_{N-2}), \eta(t_{N-1}), \eta(\tau)\}$ has diameter at most $4K_{1}$. Hence, we can take $x' = \eta(\tau)$, $y' = \eta(t_{N-1})$ and $z' = \eta(t_{N-2})$.
\item $N = 1$: in this case, we have  $d(\eta(\tau), A) \le \frac{1}{M} d(p, A)$. We first pick $x' = \eta(\tau)$ and will pick $y'$ and $z'$ later.
\end{enumerate}
\item $\diam \big(\pi_{A}(\eta(t_{N-1})), \pi_{A}(\eta(t_{N})\big) \le 2K_{1}$. Then we have \[
D_{1} = \diam \big(\pi_{A}(\eta(0)), \pi_{A}(\eta(t_{N})\big) \le \sum_{i=1}^{N} \diam \big(\pi_{A}(\eta(t_{i-1})), \pi_{A}(\eta(t_{i}))\big) \le 2K_{1} N.
\]
Since $K'\ge 4K_{1}$, $K' \ge e^{2}$ and $e<4< 2K+4=M$, we deduce \[
d(x, A) \le \frac{1}{K'} e^{D_{1} / K'} d(p, A) \le \frac{1}{K'} e^{N} d(p, A) \le (2K+4)^{N-2} d(\eta(t_{0}) , A) \le \frac{1}{M} d(\eta(t_{N-1}), A).
\]
Given this, when $N\ge 2$, we can pick $x' = x=\eta(t_{N})$, $y'= \eta(t_{N-1})$ and $z' = \eta(t_{N-2})$ and deduce a similar contradiction. When $N = 1$, we set $x' = x$. 
\end{enumerate}
So far, we have obtained either the desired triple $(x', y', z')$, or a point $x' \in [x, p]$ such that \[
\diam\big(\pi_{A}(x')\cup  \pi_{A}(p)\big) \le 2K', \quad d(x', A) \le \frac{1}{M} d(p, A).
\] A similar discussion on $[p, y]$ also gives either the desired triple, or a point $z' \in [p, y]$ such that $d(\pi_{A}(z'), \pi_{A}(p)) \le 2K'$ and $d(z', A) \le \frac{1}{M}d(p, A)$. If we fall into the latter cases in both discussions, we let $y' = p$ and deduce the contradiction.
\end{proof}

\subsection{Random walks} \label{subsection:RW}

Let $\mu$ be a probability measure on a discrete group $G$ acting on a metric space $(X, d)$. We denote by $\check{\mu}$ the \emph{reflected version of $\mu$}, which by definition satisfies $\check{\mu}(g) := \mu(g^{-1})$. The \emph{random walk} generated by $\mu$ is the Markov chain on $G$ with the transition probability $p(g, h) := \mu(g^{-1} h)$.

Consider the \emph{step space} $(G^{\Z}, \mu^{\Z})$, the product space of $G$ equipped with the product measure of $\mu$. Each element $(g_{n})_{n \in \Z}$ of the step space is called a \emph{step path}, and there is a corresponding \emph{(bi-infinite) sample path} $(Z_{n})_{n \in \Z}$ under the correspondence \[
Z_{n} = \left\{ \begin{array}{cc} g_{1} \cdots g_{n} & n > 0 \\ id & n=0 \\ g_{0}^{-1} \cdots g_{n+1}^{-1} & n < 0. \end{array}\right.
\]
We also introduce the notation $\check{g}_{n} = g_{-n+1}^{-1}$ and $\check{Z}_{n} = Z_{-n}$. Note that we have an isomorphism $(G^{\Z}, \mu^{\Z}) \rightarrow (G^{\Z_{>0}}, \mu^{\Z_{>0}}) \times (G^{\Z_{>0}}, \check{\mu}^{\Z_{>0}})$ by $(g_{n})_{n \in \Z} \mapsto ((g_{n})_{n >0} , (\check{g}_{n})_{n > 0})$. In view of this, we sometimes write the bi-infinite sample path as $((Z_{n})_{n\ge0}, (\check{Z}_{n})_{n\ge0})$, where the distribution of $(Z_{n})_{n}$ and $(\check{Z}_{n})_{n}$ are independent.

In certain circumstances, it is beneficial to consider a probability space $(\Omega, \Prob)$ where the step distributions for the random walk is defined, together with some other RVs. For this purpose, we say that $(\Omega, \Prob)$ is a \emph{probability space for $\mu$} if there is a measure-preserving map from $(\Omega, \Prob)$ to $(G^{\Z_{>0}}, \mu^{\Z_{>0}})$, or equivalently, if independent step RVs $\{g_{n}(\w)\}_{n>0}$ are defined and distributed according to $\mu$.  We  similarly define a probability space $(\check{\Omega}, \Prob)$ for $\check{\mu}$, together with RVs $\{g_{n}(\check{\w})\}_{n >0}$. Then the product space $(\Omega \times \check{\Omega}, \Prob)$ models the (bi-infinite) random walk generated by $\mu$. We often omit $\w$ while writing e.g. $g_{n} = g_{n}(\w)$ and $Z_{n} = Z_{n}(\w)$. To make a distinction, we mark RVs on $\check{\Omega}$ with the `check' sign, e.g., $\check{g}_{n} := g_{n}(\check{\w})$, $\check{Z}_{n} := Z_{n}(\check{\w})$.

We define the \emph{support} of $\mu$, denoted by $\supp \mu$, as the set of elements in $G$ that are assigned nonzero values of $\mu$. 
\begin{comment}
$\langle \supp \mu \rangle$ and $\llangle \supp \mu \rrangle$ denote the subgroup and the subsemigroup generated by the support of $\mu$, respectively. In other words, we define \[
\begin{aligned}
\langle \supp \mu \rangle &:= \{g_{1} \cdots g_{n} : n \in \Z_{\ge 0}, g_{i} \in (\supp \mu) \cup (\supp \mu)^{-1} \}, \\
\llangle \supp \mu \rrangle &:= \{g_{1} \cdots g_{n} : n \in \Z_{\ge 0}, g_{i} \in \supp \mu \}.
\end{aligned}
\]
\end{comment}
We denote by $\mu^{N}$ the product measure of $N$ copies of $\mu$, and by $\mu^{\ast N}$ the $N$-th convolution measure of $\mu$. We say that $\mu$ is \emph{non-elementary} if the subsemigroup generated by the support of $\mu$ contains two independent strongly contracting isometries $g, h$ of $X$. By taking suitable powers, we may assume that $g$ and $h$ belong to the same $\supp \mu^{\ast N}$ for some $N > 0$. 

When a constant $M_{0}$ (to be fixed later) is understood, we use the notation \[
\axes_{i}(\w):= \left(Z_{i - M_{0}}(\w) o, \,\,Z_{i-M_{0} + 1}(\w) o,\,\, \ldots,\,\, Z_{i}(\w) o\right).
\]
Similarly, we denote $(Z_{i-M_{0}}(\check{\w})o, \ldots, Z_{i}(\check{\w}) o)$ by $\axes_{i}(\check{\w})$.

\part{Random walks with strongly contracting isometries} \label{part:strong}

In Part \ref{part:strong}, we develop a theory of random walks that involve strongly contracting isometries. The following convention is employed throughout Part \ref{part:strong}.

\begin{conv}\label{conv:strong}
We assume that: \begin{itemize}
\item $(X, d)$ is a geodesic metric space;
\item $G$ is a countable group of isometries of $X$, and 
\item $G$ contains two independent strongly contracting isometries.
\end{itemize}
We also fix a basepoint $o \in X$.
\end{conv}

We emphasize that no further requirements (properness, WPD-ness, etc.) are imposed on $X$ or $G$. Convention \ref{conv:strong} includes the following situations:
\begin{enumerate}
\item $(X, d)$ is a geodesic Gromov hyperbolic space and $G$ contains independent loxodromics, e.g. 
\begin{enumerate}
\item $(X, d)$ is the curve complex of a finite-type hyperbolic surface and $G$ is the mapping class group, or
\item $(X, d)$ is the complex of free factors of the free group of rank $N \ge 3$ and $G$ is the outer automorphism group $\Out(F_{N})$;
\end{enumerate}
\item $X$ is Teichm{\"u}ller space of finite type, $G$ is the corresponding mapping class group, and $d$ is either the Teichm{\"u}ller metric $d_{\T}$ \cite{minsky1996quasi-projections} or the Weil-Petersson metric $d_{WP}$ \cite{bestvina2002bounded};
\item $(X, d)$ is the Cayley graph of a braid group modulo its center $B_{n}/Z(B_{n})$ with respect to its Garside generating set, and $G$ is the braid group $B_{n}$ \cite{calvez2021morse};
\item $(X, d)$ is the Cayley graph of a group $G$ with nontrivial Floyd boundary \cite{karlsson2003free}, \cite{gerasimov2013quasi-isometric};
\item $(X, d)$ is the Cayley graph of a $Gr'(1/6)$-labeled graphical small cancellation group $G$ \cite{arzhantseva2019negative};
\item $(X, d)$ is a (not necessarily proper nor finite-dimensional) CAT(0) space and $G$ contains  independent rank-1 isometries; e.g., $G$ is an irreducible right-angled Artin group and $(X, d)$ is the universal cover of its Salvetti complex.
\end{enumerate}

\section{Alignment I: strongly contracting axes} \label{section:strongAlign}

In this section, we will formulate and prove the following claim. Let $(\kappa_{i})_{i=1}^{n}$ be a sequence of long enough contracting axes. Suppose that each pair of consecutive axes is aligned: $\kappa_{i}$ ($\kappa_{i+1}$, resp.) projects onto $\kappa_{i+1}$ ($\kappa_{i}$, resp.) near the beginning point of $\kappa_{i+1}$ (the ending point of $\kappa_{i}$, resp.). Then the axes are globally aligned: $\kappa_{i}$ projects onto $\kappa_{j}$ near the beginning point (ending point, resp.) of $\kappa_{j}$ when $i < j$ ($i>j$, resp.).

\subsection{Contracting geodesics} \label{subsection:contractingGeod}

The goal of this subsection is to establish Corollary \ref{cor:BGIPIntersection}. We begin by recalling a lemma that appeared as \cite[Lemma 2.14]{arzhantseva2015growth}, \cite[Lemma 2.4]{sisto2018contracting} and \cite[Lemma 2.4(4)]{yang2019statistically}. For a version with explicit constant, see \cite[Lemma 2.2]{chawla2023genericity}.

\begin{lem}\label{lem:BGIPHausdorff}
Let $A$ be a $K$-strongly contracting set and let $\eta: I \rightarrow X$ be a geodesic such that $\diam(\pi_{A}(\eta)) > K$. Then there exist $t < t'$ in $I$ such that $\pi_{A}(\eta)$ and $\eta|_{[t, t']}$ are $4K$-coarsely equivalent, and moreover, such that \[
\diam\big( \pi_{A}\big(\eta|_{(-\infty, t]}\big) \cup \eta(t)\big) <2K \,\,\textrm{and}\,\, \quad \diam\big( \pi_{A}\big(\eta|_{[t', +\infty)}\big) \cup \eta(t')\big) < 2K.
\]
\end{lem}

\begin{lem}\label{lem:QuasiHausdorff}
For each $K>1$ there exists $K'=K'(K) >K$ that satisfies the following.

Let $\eta: J \rightarrow X$ be a $K$-quasigeodesic whose endpoints are $x$ and $y$, let $A$ be a subset of $\eta$ such that $d(x, A) < K$, $d(y, A) < K$, and let $\gamma: J' \rightarrow X$ be a geodesic that is $K$-coarsely equivalent to $A$. Then $\eta$ and $\gamma$ are also $K'$-coarsely equivalent, and moreover, there exists a $K'$-quasi-isometry $\varphi : J \rightarrow J'$ such that $d(\eta(t), (\gamma \circ \varphi)(t)) < K'$ for each $t \in J$.
\end{lem}

\begin{proof}
Without loss of generality, let $J = [a, b]$, $J' = [c, d]$ and $\eta(a) = x$, $\eta(b) = y$. For each $s \in J'$, we can pick $t_{s} \in J$ such that $d(\gamma(s), \eta(t_{s})) < K$ as $\gamma$ is coarsely contained in $A$. Note that \[
|t_{s_{1}}- t_{s_{2}}| \le Kd \big( \eta(t_{s_{1}}),  \eta(t_{s_{2}}) \big) + K^{2} \le K d\big(\gamma(s_{1}), \gamma(s_{2}) \big) + 3K^{2} = K|s_{1} - s_{2}| + 3K^{2}.
\]
Similarly, $|t_{s_{1}} - t_{s_{2}}| \ge \frac{1}{K} |s_{1} - s_{2}| - 1 - 2K$ holds. Hence, $s \mapsto t_{s}$ is a $3K^{2}$-quasi-isometric embedding. 

It remains to show that $\{t_{s} : s \in J'\}$ is coarsely equivalent to $J$. Note that $A$ is $3K$-coarsely connected, as it is $K$-coarsely contained in an $1$-connected set $\gamma$. It follows that $\eta^{-1}(A)$ is $4K^{2}$-coarsely connected subset of $[a, b]$. Moreover, since $x$ and $y$ are $K$-close to $A$, we have $d(a, \eta^{-1}(A)), d(b, \eta^{-1}(A)) < 2K^{2}$. Combined together, $[a, b]$ is $4K^{2}$-coarsely contained in $\eta^{-1}(A)$.

Next, for each $p \in A$ there exists $s \in J'$ such that $d(\gamma(s), p) < K$. This implies $d(\eta(t_{s}), p) < 2K$ and $\diam(t_{s}, \eta^{-1}(p)) < 3K^{2}$. Hence, $\eta^{-1}(A)$ is $3K^{2}$-coarsely contained in $\{t_{s} : s \in J'\}$.
\end{proof}

$K$-quasi-isometries between intervals are $K'$-coarsely equivalent to a monotone map for some $K'=K'(K)$. (for an explicit  $K'$, see the proof of \cite[Theorem 1.2]{sankaran2006on-homeomorphisms}). Hence, we have:

\begin{cor}\label{cor:QuasiHausdorff}
For each $K>1$ there exists $K'=K'(K) >K$ that satisfies the following.

Let $\eta: J \rightarrow X$ be a $K$-quasigeodesic connecting $x$ to $y$, let $A$ be a subset of $\eta$ such that $d(x, A) < K$ and $d(y, A) < K$, and let $\gamma: J' \rightarrow X$ be a geodesic that is $K$-coarsely equivalent to $A$. Then $\eta$ and $\gamma$ are $K'$-fellow traveling.
\end{cor}

Combining Lemma \ref{lem:BGIPHausdorff} and Lemma \ref{lem:QuasiHausdorff}, we observe an instance of the Morseness of contracting axes (\cite[Theorem 1.3]{arzhantseva2017characterizations}, \cite[Lemma 2.8.(2)]{sisto2018contracting}, \cite[Lemma 2.2]{yang2014growth}).

\begin{cor}\label{cor:BGIPSelfNbd}
For each $K>1$ there exists a constant $K' >K$ that satisfies the following. Let $\eta : J \rightarrow X$ be a $K$-contracting axis and $\gamma: J' \rightarrow X$ be a geodesic that share the endpoints. Then $\eta$ and $\gamma$ are $K'$-fellow traveling.
\end{cor}

\begin{cor}\label{cor:BGIPIntersection}
For each $K>1$ there exists a constant $K' = K'(K)$ that satisfies the following.

Let $\kappa : I \rightarrow X$ and $\eta : J \rightarrow X$ be $K$-contracting axes. Suppose that $\diam(\pi_{\kappa}(\eta)) > K'$. Then there exist $t < t'$ in $I$ and $s < s'$ in $J$ such that the following sets are all $K'$-coarsely equivalent: \[
\kappa|_{[t, t']}, \, \,\eta|_{[s, s']},\,\,  \,\pi_{\kappa}(\eta),\, \,\pi_{\eta}(\kappa).
\]
Moreover, we have\[
\diam\big( \pi_{\kappa}\left( \eta|_{(-\infty, s]}\right) \cup \eta(s)\big) < K', \quad \diam\big( \pi_{\kappa}\left( \eta|_{[s', +\infty)}\right) \cup \eta(s')\big) < K'.
\]
\end{cor}

\begin{proof}
For simplicity, we focus on the case where $\kappa$, $\eta$ have endpoints.

Let $\gamma : J' \rightarrow X$ be a geodesic that connects the endpoints of $\eta$. Then $\gamma$ and $\eta$ are coarsely equivalent by Corollary \ref{cor:BGIPSelfNbd}. Lemma \ref{lem:projBdd} tells us that $\pi_{\kappa}(\gamma)$ is coarsely equivalent to $\pi_{\kappa}(\eta)$ and hence large. By Lemma \ref{lem:BGIPHausdorff}, there exist $u < u'$ in $J'$ such that $\pi_{\kappa}(\gamma)$ and $\gamma|_{[u, u']}$ are coarsely equivalent and such that $\gamma|_{(-\infty, u]}$ and $\gamma|_{[u', +\infty)}$ project onto $\kappa$ near $\gamma(u)$ and $\gamma(u')$, respectively.

Note again that $\eta$ and $\gamma$ are fellow traveling by Corollary \ref{cor:BGIPSelfNbd} and $\pi_{\kappa}$ is coarsely Lipschitz. This enables us to replace $\gamma$ with $\eta$: there exist $s < s'$ in $J$ such that $\pi_{\kappa}(\eta)$ and $\eta|_{[s, s']}$ are coarsely equivalent and such that $\eta|_{(-\infty, s]}$ and $\eta|_{[s', +\infty)}$ project onto $\kappa$ near $\eta(s)$ and $\eta(s')$, respectively.

Since $\pi_{\kappa}(\eta) \subseteq \kappa$ and $\eta|_{[s, s']}$ are nearby, each point $\eta(t)$ in $\eta|_{[s, s']}$ is near a point $\kappa(s_{t})$ of $\kappa$. This $\kappa(s_{t})$ projects onto $\eta$ near $\eta(t)$. It follows that $\pi_{\eta}(\kappa)$ coarsely contains $\eta|_{[s, s']}$ and hence $\pi_{\kappa}(\eta)$.

This implies that $\pi_{\eta}(\kappa)$ is also large, and we have another round: there exist $t < t'$ in $I$ such that $\pi_{\eta}(\kappa)$ and $\kappa|_{[t, t']}$ are coarsely equivalent. Moreover, $\pi_{\kappa}(\eta)$ coarsely contains $\pi_{\eta}(\kappa)$. Hence, the two projections are coarsely equivalent, and \[
\pi_{\eta}(\kappa),\,\, \kappa|_{[t, t']}, \,\,\pi_{\kappa}(\eta),\,\,\eta|_{[s, s']}
\]
are all coarsely equivalent.
\end{proof}

We now digress to the proof of Lemma \ref{lem:indepEquiv}.

\begin{proof}[Proof of Lemma \ref{lem:indepEquiv}]
Let $\eta$ and $\kappa$ denote the axes of $g$ and $h$, i.e., $\eta: i \mapsto g^{i} o$ and $\kappa : j \mapsto h^{j} o$. Let $\eta$ and $\kappa$ be $K$-contracting axes for some $K>0$.

Suppose that $\pi_{\kappa}(\eta)$ has finite diameter, i.e., there exists $M$ such that \[
\pi_{\kappa}(\eta) \subseteq \big\{\kappa(-M), \kappa(-M+1), \ldots, \kappa(M-1), \kappa(M)\big\}.
\]Then for each $i \in \Z$ and $|j|> M+2K^{2}$, the diameter of $\pi_{\kappa}(\eta(i) ) \cup \kappa(j)$ is greater than $K$ and $[\eta(i), \kappa(j)]$ is $2K$-close to $\pi_{\kappa}(\eta(i))$. This forces that \[
d(\eta(i), \kappa(j)) \ge \inf_{|t| \le M} d(\kappa(j), \kappa(t)) - 2K \ge \frac{1}{K} |j| - \frac{M}{K} - 3K.
\]
Similarly, if $\pi_{\eta}(\kappa)$ has finite diameter, then there exists $M'$ such that $d(\eta(i), \kappa(j)) \ge \frac{1}{K} |i| - M'$ holds for all $j$ and $|i| > M'$. Hence $d(g^{i} o, h^{j} o)$ is a proper function, and $g$ and $h$ are independent.

Now suppose that $\pi_{\kappa}(\eta)$ has infinite diameter. By Corollary \ref{cor:BGIPIntersection}, $\eta$ and $\kappa$ have subpaths $\eta'$ and $\kappa'$, respectively, that are coarsely equivalent to $\pi_{\kappa}(\eta)$, of infinite diameter. This means that $\eta$ and $\kappa$ are not independent.
\end{proof}

\subsection{Alignment} \label{subsection:alignment}

Let us now define the notion of alignment.

\begin{definition}\label{dfn:alignment}
For $i=1, \ldots, n$, let $\kappa_{i}$ be a path on $X$ whose beginning and ending points are $x_{i}$ and $y_{i}$, respectively. We say that $(\kappa_{1}, \ldots, \kappa_{n})$ is $C$-aligned if \[
\diam_{X}\big(y_{i} \cup \pi_{\kappa_{i}}(\kappa_{i+1})\big) < C, \quad \diam_{X}\big(x_{i+1} \cup \pi_{\kappa_{i+1}} (\kappa_{i}) \big) < C
\]
hold for $i = 1, \ldots, n-1$.
\end{definition}

\begin{figure}
\begin{tikzpicture}[scale=1.1]
\draw[very thick] (-3, 0) -- (-0.8, 0);
\draw[very thick] (3, 0) -- (0.8, 0);
\draw[dashed, thick] (-3, 0) arc (160:24.5:2.27);
\draw[dashed, thick] (-0.8, 0) arc (160:27:1.03);
\draw[shift={(1.175, 0.06)}, rotate=19, dashed, thick] (-0.153, 0.153) -- (0, 0) -- (0.153, 0.153);

\begin{scope}[rotate=180]
\draw[dashed, thick] (-3, 0) arc (160:23:2.2);
\draw[dashed, thick] (-0.8, 0) arc (160:29:0.95);
\draw[shift={(1.025, 0.06)}, rotate=19, dashed, thick] (-0.153, 0.153) -- (0, 0) -- (0.153, 0.153);
\end{scope}

\draw (-3.19, 0) node {$x_{1}$};
\draw (-0.58, 0) node {$y_{1}$};
\draw (0.58, 0) node {$x_{2}$};
\draw (3.19, 0) node {$y_{2}$};

\draw (-1.9, -0.2) node {$\kappa_{1}$};
\draw (1.9, -0.2) node {$\kappa_{2}$};

\end{tikzpicture}
\caption{Schematics for an aligned sequence of paths.}
\label{fig:alignment}
\end{figure}
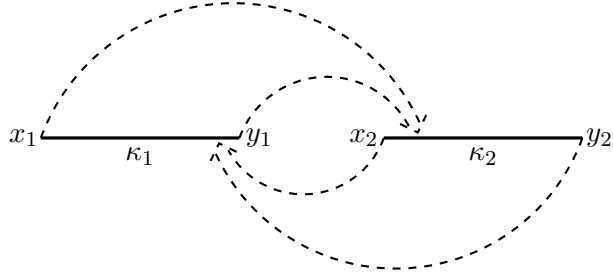

Note that if $(\kappa_{i}, \ldots, \kappa_{j})$ and $(\kappa_{j}, \ldots, \kappa_{k})$ are $C$-aligned, then  $(\kappa_{i}, \ldots, \kappa_{j}, \ldots, \kappa_{k})$ is also $C$-aligned. We allow degenerate paths, e.g., the case where $\kappa_{1}$ or $\kappa_{n}$ is a point.

Combining Lemma \ref{lem:BGIPHausdorff} and Corollary \ref{cor:QuasiHausdorff}, we obtain the following consequence of alignment. 

\begin{cor}\label{cor:singleAlign}
For each $C, K>1$, there exists $K' = K'(K, C) > \max(K, C)$ such that the following holds. 

Let $x, y \in X$ and let $\kappa$ be a $K$-contracting axis such that $\diam(\kappa) > K+2C$ and such that $(x, \kappa, y)$ is $C$-aligned. Then $[x, y]$ contains a subsegment $\eta$ that is $4K$-coarsely contained in $\kappa$ and is $K'$-fellow traveling with $\kappa$.
\end{cor}

Our first lemma states that the alignment of two strongly contracting axes is governed by the projections of their endpoints to the other axis.

\begin{lem}\label{lem:1segment}
For each $C, K>1$, there exists $D= D(K, C)>\max(K, C)$ such that the following holds.

Let $\kappa, \eta$ be $K$-contracting axes. If $\big(\kappa, (\textrm{beginning point of $\eta$})\big)$ and $\big((\textrm{beginning point of $\kappa$}), \eta\big)$ are each $C$-aligned, then $(\kappa, \eta)$ is $D$-aligned.
\end{lem}

\begin{proof}
For simplicity, let us assume that the domains of $\kappa$ and $\eta$ are closed intervals, say, $I = [t_{0}, t_{1}]$ and $J = [s_{0}, s_{1}]$, respectively.

It suffices to show that $\pi_{\kappa}(\eta)$ and $\pi_{\eta}(\kappa)$ are both small. Suppose not. Then Corollary \ref{cor:BGIPIntersection} provides $t < t'$ in $I$ and $s < s'$ in $J$ such that \[
\pi_{\kappa}(\eta),\,\, \pi_{\eta}(\kappa), \,\, \kappa|_{[t, t']},\,\, \eta|_{[s. s']}
\]
are all coarsely equivalent and large. Moreover, $\pi_{\eta}(\kappa(t_{0}))$ is near $\kappa(t)$ and $\pi_{\eta}(\kappa(t_{1}))$ is near $\kappa(t')$. Similarly, $\pi_{\kappa}(\eta(s_{0}))$ is near $\eta(s)$ and $\pi_{\kappa}(\eta(s_{1}))$ is near $\eta(s')$.

Since $\pi_{\kappa}(\eta)$ and $\pi_{\eta}(\kappa)$ are large, both $t' - t$ and $s' - s$ are large. Since $[\kappa(t), \kappa(t')]$ and $[\eta(s), \eta(s')]$ are coarsely equivalent, one of the following is true:\begin{itemize}
\item $\kappa(t)$ is near $\eta(s)$ and $\kappa(t')$ is near $\eta(s')$; or,
\item $\kappa(t)$ is near $\eta(s')$ and $\kappa(t')$ is near $\eta(s)$.
\end{itemize}
This leads to the following contradictions: \begin{itemize}
\item If $\kappa(t)$ is near $\eta(s)$, then $\eta(s_{0})$ projects onto $\kappa$ near $\kappa(t)$. Since $t_{1} - t \ge t' - t$ is large, this projection cannot be near $\kappa(t_{1}) = y$.
\item If $\kappa(t)$ is near $\eta(s')$, then $\kappa(t_{0})$ projects onto $\eta$ near $\eta(s')$. Since $s' - s_{0} \ge s' - s$ is large, this projection cannot be near $\eta(s_{0}) = x'$.
\end{itemize}
Hence, $\pi_{\kappa}(\eta)$ and $\pi_{\eta}(\kappa)$ cannot be large and the conclusion follows.
\end{proof}

The following lemma was inspired by Behrstock's inequality for subsurface projections and curve complexes \cite[Theorem 4.3]{behrstock2006asymptotic}.

\begin{lem}[{\cite[Lemma 2.5]{sisto2018contracting}}]\label{lem:1segmentVar}
For each $D, K>1$, there exists $E = E(K, D)>\max( K, D)$ that satisfies the following.

Let $\kappa$, $\eta$ be $K$-contracting axes in $X$. Suppose that $(\kappa, \eta)$ is $D$-aligned. Then for any $p \in X$, either $(p, \eta)$ is $E$-aligned or $(\kappa, p)$ is $E$-aligned.
\end{lem}

We are now ready to prove the main result of this section.

\begin{prop}\label{prop:BGIPConcat}
For each $D, K>1$, there exist $E = E(K, D)>\max(K, D)$ and $L = L(K,D)>\max(K, D)$ that satisfy the following.

Let $x, y \in X$ and let $\kappa_{1}, \ldots, \kappa_{n}$ be $K$-contracting axes  whose domains are longer than $L$. Suppose that $(x, \kappa_{1}, \ldots, \kappa_{n}, y)$ is $D$-aligned. Then $(x, \kappa_{i}, y)$ is $E$-aligned for each $i$. 
\end{prop}

\begin{proof}
Let $E = E(K, D)$ be as in Lemma \ref{lem:1segmentVar} and let $L = 3KE + K^{2}$. Our claim is that $(x, \kappa_{i})$ and $(\kappa_{i}, y)$ are $E$-aligned for each $i$. By symmetry, it suffices to prove the alignment of $(x, \kappa_{i})$.

Let $\kappa$ be a $K$-contracting axis whose domain is longer than $L$. Then the endpoints of $\kappa$ are at least $3E$-apart. Consequently, no point  $p$ in $X$ satisfy the following at the same time: \[
\textrm{$(p, \kappa)$ is $E$-aligned}, \quad \textrm{$(\kappa, p)$ is $E$-aligned}.
\]
From this observation, we inductively deduce \[
\textrm{$(x, \kappa_{i})$ is $E$-aligned} \Rightarrow \textrm{$(\kappa_{i}, x)$ is not $E$-aligned}  \Rightarrow \textrm{$(x, \kappa_{i+1})$ is $E$-aligned}.
\]
for $i=1, \ldots, n$, where the latter implication follows from Lemma \ref{lem:1segmentVar}.\end{proof}

The above proposition can be strengthened as follows. First, we record an immediate consequence of the definition of fellow-traveling.

\begin{lem}\label{lem:fellowTravelAlign}
Let $E>0$ and $x, y \in X$. Let $\kappa$ be a path that $E$-fellow travels with a subsegment of $[x, y]$. Then $(x, \kappa, y)$ is $4E$-aligned.
\end{lem}

\begin{prop}\label{prop:BGIPWitness}
For each $D, K>1$, there exist $E = E(K, D)>\max(K, D)$ and $L = L(K,D)>\max(K, D)$ that satisfy the following. 

Let $x, y \in X$ and let $\kappa_{1}, \ldots, \kappa_{n}$ be $K$-contracting axes whose domains are longer than $L$ and such that $(x, \kappa_{1}, \ldots, \kappa_{n}, y)$ is $D$-aligned. Then the geodesic $[x, y]$ has subsegments $\eta_{1}, \ldots, \eta_{n}$, in order from left to right, that are longer than $100E$ and such that $\eta_{i}$ and $\kappa_{i}$ are $0.1E$-fellow traveling for each $i$.  In particular, $(x, \kappa_{i}, y)$ are $E$-aligned for each $i$.
\end{prop}

%%there exist subsegments $[x_{i}', y_{i}']$ of $[x, y]$ such that: \begin{enumerate}\item (Ordering) $d(x, y_{i}') \le d(x, x_{i+1}')$ for $i=1, \ldots, N-1$;\item (Fellow traveling) $\kappa_{i}$ and $[x_{i}', y_{i}']$ are $E$-coarsely equivalent;\item (Orientation) $d(x_{i}, x_{i}') \le E$ and $d(y_{i}, y_{i}') \le E$.\end{enumerate}

\begin{proof}
Let $E_{1} = E(K, D)$ and $L_{1} = L(K, D)$ be as in Proposition \ref{prop:BGIPConcat}. Let $K_{1} = E_{1} + 8K$ and let $E = 10 K'(K, K_{1})$, where $K'(K, K_{1})$ is as in Corollary \ref{cor:singleAlign}. Let also $L = L_{1} + 101K(K+E) + 2K$.

We will inductively prove a variant of the given statement, namely: \begin{quote}
If $(x, \kappa_{1})$ is $K_{1}$-aligned and $(\kappa_{1}, \ldots, \kappa_{n}, y)$ is $D$-aligned, then the conclusion holds.
\end{quote}
First,  we know that $(\kappa_{1}, y)$ is $E_{1}$-aligned by Proposition \ref{prop:BGIPConcat}. Since $(x, \kappa_{1}, y)$ is $K_{1}$-aligned and $\kappa_{1}$ is long enough, Corollary \ref{cor:singleAlign} provides a subsegment $\eta_{1} = [x_{1}', y_{1}']$ of $[x, y]$ that is $4K$-coarsely contained in $\kappa_{1}$ and is $0.1E$-fellow traveling with $\kappa_{1}$. We then have \[
d(x_{1}', y_{1}') \ge \diam(\kappa_{1}) - 0.2E \ge \frac{L}{K} - K - 0.2E \ge 100E.
\]

If $n=1$, this finishes the proof. If not, note that $y_{1}'$ is $4K$-close to $\kappa_{1}$. Lemma \ref{lem:projBdd} implies that $(y_{1}', \kappa_{2})$ is $(D + 8K)$-aligned, and hence $K_{1}$-aligned. Now the induction hypothesis implies that $[y_{1}', y]$ has subsegments $\eta_{2}, \ldots, \eta_{n}$, in order from left to right, that are longer than $100E$ and such that $\eta_{i}$ and $\kappa_{i}$ $0.1E$-fellow travel for $i\ge 2$. Then $\eta_{1}, \ldots, \eta_{n}$ become the desired subsegments.
\begin{comment}
$\pi_{\kappa_{1}}([x, y])$ is larger than $K_{1}$. Hence, Lemma \ref{lem:BGIPHausdorff} provides a subsegment $\eta_{1} = [x_{1}', y_{1}']$ of $[x, y]$ that is $K_{1}$-coarsely equivalent to $\pi_{\kappa_{1}}([x, y])$, and hence $K_{2}$-coarsely equivalent to $\kappa_{1}$. Moreover, $\pi_{\kappa_{1}}(x)$, $x_{1}$ and $x_{1}'$ are all close to each other and we have $d(x_{1}, x_{1}') \le 3K_{2}$. Similarly we have $d(y_{1}, y_{1}') \le 3K_{2}$. 

If $n=1$, this finishes the proof. If not, note that $y_{1}$ is $K_{1}$-close to $\kappa_{1}$. Lemma \ref{lem:projBdd} tells us that $(y_{1}, \kappa_{2})$ is $(D + K_{1} + 20K)$-aligned, and hence $K_{2}$-aligned. Now the induction hypothesis implies that $[y_{1}, y]$ has subsegments $\eta_{2}, \ldots, \eta_{n}$, in order from left to right, such that $\eta_{i}$ and $\kappa_{i}$ $E$-fellow travel for each $i$. Then $\eta_{1}, \ldots, \eta_{N}$ become the desired subsegments.
\end{comment}
\end{proof}

Using Proposition \ref{prop:BGIPWitness}, we can recover the following results by Yang. 

\begin{lem}[{\cite[Lemma 4.4]{yang2014growth}, \cite[Proposition 2.9]{yang2019statistically}}]\label{lem:BGIPConcatContracting}
For each $D, M>0$ and $K>1$, there exist $E = E(K, D, M)>D$ and $L = L(K,D)>D$ that satisfies the following.

Let $\kappa_{1}, \ldots, \kappa_{n}$ be $K$-contracting axes whose domains are  longer than $L$. Suppose that $(\kappa_{1}, \ldots, \kappa_{n})$ is $D$-aligned and $d(\kappa_{i}, \kappa_{i+1}) < M$ for each $i$. Then the concatenation $\kappa_{1} \cup \ldots \cup \kappa_{n}$ of $\kappa_{1}, \ldots, \kappa_{n}$ is an $E$-contracting axis.
\end{lem}

\begin{comment}
\begin{lem}[{ \cite[Proposition 2.9]{yang2019statistically}}]\label{lem:BGIPConcatContracting}
Let $D>0$ and $K>1$, and let $E = E(K, D)$ and $L = L(K, D)$ be as in Proposition \ref{prop:BGIPWitness}. Let $(\kappa_{1}, \ldots, \kappa_{N})$ be a $D$-aligned sequence of $K$-contracting axes whose domains are longer than $L$, and suppose that $d(\kappa_{i}, \kappa_{i+1}) < D$ for each $i$. Then the concatenation $\kappa_{1} \cup \ldots \cup \kappa_{N}$ of $\kappa_{1}, \ldots, \kappa_{N}$ is a $K'$-contracting axis.
\end{lem}
\end{comment}

\begin{lem}[{\cite[Corollary 3.2]{yang2014growth}}]\label{lem:BGIPQuasi}
For each $D>0$ and $K>1$, there exist $E = E(K, D)>D$ and $L = L(K,D)>D$ that satisfies the following.

For each $i \in \Z$, let $\kappa_{i}$ be a $K$-contracting axis whose beginning and ending points are $x_{i}$ and $y_{i}$, respectively, and whose domain is longer than $L$. Suppose that $(\ldots, \kappa_{i}, \kappa_{i+1}, \ldots)$ is $D$-aligned. Then the concatenation of $(\ldots, [x_{i-1}, y_{i-1}], [y_{i-1}, x_{i}], [x_{i}, y_{i}], [y_{i}, x_{i+1}], \ldots)$ is an $E$-quasigeodesic.
\end{lem}

\subsection{Schottky sets} \label{subsection:Schottky}

Using the previous concatenation lemmata, we will construct arbitrarily many independent contracting isometries. Recall again the notation introduced in Subsection \ref{subsection:path}.

\begin{definition}[cf. {\cite[Definition 3.11]{gouezel2022exponential}}]\label{dfn:Schottky}
Let $K > 0$ and let $S \subseteq G^{n}$ be a set of sequences of isometries. We say that $S$ is \emph{$K$-Schottky} if: \begin{enumerate}
\item $\Gamma^{+}(s)$ and $\Gamma^{-}(s)$ are $K$-contracting axes for all $s \in S$;
\item for each $x \in X$ we have \[
\#\Big\{ s \in S : \textrm{$\big(x, \Gamma^{+}(s)\big)$ and $\big(x, \Gamma^{-}(s)\big)$ are $K$-aligned}\Big\} \ge \# S - 1;
\]
\item for each $s \in S$, $\big(\bar{\Gamma}^{-}(s), \Gamma^{+}(s)\big)$ is $K$-aligned.
\end{enumerate}
%In this case, we define the \emph{length} of the Schottky set $S$ by $n$, the length of the domains of its elements.
Once a Schottky set $S$ is understood, its element $s$ is called a \emph{Schottky sequence} and the translates of $\Gamma^{\pm}(s)$ are called \emph{Schottky axes}. We say that $S$ is \emph{large enough} if its cardinality is at least 400.

Let $\mu$ be a probability measure on $G$. If each element $s$ of $S$ is attained by the product measure of $\mu$, i.e., $S \subseteq (\supp \mu)^{n}$, then we say that \emph{$S$ is a Schottky set for $\mu$}.
\end{definition}

An intuitive example was given in the introduction. Consider $S_{M}:= \big\{ s_{1}s_{2} \cdots s_{M} : s_{i} \in \{a, b\} \big\}$ in $F_{2}= \langle a, b \rangle$. For any infinite ray on $F_{2}$, at most 1 element $s \in S_{m}$ heads into the direction: \[
\#\{ s \in S_{M} :  (\xi, s)_{id} \ge M \,\, \textrm{or}\,\, (\xi, s^{-1})_{id} \ge M\} \le 1
\]
for each infinite ray $\xi$. Moreover, $s$ and $s^{-1}$ diverge early for any $s \in S_{M}$: \[
(s^{-1}, s)_{id} < 1 \quad \textrm{for each $s \in S$}.
\]
These properties are also satisfied by  the set of $N$-th powers of elements of $S_{M}$ \[
S_{N, M}:= \{ s^{N} : s \in S_{M} \}.
\]

\begin{definition} \label{dfn:SchottkyLong}
Given a constant $K_{0}>0$, we define: \begin{itemize}
\item $D_{0}= D(K_{0}, K_{0})$ be as in Lemma \ref{lem:1segment}, 
\item $E_{0}= E(K_{0}, D_{0})$, $L_{0} = L(K_{0}, D_{0})$ be as in Proposition \ref{prop:BGIPWitness}.
\end{itemize}
A $K_{0}$-Schottky set $S$ whose elements have domains longer than $L_{0}$ is called a \emph{long enough $K_{0}$-Schottky set}. In other words, when $S \subseteq G^{n}$ is $K_{0}$-Schottky and $n > L_{0}$, $S$ is called a long enough $K_{0}$-Schottky set. In this case, note that the endpoints of $\Gamma^{+}(s)$ are $100E_{0}$-apart for each $s \in S$.
\end{definition}

This definition is motivated by the alignment lemmata. Note that the $D_{0}$-alignment of a sequence of Schottky axes $(\gamma_{1}, \ldots, \gamma_{N})$ is a \emph{local} condition, between consecutive pairs of axes. Proposition \ref{prop:BGIPWitness} then promotes this into the global alignment, i.e., the $E_{0}$-alignment of $(\gamma_{i}, \gamma_{j})$ for any $i < j$, given that the involved Schottky set is long enough. The following definition is designed to capture this local-to-global phenomenon.

\begin{definition}\label{dfn:semiAlign}
Let $S$ be a Schottky set, let $x, y\in X$ and let $\kappa_{1}, \ldots, \kappa_{N}$ be Schottky axes. We say that $(x, \kappa_{1}, \ldots, \kappa_{N}, y)$ is \emph{$C$-semi-aligned} if it is a subsequence of a $C$-aligned sequence of $x$, $y$ and Schottky axes, i.e., if there exist Schottky axes $\eta_{1}, \ldots, \eta_{N'}$ and $1 \le i(1) < \ldots < i(N) \le N'$ such that: \begin{enumerate}
\item $(x, \eta_{1}, \ldots, \eta_{N'}, y)$ is $C$-aligned,
\item $\kappa_{k} = \eta_{i(k)}$ for $k = 1, \ldots, N$.
\end{enumerate}
Here, we also say that $(x, \kappa_{1}, \ldots, \kappa_{N})$ and $(\kappa_{1}, \ldots, \kappa_{N}, y)$ are $C$-semi-aligned.
\end{definition}

\begin{lem}\label{lem:SchottkyAlign}
Let $S$ be a long-enough $K_{0}$-Schottky set. Let $x, y \in X$, and for each $i=1, \ldots, N$, let $\kappa_{i}$ be a Schottky axis whose beginning and ending points are $x_{i}$ and $y_{i}$, respectively.
\begin{enumerate}
\item If $(x_{1}, \kappa_{2})$ and $(\kappa_{1}, x_{2})$ are $K_{0}$-aligned, then $(\kappa_{1}, \kappa_{2})$ is $D_{0}$-aligned.
\item If $(x, \kappa_{1}, \ldots, \kappa_{N}, y)$ is $D_{0}$-semi-aligned, then $(\kappa_{i}, \kappa_{j})$ is $E_{0}$-aligned for each $i<j$. Moreover, $\kappa_{i}$ is $0.1E_{0}$-coarsely contained in $[x, y]$ and $(x, \kappa_{i}, y)$ is $E_{0}$-aligned for each $i$. We also have \[
\begin{aligned}
d(x, x_{1}) + \sum_{i=1}^{N} d(x_{i}, y_{i}) + \sum_{i=1}^{N-1} d(y_{i}, x_{i+1}) + d(y_{N}, y) &\le d(x, y) + E_{0} N, \\
d(x, x_{1}) + \sum_{i=1}^{N-1} d(y_{i}, x_{i+1}) + d(y_{N}, y)  &\le d(x, y) - 50E_{0} N.
\end{aligned}
\] 
\end{enumerate}
\end{lem}

\begin{proof}
(1) By Lemma \ref{lem:1segment}. (2) Proposition \ref{prop:BGIPWitness} explains the first two claims in Item (ii). More explicitly, $[x, y]$ contains subsegments $[x_{1}', y_{1}']$, $\ldots$, $[x_{N}', y_{N}']$, in order from left to right, such that: \begin{enumerate}[label=(\alph*)]
\item $[x_{i}', y_{i}']$ and $\kappa_{i}$ are $0.1E_{0}$-coarsely equivalent;
\item $d(x_{i}', x_{i}) < 0.1E_{0}$ and $d(y_{i}', y_{i}) < 0.1E_{0}$;
\item $d(x_{i}', y_{i}') > 100E_{0}$
\end{enumerate}
for each $i$. This implies that \[
\begin{aligned}
d(x, y) &= d(x, x_{1}') + \sum_{i=1}^{N} d(x_{i}', y_{i}') + \sum_{i=1}^{N-1} d(y_{i}', x_{i+1}') + d(y_{N}', y)  \\
&\ge d(x, x_{1}) + \sum_{i=1}^{N} d(x_{i}, y_{i}) + \sum_{i=1}^{N-1} d(y_{i}, x_{i+1}) + d(y_{N}, y) - 2 \sum_{i=1}^{N}  \big(d(x_{i}, x_{i}') + d(y_{i}, y_{i}') \big)\\
&\ge d(x, x_{1}) + \sum_{i=1}^{N} d(x_{i}, y_{i}) + \sum_{i=1}^{N-1} d(y_{i}, x_{i+1}) + d(y_{N}, y)  - E_{0} N \\
&\ge d(x, x_{1}) + \sum_{i=1}^{N-1} d(y_{i}, x_{i+1}) + d(y_{N}, y)  + 50 E_{0} N. \qedhere
\end{aligned}
\] 
\end{proof}

We now associate long enough and large Schottky sets with non-elementary measures.

\begin{prop}[cf. {\cite[Proposition 3.12]{gouezel2022exponential}}]\label{prop:Schottky}
Let $\mu$ be a non-elementary probability measure on $G$. Then for each $N>0$,  there exists $K = K(N) > 0$ such that for each $L>0$ there exists a $K$-Schottky set of cardinality $N$ in $(\supp \mu)^{n}$ for some $n > L$.
\end{prop}

\begin{proof}
Since $\mu$ is non-elementary, the semigroup generated by $\supp \mu$ contains independent strongly contracting isometries $a$ and $b$. By taking suitable powers, we may assume that $a = \prodSeq(\alpha)$ and $b = \prodSeq(\beta)$ for some $\alpha, \beta \in (\supp \mu)^{L_{0}}$ for some $L_{0}>0$. There exists $K_{0} > 0$ such that:
\begin{enumerate}[label=(\roman*)]
\item $\Gamma^{+}(\alpha)$, $\Gamma^{-}(\beta)$ are $K_{0}$-contracting axes, and
\item\label{cond:AlignSchottky} $\diam(o \cup \pi_{\gamma}(\eta)) < K_{0}$ for distinct axes $\gamma$, $\eta$ among \[
\Gamma^{+}(\alpha),\,\,\Gamma^{-}(\alpha),\,\,\Gamma^{+}(\beta), \,\,\Gamma^{-}(\beta).
\]
\end{enumerate}
The above statements still hold with the same $K_{0}$ when $\alpha$ and $\beta$ are replaced with their self-concatenations, thanks to Lemma \ref{lem:BGIPRestriction} and Lemma \ref{lem:indepEquiv}. Let: \begin{itemize}
\item $K_{1} = E(K_{0}, K_{0}) > K_{0}$ be as in Lemma \ref{lem:1segmentVar};
\item $K_{2} = E(K_{0}, K_{1}) > K_{1}$, $L_{2} = L(K_{0}, K_{1})$ be as in Proposition \ref{prop:BGIPWitness};
\item $K_{3} = E(K_{0}, K_{2}) > K_{2}$, $L_{3} = L(K_{0}, K_{2})$ be as in Proposition \ref{prop:BGIPWitness};
\item $L_{3}' = 3K_{0}K_{3}$;
\item $K_{4} = E(K_{0}, K_{0}, 0)$, $L_{4} = L(K_{0}, K_{0})$ be as in Lemma \ref{lem:BGIPConcatContracting}.
\end{itemize}
 By self-concatenating $\alpha$ and $\beta$ if necessary, we may assume that \[
L_{0}>L_{2}+L_{3} + L_{3}' + L_{4}.
\]
Since $\Gamma^{+}(\alpha)$ is a $K_{0}$-quasigeodesic whose domain $L_{0}$-long, the endpoints of $\Gamma^{+}(\alpha)$ are at least $(L_{0}/K_{1} - K_{1})$-apart. Since $L_{0}$ is greater than $3K_{0} K_{3} \ge 2K_{0} K_{1} + K_{0}^{2}$,  the endpoints of $\Gamma^{+}(\alpha)$ are $2K_{1}$-far. In particular, no set $A \subseteq X$ can be simultaneously contained in the $K_{1}$-neighborhoods of the two endpoints of $\Gamma^{+}(\alpha)$. Hence, the statements \[
\big(x, \Gamma^{\pm}(\alpha)\big)\,\, \textrm{is $K_{1}$-aligned}, \quad\big(\Gamma^{\pm}(\alpha), x\big)\,\, \textrm{is $K_{1}$-aligned}
\]
are mutually exclusive for any $x \in X$. Similarly, the statements \[
\big(x, \Gamma^{\pm}(\beta)\big)\,\, \textrm{is $K_{1}$-aligned}, \quad\big(\Gamma^{\pm}(\beta), x\big)\,\, \textrm{is $K_{1}$-aligned}
\]are mutually exclusive.

Let $S_{0}$ be the set of sequences of $NL_{0}$ isometries that are concatenations of $\alpha$'s and $\beta$'s, i.e., 
\[
S_{0}:= \left\{ (\phi_{1}, \ldots, \phi_{NL_{0}})  \in G^{NL_{0}} : (\phi_{L_{0}(i-1)+1}, \ldots, \phi_{L_{0}i}) \in \{\alpha, \beta\}\,\,\textrm{for} \,\, i =1, \ldots, N\right\}.
\]
Note that $\#S_{0} = 2^{N}$ is greater than $N$. We claim that for each $m>0$, the set \[
S_{0}^{(m)} := \Big\{ \textrm{$m$-self-concatenations of $s \in S_{0}$}\Big\} = \Big\{ \big(\underbrace{s, \ldots, s}_{\textrm{$m$ times}}\big) : s \in S_{0}\Big\}
\]
is a $((K_{0} L_{0} + K_{0})N + K_{2} + K_{4})$-Schottky set.

\emph{Step 1: Investigating $\Gamma^{m}(s)$}.

Pick $s =(\phi_{1}, \ldots, \phi_{NL_{0}})$ and $s' = (\phi_{1}', \ldots, \phi_{NL_{0}}')$ in $S_{0}$. Recall the notation \[
x_{nNL_{0} + i}(s) :=  (\phi_{1} \cdots \phi_{NL_{0}})^{n} \phi_{1} \cdots \phi_{i} o
\] for $n \in \Z$ and $i = 0, \ldots, NL_{0}-1$. We now define ``sub-axes''\[\begin{aligned}
\Gamma_{i}(s) := \left(x_{L_{0}(i-1)}(s), \ldots, x_{L_{0}i}(s)\right),\\
\Gamma_{-i}(s) := \left(x_{-L_{0}(i-1)}(s), \ldots, x_{-L_{0}i}(s)\right)
\end{aligned}
\] for each $i >0$. These are translates of $\Gamma^{\pm}(\alpha)$ and $\Gamma^{\pm}(\beta)$. Our initial choices of $K_{0}$ and $L_{0}$ guarantee that: \begin{itemize}
\item $\Gamma_{i}(s)$ is a $K_{0}$-contracting axis whose domain is longer than $L_{2}, L_{3}, L_{3}'$ and $L_{4}$ for each $i \in \Z$;
\item $(\Gamma_{i}(s), \Gamma_{i+1}(s))$ and $(\Gamma_{-i}(s'), \Gamma_{-(i+1)}(s'))$ are $K_{0}$-aligned for each $i > 0$. Moreover, $(\bar{\Gamma}_{-1}(s'), \Gamma_{1}(s))$ is $K_{0}$-aligned. 
\end{itemize}
Lemma \ref{lem:BGIPConcatContracting} tells us that $\cup_{i > 0} \Gamma_{i}(s)$ is a $K_{4}$-contracting axis. In particular, $\Gamma^{m}(s)$ is a $K_{4}$-contracting axis for each $m>0$. Similarly, $\Gamma^{-m}(s')$ is a $K_{4}$-contracting axis for each $m>0$.

Now note that the following sequence of sub-axes is $K_{0}$-aligned: \[
(\ldots, \bar{\Gamma}_{-2}(s'), \bar{\Gamma}_{-1}(s'), \Gamma_{1}(s), \Gamma_{2}(s), \ldots).
\]
Let $i > 0$ and let $p \in \Gamma_{-i} (s')$. Then Proposition \ref{prop:BGIPWitness} tells us that $d\big(p, \Gamma_{1}(s)\big) < d\big(p, \Gamma_{j}(s)\big)$ for each $j>1$ and that $\big(p, \Gamma_{1}(s)\big)$ is $K_{2}$-aligned. It follows that $\big(p, \cup_{i > 0} \Gamma_{i}(s)\big)$ is $K_{2}$-aligned. For this reason (and its symmetric counterpart),  $(\bar{\Gamma}^{-m}(s'), \Gamma^{m}(s))$ is $K_{2}$-aligned for each $m>0$.

Next, fix $x \in X$ and consider the condition\begin{equation}\label{eqn:SchottkyIneq1}
\big(x, \Gamma_{N}(s)\big)\,\,\textrm{is $K_{2}$-aligned}.
\end{equation}
If Condition \ref{eqn:SchottkyIneq1} holds, then for each $i>N$ \[
\big(x, \Gamma_{N}(s), \Gamma_{N+1}(s) \ldots, \Gamma_{i} (s)\big)
\]
is $K_{2}$-aligned and $d(x, \Gamma_{N}(s)) < d(x, \Gamma_{i}(s))$ holds by Proposition \ref{prop:BGIPWitness}. Hence, $\pi_{\cup_{i > 0} \Gamma_{i}(s)}(x)$ is contained in $\Gamma_{1}(s) \cup \cdots \cup \Gamma_{N}(s)$. Meanwhile, recall that for each $i$, $\Gamma_{i}(s)$ is a $K_{0}$-quasigeodesic whose domain is $L_{0}$-long. Hence, we have \[
\diam(\Gamma_{i}(s)) \le K_{0} \cdot (\textrm{length of the domain of $s$}) + K_{0} = K_{0} L_{0} + K_{0}.
\]
Combining these ingredients, we observe that  \[
\diam\left(\pi_{\Gamma^{m}(s)} (x) \cup o\right) \le \diam(\Gamma_{1}(s)) + \ldots + \diam(\Gamma_{N}(s)) \le  (K_{0}L_{0} + K_{0})N
\]
holds for every $m>0$. For a similar reason, the condition \begin{equation}\label{eqn:SchottkyIneq2}
\big(x, \Gamma_{-N}(s)\big) \,\,\textrm{is $K_{2}$-aligned}
\end{equation}
implies $\diam\left( \pi_{\Gamma^{m}(s)}(x) \cup o \right) \le (K_{0} L_{0} + K_{0})N$ for all $m <0$. In summary,

\begin{obs}\label{obs:Schottky1}
If $s \in S_{0}^{(m)}$ satisfies Condition \ref{eqn:SchottkyIneq1} and \ref{eqn:SchottkyIneq2}, then $\big(x, \Gamma^{m}(s)\big)$ is $(K_{0} L_{0} + K_{0})N$-aligned for all $m \in \Z$. 
\end{obs}

\emph{Step 2. Comparing two distinct axes.}

We now pick $m>0$ and consider an element of $S_{0}^{(m)}$ which violates these conditions.

\begin{obs}\label{obs:Schottky2}
If $s = (\phi_{1}, \ldots, \phi_{mNL_{0}}) \in S_{0}^{(m)}$ violates Condition \ref{eqn:SchottkyIneq1}, then all the other elements $s' = (\phi_{1}', \ldots, \phi_{mNL_{0}}') \in S_{0}^{(m)}$ satisfy Condition \ref{eqn:SchottkyIneq1} and Condition \ref{eqn:SchottkyIneq2}. 
\end{obs}

To show this, let $k \in \{1, \ldots, N\}$ be the first index such that $(\phi_{L_{0}(k-1)+1}, \ldots, \phi_{L_{0}k})$ and $(\phi_{L_{0}(k-1) + 1}', \ldots, \phi_{L_{0}k}')$ differ. Let us denote $x_{i}(s)$ by $x_{i}$ and $x_{i}(s')$ by $x_{i}'$. Note that the path \[
\left(x_{NL_{0}}, \,\, x_{NL_{0}-1}, \,\, \ldots, \,\,x_{(k-1)L_{0}} = x'_{(k-1)L_{0}},\,\, x'_{(k-1)L_{0}+1}, \,\, \ldots, \,\,x'_{NL_{0}}\right)
\]
is the concatenation of $K_{0}$-aligned $K_{0}$-contracting axes \[
(\eta_{i})_{i=1}^{2(N-k+1)}:= \left(\bar{\Gamma}_{N}(s), \bar{\Gamma}_{N-1}(s), \ldots, \bar{\Gamma}_{k}(s), \Gamma_{k}(s'), \ldots, \Gamma_{N}(s')\right).
\] Recall that $s$ violates Condition \ref{eqn:SchottkyIneq1}: $(\bar{\Gamma}_{N}(s), x) = (\eta_{1}, x)$ is not $K_{2}$-aligned. Since $(\eta_{1}, \eta_{2})$ is $K_{0}$-aligned, Lemma \ref{lem:1segmentVar} tells us that $(x, \eta_{2})$ is $K_{1}$-aligned. Then $(x, \eta_{2}, \ldots, \eta_{2(N-k+1)})$ is $K_{1}$-aligned and Proposition \ref{prop:BGIPWitness} tells us that $(x, \eta_{2(N-k+1)}) = (x, \Gamma_{N}(s'))$ is $K_{2}$-aligned. Hence, $s'$ satisfies Condition \ref{eqn:SchottkyIneq1}.

Similarly, by considering the $K_{0}$-aligned sequence \[
\big(\bar{\Gamma}_{N}(s), \,\,\bar{\Gamma}_{N-1}(s),\,\, \ldots,\,\, \bar{\Gamma}_{1}(s), \,\,\Gamma_{-1}(s'), \,\,\Gamma_{-2}(s'),\,\, \ldots, \,\,\Gamma_{-N}(s') \big),
\]
we can deduce that $(x, \Gamma_{-N}(s'))$ is $K_{2}$-aligned as desired.

A similar argument leads to the following.
\begin{obs}\label{obs:Schottky3}
If $s \in S_{0}^{(m)}$ violates Condition \ref{eqn:SchottkyIneq2}, then all the other elements in $S_{0}^{(m)}$ satisfy Condition \ref{eqn:SchottkyIneq1} and Condition \ref{eqn:SchottkyIneq2}.
\end{obs}

\emph{Step 3: Summary}.

We claim that $S_{0}^{(m)}$ is $((K_{0} L_{0} + K_{0})N + K_{2}+K_{4})$-Schottky. The first and the third requirements for Schottky sets were already observed before, so it remains to discuss the second requirement. Considering Observation \ref{obs:Schottky1}, it suffices to show that Condition \ref{eqn:SchottkyIneq1} and Condition \ref{eqn:SchottkyIneq2} are satisfied by all but at most 1 element of $S_{0}^{(m)}$. Observation \ref{obs:Schottky2} and \ref{obs:Schottky3} imply that this is the case.

Given these observations, we can finish the proof by taking $K = (K_{0} L + K_{0} )N + K_{2} + K_{4}$, $m =L$ and by taking any subset $S \subseteq S_{0}^{(m)}$ such that $\#S = N$.
\end{proof}

\section{Pivoting and limit laws} \label{section:pivoting}

In this section, we establish the notion of pivotal times and pivoting. We will then deduce CLT, LIL and geodesic tracking of random walks using probabilistic estimates about pivotal times. The proof of a key probabilistic estimate will be postponed to Section \ref{section:pivotConst}.

\subsection{Pivotal times: statement}\label{subsection:pivotState}

Let $\mu$ be a non-elementary probability measure on $G$ and let $S$ be a long enough and large Schottky set for $\mu$. Then for sufficiently small $\epsilon > 0$, an $n$-step random path $(g_{1}, \ldots, g_{n})$ in the $\mu$-random walk contains at least $\epsilon n$ subsegments \[
(g_{j(i)- M_{0}+1}, \ldots, g_{j(i)}) \in S \quad (i = 1, \ldots, \epsilon n).
\]
The appearance of Schottky sequences in a random path does not necessarily imply something about $Z_{n} = g_{1} \cdots g_{n}$. For example, every Schottky sequence might be cancelled out with the next step, resulting in $Z_{n} = id$. We nonetheless claim that for a high probability, certain number of Schottky axes survive. More explicitly, we seek indices $j(1) < \ldots < j(M)$, called the \emph{pivotal times}, such that the Schottky axes arising at these indices are aligned along $[o, Z_{n} o]$: \[\begin{aligned}
(o, \axes_{j(1)}, \ldots, \axes_{j(M)}, Z_{n} o) \,\,\textrm{is aligned, where} \axes_{j(k)} = (Z_{j(k) - M_{0}} o, \ldots, Z_{j(k)}o).
\end{aligned}
\]
We will observe that for a high probability, a random path has sufficiently many pivotal times. Then, we will freeze the steps except at the pivotal slots and choose the Schottky sequences at the pivotal times from $S$. More explicitly, we will realize a structure where $\axes_{j(k)}$'s are i.i.d.s on the uniform measure on $\{\Gamma(s) : s \in S\}$: once this is guaranteed, we can control the direction $[o, Z_{n} o]$ and establish the deviation inequality.

We now formulate the discussion above.

\begin{definition}\label{dfn:pivotalEquiv}
Let $\mu$ be a non-elementary probability measure on $G$, let $(\Omega, \Prob)$ be a probability space for $\mu$, let $K_{0}, M_{0}>0$ and let $S$ be a long enough $K_{0}$-Schottky set contained in $(\supp \mu)^{M_{0}}$, i.e., $M_{0}$ is as large as described in Definition \ref{dfn:SchottkyLong}.

A subset $\mathcal{E}$ of $\Omega$, accompanied with the choice of a subset $\diffPivot(\mathcal{E}) = \{j(1) < j(2) < \ldots\}\subseteq M_{0} \Z_{>0}$, is called a \emph{pivotal equivalence class} if: \begin{enumerate}
\item for each $i \notin \{j(k) - l : k \ge 1, l =0, \ldots, M_{0} - 1 \}$, $g_{i}(\w)$ is fixed on $\mathcal{E}$;
\item for each $\w \in \mathcal{E}$ and $k \ge 1$, the following is a Schottky sequence: \[
s_{k}(\w) := \big(g_{j(k) - M_{0} + 1}(\w),\, g_{j(k) - M_{0} + 2}(\w), \, \ldots, \, g_{j(k)} (\w) \big) \in S;
\]
\item for each $\w \in \mathcal{E}$, $(o, \, \axes_{j(1)} (\w), \, \axes_{j(2)} (\w), \, \ldots)$ is $D_{0}$-semi-aligned, and
\item on $\mathcal{E}$, $\{s_{1}(\w), s_{2}(\w), \ldots \}$ are i.i.d.s distributed according to the uniform measure on $S$.
\end{enumerate}
We say that $\mathcal{P}(\mathcal{E})$ is the \emph{set of pivotal times} for $\mathcal{E}$.

When a pivotal equivalence class $\mathcal{E} \subseteq \Omega$ is understood, with the set of pivotal times $\mathcal{P}(\mathcal{E})$, for each element $\w$ of $\mathcal{E}$ we call $\diffPivot(\mathcal{E})$ the \emph{set of pivotal times for $\w$} and write it as $\diffPivot(\w)$.
\end{definition}

When the probability space $(\Omega, \Prob)$ for $\mu$ is partitioned into pivotal equivalence classes $\{\mathcal{E}_{\alpha}\}_{\alpha}$, then belonging to the same $\mathcal{E}_{\alpha}$ becomes an equivalence relation. Choosing a different element from the same pivotal equivalence class is called \emph{pivoting}. But note that the choice of pivotal equivalence classes is not canonical: given an $\w \in \Omega$, there are several ways to define the pivotal equivalence class for $\w$. Proposition \ref{prop:gouezelRW1} below describes  a particular choice of pivotal equivalence classes that will be useful.

Let $k$ be a positive integer. We say that a pivotal equivalence class $\mathcal{E}$ \emph{avoids $k$} if $k$ is not in $\{j - l : j \in \diffPivot(\mathcal{E}), \, l = 0, \ldots, M_{0} - 1\}$; in this case, $g_{k}$ is fixed on $\mathcal{E}$.

\begin{prop}\label{prop:gouezelRW1}
Let $\mu$ be a non-elementary probability measure on $G$ and let $S$ be a long enough and large Schottky set for $\mu$. Then there exist a probability space $(\Omega, \Prob)$ for $\mu$ and a constant $K>0$ such that, for each $n \ge 0$, we have a measurable partition $\mathscr{P}_{n} = \{\mathcal{E}_{\alpha}\}_{\alpha}$ of $\Omega$ into pivotal equivalence classes avoiding $1, \ldots, \lfloor n/2 \rfloor + 1$ and $n+1$ that satisfies \begin{equation}\label{eqn:gouezelRWDecay}
\Prob\big(\w : \# (\diffPivot(\w) \cap \{1, \ldots, k\}) \le k/K \, \big| \, g_{1}, \ldots, g_{\lfloor n/2 \rfloor + 1}, g_{n+1} \big) \le K e^{-k/K}
\end{equation}
for each choice of $g_{1}, \ldots, g_{\lfloor n/2 \rfloor + 1}, g_{n+1} \in G$ and $k \ge n$.
\end{prop}

We postpone the proof of Proposition \ref{prop:gouezelRW1} to the next section and  first see its consequence.

\subsection{Pivoting} \label{subsection:pivoting}

Let $K_{0}, N_{0}>0$ and let $S$ be a long enough $K_{0}$-Schottky set with cardinality $N_{0}$. Given isometries $u_{i}$'s, let us draw a choice $s = (s_{1}, s_{2}, \ldots, s_{n})$ from $S^{n}$ with the uniform measure and define \[
U_{n} = u_{0} \Pi(s_{1}) u_{1} \Pi(s_{2})u_{2}\cdots \Pi(s_{n}) u_{n}.
\]
Let $\kappa_{i} := U_{i-1} \Gamma^{+}(s_{i}) =  u_{0} \Pi(s_{1}) \cdots u_{i-1} \Gamma^{+}(s_{i})$. We claim that:

\begin{lem}\label{lem:pivotingAlign1}
We have  \[\begin{aligned}
\Prob \Big( (x, \kappa_{i})\,\,\textrm{is $K_{0}$-aligned for some}\,\, i \le k \Big) \ge 1 - (1/N_{0})^{k}, \\
\Prob \Big( (\kappa_{n-i+1}, U_{n} x)\,\,\textrm{is $K_{0}$-aligned for some}\,\, i \le k \Big) \ge 1 - (1/N_{0})^{k}
\end{aligned}
\]
for each $1 \le k \le n$ and $x \in X$.
\end{lem}

\begin{proof}
We prove the first estimate only; the second one follows similarly. Consider the statement \[
\big( u_{0}^{-1} x, \Gamma^{+}(s_{1})\big) \,\,\textrm{is $K_{0}$-aligned}.
\]
Thanks to the Schottky property, at most 1 choice of $s_{1}$ from $S$ violates this statement. Fixing that bad choice, consider the statement \[
\big( (u_{0} \Pi(s_{1}) u_{1})^{-1} x, \Gamma^{+}(s_{2}) \big) \,\,\textrm{is $K_{0}$-aligned}.
\]
Again, at most 1 choice of $s_{2}$ from $S$ violates this. Keeping this manner, we conclude the following: except at most 1 bad choice among $S^{k}$, \[
\big((u_{0} \Pi(s_{1}) \cdots u_{i-1})^{-1} x, \Gamma^{+}(s_{i})\big) \,\,\textrm{is $K_{0}$-aligned, i.e.,}\,\, \left(x, \kappa_{i} \right)\,\,\textrm{is $K_{0}$-aligned}
\]
holds for at least one $i \le k$. This happens for probability at least $1- (1/N_{0})^{k}$.
\end{proof}

Now fix another set of isometries $\check{u}_{i}$'s and another $K_{0}$-Schottky set $\check{S}$ with cardinality $N_{0}$. We draw  $\check{s} = (\check{s}_{1}, \check{s}_{2}, \ldots, \check{s}_{n})$ from $\check{S}^{n}$ with the uniform measure, independently from $s$, and define \[
\Check{U}_{n} = \check{u}_{0} \prodSeq(\check{s}_{1}) \check{u}_{1} \cdots \prodSeq(\check{s}_{n}) \check{u}_{n}.
\]
Let $\eta_{i} := \check{U}_{i-1} \Gamma^{+}(\check{s}_{i})$. Recall that $\bar{\eta}_{i}$ denotes the reversal of $\eta_{i}$.

\begin{lem}\label{lem:pivotingAlign2}
We have \[\begin{aligned}
\Prob \Big( (\bar{\eta}_{i}, \kappa_{i})\,\,\textrm{is $D_{0}$-aligned for some}\,\, i \le k \Big) \ge 1 - (2/N_{0})^{k}
\end{aligned}
\]
for each $1 \le k \le n$.
\end{lem}

\begin{proof}
Consider the statements \[\begin{aligned}
\big( u_{0}^{-1} \check{u}_{0}\cdot o, \,\,\Gamma^{+}(s_{1})\big) \,\,\textrm{is $K_{0}$-aligned},\\
\big( \check{u}_{0}^{-1} u_{0} \Pi(s_{1}) \cdot o,\,\, \Gamma^{+}(\check{s}_{1})\big) \,\,\textrm{is $K_{0}$-aligned}.
\end{aligned}
\]
Thanks to the Schottky property, at most 1 choice of $s_{1}$ from $S$ violates the first statement. Similarly, given $s_{1}$, at most 1 choice of $\check{s}_{1}$ from $\check{S}$ violates the second statement. In short, the two statements hold for all but at most $2N_{0}$ choices of $(s_{1}, \check{s}_{1}) \in S\times \check{S}$.

Fixing a bad choice $(s_{1}, \check{s}_{1})$, consider the statements \[
\begin{aligned}
\big( (u_{0} \prodSeq(s_{1}) u_{1})^{-1} \check{u}_{0} \prodSeq(\check{s}_{1}) \check{u}_{1} \cdot o,\,\, \Gamma^{+}(s_{2}) \big) \,\,\textrm{is $K_{0}$-aligned}, \\
\big( (\check{u}_{0} \prodSeq(\check{s}_{1}) \check{u}_{1})^{-1} u_{0} \prodSeq(s_{1}) u_{1} \prodSeq(s_{2}) \cdot o,\,\, \Gamma^{+}(\check{s}_{2}) \big) \,\,\textrm{is $K_{0}$-aligned}.
\end{aligned}
\]
Again, at most $2N_{0}$ choices of $(s_{2}, \check{s}_{2}) \in S \times \check{S}$ violates the statements. Keeping this manner, we conclude the following: for probability at least $1 - (2/N_{0})^{k}$, there exists $i \le k$ such that \[\begin{aligned}
\big( (u_{0} \prodSeq(s_{1})  \cdots u_{i-1})^{-1} \check{u}_{0} \prodSeq(\check{s}_{1}) \cdots \check{u}_{i-1}\cdot o, \,\, \Gamma^{+}(s_{i}) \big) \,\, \textrm{is $K_{0}$-aligned}, \\
\big( (\check{u}_{0} \prodSeq(\check{s}_{1})  \cdots \check{u}_{i-1})^{-1} u_{0} \prodSeq(s_{1}) \cdots u_{i-1} \prodSeq(s_{i})\cdot o, \,\,\Gamma^{+}(\check{s}_{i}) \big) \,\, \textrm{is $K_{0}$-aligned}.
\end{aligned}
\]
In other words, $(\textrm{ending point of}\,\,\bar{\eta}_{i}, \kappa_{i})$ and $(\bar{\eta}_{i}, \textrm{ending point of}\,\, \kappa_{i})$ are $K_{0}$-aligned. Lemma \ref{lem:1segment} then tells us that $(\bar{\eta}_{i}, \kappa_{i})$ is $D_{0}$-aligned.
\end{proof}

Applying Lemma \ref{lem:pivotingAlign1} and \ref{lem:pivotingAlign2} to pivotal equivalence classes, we obtain the following corollaries.

\begin{cor}\label{cor:pivotingRW1}
Let $\mu$ be a non-elementary probability measure on $G$, let $K_{0}, N_{0}>0$ and let  $S$ be a long enough $K_{0}$-Schottky set for $\mu$ with cardinality $N_{0}$. Let $\mathcal{E}$ be a pivotal equivalence class for $\mu$ with $\diffPivot(\mathcal{E}) = \{j(1) < j(2) < \ldots\}$ and let $x \in X$. Then for each $k \ge 1$ we have \[
\Prob\left( \big(x, \,\axes_{j(k)}(\w),\, \axes_{j(k+1)}(\w), \ldots \big) \,\, \textrm{is $D_{0}$-semi-aligned} \, \Big| \, \mathcal{E} \right) \ge 1- (1/N_{0})^{k}.
\]
Moreover, for any $m\ge 1$, $n \ge j(m)$ and $k = 1, \ldots, m$, we have \[
\Prob\left( \big(\axes_{j(1)}(\w), \,\ldots, \,\axes_{j(m - k+1)}(\w),\, Z_{n}(\w) o \big) \,\, \textrm{is $D_{0}$-semi-aligned} \, \Big| \, \mathcal{E} \right) \ge 1- (1/N_{0})^{k}.
\]
\end{cor}

\begin{cor}\label{cor:pivotingRW2}
Let $\mu$ be a non-elementary probability measure on $G$ and let $\check{\mu}$ be its reflected version, let $K_{0}, N_{0}>0$ and let $S$ and $\check{S}$ be long enough $K_{0}$-Schottky sets for $\mu$ and $\check{\mu}$, respectively, with cardinality $N_{0}$. Let $\mathcal{E}$ be a pivotal equivalence class for $\mu$ with $\diffPivot(\mathcal{E}) = \{j(1) < j(2) < \ldots\}$, and let $\check{\mathcal{E}}$ be a pivotal equivalence class for $\check{\mu}$ with $\diffPivot(\check{\mathcal{E}}) = \{\check{j}(1) < \check{j}(2) < \ldots\}$.  Then we have  \[
\Prob\left( \big(\bar{\axes}_{j(k)}(\w),\, \axes_{\check{j}(k)}(\check{\w})\big) \,\, \textrm{is $D_{0}$-semi-aligned} \, \Big| \, \mathcal{E} \right) \ge 1- (2/N_{0})^{k}. \quad (\forall k > 0)
\]
\end{cor}

We now record a small consequence of pivoting.

\begin{cor}\label{cor:escapeInfty}
Let $(Z_{n})_{n>0}$ be the random walk generated by a non-elementary probability measure $\mu$ on $G$ with finite first moment. Then there exists a strictly positive quantity $\lambda(\mu) \in (0, +\infty]$, called the \emph{drift} of $\mu$, such that  \[
\lambda(\mu) := \lim_{n \rightarrow \infty} \frac{1}{n} d(o, Z_{n} o) \quad \textrm{almost surely.}
\]
\end{cor}

\begin{remark}
The statement in Corollary \ref{cor:escapeInfty} holds true even without the moment condition. This will be the consequence of Theorem \ref{thm:LDPStrong} in Section \ref{section:LDP}.
\end{remark}

\begin{proof}
By Kingman's subadditive ergodic theorem, $\lambda(\mu) = \lim_{n} \frac{1}{n} d(o, Z_{n} o)$ exists and is constant almost surely. It remains to show that $\lambda(\mu) > 0$.

Since $\mu$ is non-elementary, Proposition \ref{prop:Schottky} provides a long enough and large Schottky set $S$ for $\mu$. Given this, Proposition \ref{prop:gouezelRW1} provides a constant $K>0$ and a measurable partition $\mathscr{P} = \{\mathcal{E}_{\alpha}\}_{\alpha}$ into pivotal equivalence classes such that \[
\Prob\big(\w : \# (\mathcal{P}(\w) \cap \{1, \ldots, k\}) \le k/K \big) \le K e^{-k/K}
\]
for each $k$. Now let $n >0$ and let $\mathcal{E}$ be a pivotal equivalence class with $\diffPivot(\mathcal{E}) = \{j(1) < j(2) < \ldots\}$ such that $\# (\mathcal{P}(\mathcal{E}) \cap \{1, \ldots, n\}) \ge n/K$, i.e., $j(\lfloor n/ K \rfloor) \le n$. Corollary \ref{cor:pivotingRW1} tells us that \[
\Prob \Big( ( o, \axes_{j(1)} (\w), \ldots, \axes_{j(\lfloor n/2K \rfloor)}(\w), Z_{n} o)\,\textrm{is $D_{0}$-semi-aligned} \, \Big| \, \mathcal{E} \Big) \ge 1 - \left(1/\#S_{0}\right)^{-n/2K + 1}.
\]
By Lemma \ref{lem:SchottkyAlign}, we then have  \[
\Prob \Big( d(o, Z_{n} o) < 50 E_{0}  n/2K \, \Big| \, \mathcal{E} \Big) \le \left(1/\#S_{0}\right)^{-n/2K + 1}.
\]
We sum up these conditional probabilities on $\{\w : \#(\diffPivot(\mathcal{E} \cap \{1, \ldots, n\}) \ge n/K\}$ to conclude \[
\Prob \Big( d(o, Z_{n} o) < 50 E_{0}  n/2K \Big) \le \left(1/\#S_{0}\right)^{-n/2K + 1} + K e^{-n/K}.
\]
The Borel-Cantelli lemma then implies $d(o, Z_{n} o) \ge 50 E_{0} n/2K$ eventually almost surely.
\end{proof}

\subsection{Deviation inequality} \label{subsection:deviation}

Let $\mu$ be a non-elementary probability measure on $G$ and let $S$ be a long enough and large $K_{0}$-Schottky set contained in $(\supp \mu)^{M_{0}}$ for some $K_{0}, M_{0} > 0$. Consider a bi-infinite path $\big((Z_{n}(\w))_{n > 0}, (Z_{n}(\check{\w}))_{n > 0}\big)$ arising from the random walk generated by $\mu$. Recall: \[\begin{aligned}
\axes_{i}(\w) &:= (Z_{i-M_{0}} o, \,Z_{i-M_{0} + 1} o,\, \ldots, \,Z_{i} o), \\
\axes_{i}(\check{\w}) &:= (\check{Z}_{i-M_{0}} o, \,\check{Z}_{i-M_{0} + 1} o,\, \ldots, \,\check{Z}_{i} o).
\end{aligned}
\]
For each $k\ge M_{0}$, we investigate whether there exists $M_{0} \le i \le k$ such that: \begin{enumerate}
\item $(g_{i-M_{0} + 1}, \ldots, g_{i})$ is a Schottky sequence;
\item $(\check{Z}_{m} o, \axes_{i}(\w), Z_{n} o)$ is $D_{0}$-semi-aligned for all $n \ge k$ and $m \ge 0$.
\end{enumerate}
We define $\Devi = \Devi(\check{\w}, \w)$ as the minimal index $k$ with the auxiliary index $i \le k$ as described above. 

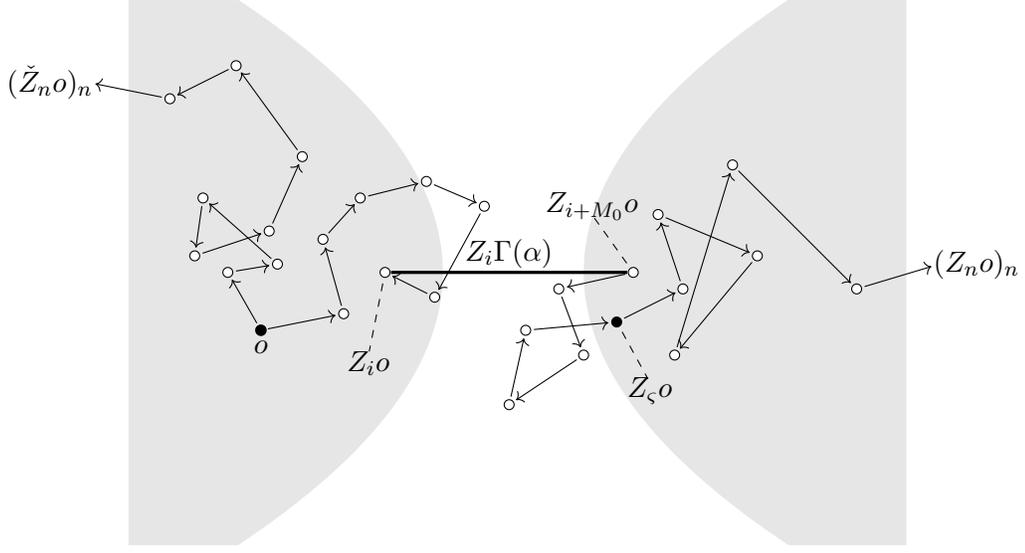
\begin{figure}
\begin{tikzpicture}
\def\c{1.1}

\fill[opacity=0.1, shift={(0, 0.7*\c)}] (-1.6*\c, 3.3*\c) -- (-1.6*\c, -3.3*\c) -- plot[smooth,domain=-3.3:3.3] (2.2*\c - 0.25*\x*\x, \x*\c) -- cycle;

\fill[opacity=0.1, shift={(0, 0.7*\c)}] (7.8*\c, 3.3*\c) -- (7.8*\c, -3.3*\c) -- plot[smooth,domain=-3.3:3.3] (3.9*\c + 0.25*\x*\x, \x*\c) -- cycle;

\fill (0, 0) circle (0.07*\c);
\draw[fill=black!0] (1*\c, 0.2*\c) circle (0.06*\c);
\draw[fill=black!0] (0.75*\c, 1.1*\c) circle (0.06*\c);
\draw[fill=black!0] (1.2*\c, 1.6*\c) circle (0.06*\c);
\draw[fill=black!0] (2*\c, 1.8*\c) circle (0.06*\c);
\draw[fill=black!0] (2.7*\c, 1.5*\c) circle (0.06*\c);
\draw[fill=black!0] (2.1*\c, 0.4*\c) circle (0.06*\c);
\draw[fill=black!0] (1.5*\c, 0.7*\c) circle (0.06*\c);
\draw[very thick, shorten >=0.08cm,shorten <=0.08cm] (1.5*\c, 0.7*\c) -- (4.5*\c, 0.7*\c);

\draw[shorten >=0.11cm,shorten <=0.095cm,->] (0, 0) -- (1*\c, 0.2*\c);
\draw[shorten >=0.11cm,shorten <=0.11cm,->] (1*\c, 0.2*\c) -- (0.75*\c, 1.1*\c);
\draw[shorten >=0.11cm,shorten <=0.11cm,->] (0.75*\c, 1.1*\c) -- (1.2*\c, 1.6*\c);
\draw[shorten >=0.11cm,shorten <=0.11cm,->] (1.2*\c, 1.6*\c) -- (2*\c, 1.8*\c);
\draw[shorten >=0.11cm,shorten <=0.11cm,->] (2*\c, 1.8*\c) -- (2.7*\c, 1.5*\c);
\draw[shorten >=0.11cm,shorten <=0.11cm,->] (2.7*\c, 1.5*\c) -- (2.1*\c, 0.4*\c);
\draw[shorten >=0.11cm,shorten <=0.11cm,->] (2.1*\c, 0.4*\c) -- (1.5*\c, 0.7*\c);

\draw[fill=black!0] (4.5*\c, 0.7*\c) circle (0.06*\c);
\draw[fill=black!0] (3.6*\c, 0.5*\c) circle (0.06*\c);
\draw[fill=black!0] (3.9*\c, -0.3*\c) circle (0.06*\c);
\draw[fill=black!0] (3*\c, -0.9*\c) circle (0.06*\c);
\draw[fill=black!0] (3.2*\c, 0) circle (0.06*\c);
\draw[fill=black] (4.3*\c, 0.1*\c) circle (0.06*\c);
\draw[fill=black!0] (5.1*\c, 0.5*\c) circle (0.06*\c);
\draw[fill=black!0] (4.8*\c, 1.4*\c) circle (0.06*\c);
\draw[fill=black!0] (6*\c, 0.9*\c) circle (0.06*\c);
\draw[fill=black!0] (5*\c, -0.3*\c) circle (0.06*\c);
\draw[fill=black!0] (5.7*\c, 2*\c) circle (0.06*\c);
\draw[fill=black!0] (7.2*\c, 0.5*\c) circle (0.06*\c);

\draw[shorten >=0.11cm,shorten <=0.11cm,->] (4.5*\c, 0.7*\c) -- (3.6*\c, 0.5*\c);
\draw[shorten >=0.11cm,shorten <=0.11cm,->] (3.6*\c, 0.5*\c) -- (3.9*\c, -0.3*\c);
\draw[shorten >=0.11cm,shorten <=0.11cm,->] (3.9*\c, -0.3*\c) -- (3*\c, -0.9*\c) ;
\draw[shorten >=0.11cm,shorten <=0.11cm,->] (3*\c, -0.9*\c) -- (3.2*\c, 0);
\draw[shorten >=0.11cm,shorten <=0.11cm,->] (3.2*\c, 0) -- (4.3*\c, 0.1*\c);
\draw[shorten >=0.11cm,shorten <=0.11cm,->] (4.3*\c, 0.1*\c) -- (5.1*\c, 0.5*\c);
\draw[shorten >=0.11cm,shorten <=0.11cm,->] (5.1*\c, 0.5*\c) -- (4.8*\c, 1.4*\c);
\draw[shorten >=0.11cm,shorten <=0.11cm,->] (4.8*\c, 1.4*\c) -- (6*\c, 0.9*\c);
\draw[shorten >=0.11cm,shorten <=0.11cm,->] (6*\c, 0.9*\c) -- (5*\c, -0.3*\c);
\draw[shorten >=0.11cm,shorten <=0.11cm,->] (5*\c, -0.3*\c) -- (5.7*\c, 2*\c);
\draw[shorten >=0.11cm,shorten <=0.11cm,->] (5.7*\c, 2*\c) -- (7.2*\c, 0.5*\c);
\draw[shorten >=0.11cm,shorten <=0.11cm,->] (7.2*\c, 0.5*\c) -- (8.2*\c, 0.8*\c);

\draw[fill=black!0] (-0.4*\c, 0.7*\c) circle (0.06*\c);
\draw[fill=black!0] (0.2*\c, 0.8*\c) circle (0.06*\c);
\draw[fill=black!0] (-0.7*\c, 1.6*\c) circle (0.06*\c);
\draw[fill=black!0] (-0.8*\c, 0.9*\c) circle (0.06*\c);
\draw[fill=black!0] (0.1*\c, 1.2*\c) circle (0.06*\c);
\draw[fill=black!0] (0.5*\c, 2.1*\c) circle (0.06*\c);
\draw[fill=black!0] (-0.3*\c, 3.2*\c) circle (0.06*\c);
\draw[fill=black!0] (-1.1*\c, 2.8*\c) circle (0.06*\c);

\draw[shorten >=0.11cm,shorten <=0.11cm,->] (0, 0) --(-0.4*\c, 0.7*\c) ;
\draw[shorten >=0.11cm,shorten <=0.11cm,->] (-0.4*\c, 0.7*\c) -- (0.2*\c, 0.8*\c) ;
\draw[shorten >=0.11cm,shorten <=0.11cm,->] (0.2*\c, 0.8*\c) -- (-0.7*\c, 1.6*\c);
\draw[shorten >=0.11cm,shorten <=0.11cm,->] (-0.7*\c, 1.6*\c) -- (-0.8*\c, 0.9*\c);
\draw[shorten >=0.11cm,shorten <=0.11cm,->] (-0.8*\c, 0.9*\c) -- (0.1*\c, 1.2*\c);
\draw[shorten >=0.11cm,shorten <=0.11cm,->] (0.1*\c, 1.2*\c)-- (0.5*\c, 2.1*\c);
\draw[shorten >=0.11cm,shorten <=0.11cm,->] (0.5*\c, 2.1*\c) -- (-0.3*\c, 3.2*\c);
\draw[shorten >=0.11cm,shorten <=0.11cm,->] (-0.3*\c, 3.2*\c) -- (-1.1*\c, 2.8*\c);
\draw[shorten >=0.11cm,shorten <=0.11cm,->] (-1.1*\c, 2.8*\c) -- (-2.1*\c, 3*\c);

\draw (0,0) node[below] {$o$};
\draw[dashed, shorten >=0.05cm, shorten <=0.15cm] (1.5*\c, 0.7*\c) -- (1.3*\c, -0.3*\c);
\draw (1.3*\c, -0.4*\c) node{$Z_{i} o$};
\draw[dashed, shorten >=0.05cm, shorten <=0.15cm] (4.5*\c, 0.7*\c) -- (4*\c, 1.4*\c);
\draw (4*\c, 1.5*\c) node {$Z_{i+M_{0}} o$};
\draw (3*\c, 0.93*\c) node {$Z_{i} \Gamma^{+}(\alpha)$};
\draw[dashed, shorten >=0.05cm, shorten <=0.15cm] (4.3*\c, 0.1*\c) -- (4.7*\c, -0.63*\c);
\draw (4.7*\c, -0.73*\c) node {$Z_{\DeviUni}o$};
\draw (8.65*\c, 0.8*\c) node {$(Z_{n}o)_{n}$};
\draw (-2.55*\c, 3*\c) node {$(\check{Z}_{n} o)_{n}$};

\end{tikzpicture}
\caption{Persistent progress and $\Devi$. Here, all of the backward loci $(\check{Z}_{n} o)_{n \ge 0}$ are on the left of the persistent progress $Z_{i}\Gamma^{+}(\alpha)$, while the forward loci after $Z_{\DeviUni} o$ are all on the right.}
\label{fig:Devi}
\end{figure}

A motivating observation for the definition of $\Devi(\check{\w}, \w)$ is as follows. 

\begin{lem}\label{lem:DeviGromProd}
Let $\Omega = G^{\Z_{>0}} \times G^{\Z_{>0}}$ be the space of (bi-directional) step paths in $G$, let $K_{0}>0$ and let $S$ be a long enough $K_{0}$-Schottky set. Then for each $(\check{\w}, \w) \in \Omega$, we have \[
(\check{Z}_{m} o, Z_{n} o)_{o} \le d(o, Z_{k} o)
\]
for all $m\ge 0$ and $n, k \ge \Devi(\check{\w}, \w)$.
\end{lem}

\begin{proof}
Let $i \le \Devi(\check{\w}, \w)$ be the index such that $(\check{Z}_{m'} o, \axes_{i}(\w), Z_{n'} o)$ is $D_{0}$-semi-aligned for all $n' \ge \Devi(\check{\w}, \w)$ and $m' \ge 0$. Lemma \ref{lem:SchottkyAlign} tells us that \[\begin{aligned}
d(\check{Z}_{m'} o, Z_{n'} o) &\ge d(\check{Z}_{m'} o, Z_{i-M_{0}} o) + d(Z_{i-M_{0}} o, Z_{i} o) +  d(Z_{i} o, Z_{n'} o) -  E_{0},\\
d(\check{Z}_{m'} o, Z_{n'} o) &\ge d(\check{Z}_{m'} o, Z_{i-M_{0}} o) + d(Z_{i} o, Z_{n'} o) + 50E_{0}. 
\end{aligned}  \quad (n' \ge \Devi(\check{\w}, \w), m' \ge 0)
\]
Let us now pick $n, k \ge \Devi(\check{\w}, \w)$ and $m \ge 0$. Then we have \begin{align}
\begin{aligned}
d(\check{Z}_{m} o, Z_{n} o) &\ge d(\check{Z}_{m} o, Z_{i-M_{0}} o) + d(Z_{i-M_{0}} o, Z_{i} o) +  d(Z_{i} o, Z_{n} o) -  E_{0}\\
&\ge d(\check{Z}_{m} o, Z_{i-M_{0}} o) + d(Z_{i-M_{0}} o, Z_{n} o) -  E_{0},\\ \label{eqn:deviDomDist}
d( o, Z_{k} o) &\ge d( o, Z_{i-M_{0}} o) + d(Z_{i} o, Z_{k} o) + 50E_{0} \ge d(o, Z_{i-M_{0}} o) + 50E_{0}.
\end{aligned} 
\end{align}
Hence,
\[\begin{aligned}
2(\check{Z}_{m} o, Z_{n} o)_{o} &= d(\check{Z}_{m} o, o) + d(o, Z_{n} o) - d(\check{Z}_{m} o, Z_{n} o) \\
&\le \big(d(\check{Z}_{m} o, Z_{i-M_{0}} o) +d(Z_{i-M_{0}} o, o)\big) + \big(d(o, Z_{i-M_{0}} o) + d(Z_{i-M_{0}} o, Z_{n} o)\big)\\
& \quad - \big(d(\check{Z}_{m} o, Z_{i-M_{0}} o) + d(Z_{i-M_{0}} o, Z_{n} o) -  E_{0}\big) \\
&\le 2d(o, Z_{i-M_{0}} o) + E_{0} \le 2d(o, Z_{k} o). \qedhere
\end{aligned}
\]
\end{proof}

We now provide a probabilistic estimate for $\Devi(\check{\w}, \w)$.

\begin{lem}\label{lem:Devi}
Let $\mu$ be a non-elementary probability measure on $G$, let $K_{0}>0$ and let $S$ be a long enough and large $K_{0}$-Schottky set for $\mu$. Then there exists $K'>0$ such that 
\begin{equation}\label{eqn:DeviDecays1}
\Prob\left(\Devi(\check{\w}, \w) \ge k \, \Big|\, g_{k+1}, \check{g}_{1}, \ldots, \check{g}_{k+1}\right) \le K' e^{-k/K'}
\end{equation}
holds for all $k \ge 0$ and all choices of $g_{k+1}, \check{g}_{1}, \ldots, \check{g}_{k+1} \in G$.
\end{lem}

\begin{proof}
Let $S$ be a long enough and large $K_{0}$-Schottky set in $(\supp \mu)^{M_{0}}$ for some $M_{0}>0$. Let $\check{S}$ be the reflected version of $S$, that means, \[
\check{S} := \big\{ \big(s_{M_{0}}^{-1}, \ldots, s_{1}^{-1}\big) : (s_{1}, \ldots, s_{M_{0}}) \in S\big\}.
\]
Then $\check{S}$ is a long enough and large $K_{0}$-Schottky set for $\check{\mu}$. Let $K>0$ be the constant determined for $S$ and $\check{S}$ in Proposition \ref{prop:gouezelRW1}. We now fix $k$ and $g_{k+1}, \check{g}_{1}, \ldots, \check{g}_{k+1} \in G$.

Let $\mathscr{P}_{k} = \{\mathcal{E}_{\alpha}\}_{\alpha}$ be the partition of $\Omega$ into pivotal equivalence classes avoiding $1, \ldots, \lfloor k/2 \rfloor+1$ and $k+1$, given by Proposition \ref{prop:gouezelRW1}. Let also $\check{\mathscr{P}}_{2k} = \{\check{\mathcal{E}}_{\alpha}\}_{\alpha}$ be the partition of $\check{\Omega}$ into pivotal equivalence classes avoiding $1, \ldots, k+1$ and $2k+1$, given by Proposition \ref{prop:gouezelRW1}. We have \[\begin{aligned}
\Prob\Big(A := \big\{\w : \#(\diffPivot (\w)  \cap \{1, \ldots, n\}) \ge n/K \,\, \textrm{for all $n \ge k$} \big\} \Big) &\ge 1 - \frac{K}{1-e^{-1/K}} e^{-k/K}, \\
\Prob\Big(\check{A} := \big\{\check{\w}: \# (\diffPivot(\check{\w}) \cap \{1, \ldots, n\}) \ge n/K \,\,\textrm{for all $n \ge 2k$}\big\}\Big) &\ge 1- \frac{K}{1-e^{-1/K}} e^{-2k/K}.
\end{aligned}
\]
Let us enumerate $\mathcal{P}(\w)$ by $\{j(1) < j(2) < \ldots \}$, and $\mathcal{P}(\check{\w})$ by $\{\check{j}(1) < \check{j}(2) < \ldots\}$. Let $\mathcal{E} \in \mathscr{P}_{k}$ and $\check{\mathcal{E}} \in \check{\mathscr{P}}_{2k}$ be pivotal equivalence classes in $A$ and $\check{A}$, respectively. In $\check{\mathcal{E}}\times \mathcal{E}$, let $B$ be the set of $(\check{\w}, \w)$ that satisfies the following: \begin{enumerate}
\item for $x \in \{o, \check{Z}_{1} o, \ldots, \check{Z}_{2k} o\}$, the following sequence is $D_{0}$-semi-aligned: \[
\big(x, \,\axes_{j(\lceil k/3K\rceil )}(\w),  \,\axes_{j(\lceil k/3K\rceil + 1)}(\w), \,\ldots\big);
\]
\item for each $n \ge k$ and $m \ge 2k$, the following are $D_{0}$-semi-aligned: \[\begin{aligned}
&\big(o, \,\axes_{j(1)}(\w), \, \axes_{j(2)}(\w),  \, \ldots, \, \axes_{j(\lceil 2n/3K\rceil)}(\w), \,Z_{n} o\big),\\
&\big(o, \,\axes_{\check{j}(1)}(\check{\w}), \, \axes_{\check{j}(2)}(\check{\w}),  \, \ldots, \, \axes_{\check{j}(\lceil 2m/3K\rceil)}(\check{\w}), \,\check{Z}_{m} o\big);
\end{aligned}
\]
\item $\big(\bar{\axes}_{\check{j}(i)}(\check{\w}),\, \axes_{j(i)}(\w) \big)$ is $D_{0}$-aligned for some $i \le k/3K$.
\end{enumerate}
The first item is handled by Lemma \ref{lem:pivotingAlign1}: it holds for probability at least $1-2k \cdot (1/400)^{k/3K}$. 

Next, recall that for each $n \ge k$, there are at least $n/K$ pivotal times for $\mathcal{E}$ before $n$. Also, for each $m \ge 2k$, there are at least $m/K$ pivotal times for $\check{\mathcal{E}}$ before $m$. Hence, we can apply Lemma \ref{lem:pivotingAlign1} and deduce that the following are $D_{0}$-semi-aligned:\[\begin{aligned}
\big(\axes_{j(1)}(\w), \, \ldots, \, \axes_{j(\lceil 2n/3K\rceil)}(\w), \,Z_{n} o\big),\quad \big( \axes_{\check{j}(1)}(\check{\w}),\, \ldots,\,\axes_{\check{j}(\lceil 2m/3K\rceil )}(\check{\w}), \, \check{Z}_{m} o\big),
\end{aligned}
\]
for probability at least $1-(1/400)^{n/3K-1}$ and $1-(1/400)^{m/3K - 1}$, respectively. Taking intersection for $n \ge k$ and $m \ge 2k$, we observe that Item (ii) holds for probability at least $1-3 \cdot (1/400)^{k/3K - 1}$. 

Finally, Item (iii) is handled by Lemma \ref{lem:pivotingAlign2}: it holds for probability at least $1- (1/200)^{k/3K - 1}$. Combining these, we deduce \[
\Prob\big( B \, \big| \, \check{\mathcal{E}} \times \mathcal{E}\big) \ge 1 - (2k+4) \cdot (1/200)^{k/3K - 1} \ge 1 - 200 \cdot(2k+4) \cdot 0.01^{k/3K}.
\]

It remains to prove that $\Devi(\check{\w}, \w) \le k$ for $(\check{\w}, \w) \in B$. First, by  definition of $A$, $j(\lceil k/3K\rceil )$ is smaller than $k$ and \[
s_{\lceil k/3K\rceil} = (g_{j(\lceil k/3K\rceil )-M_{0}+1}, \ldots, g_{j(\lceil k/3K\rceil ) })
\] is Schottky. Next, for each $n \ge k$, $(o, \axes_{j(1)}(\w), \axes_{j(2)}(\w),   \ldots, \axes_{j(\lceil 2n/3K \rceil)}(\w), Z_{n} o)$ is $D_{0}$-semi-aligned. Hence, $(o, \axes_{j(\lceil k/3K\rceil)}(\w), Z_{n} o)$ is also $D_{0}$-semi-aligned.

We now investigate the alignment of $(\check{Z}_{m} o, \axes_{j(\lceil k/3K\rceil )}(\w))$. For $m \le 2k$, this is guaranteed by item (1). When $m \ge 2k$, we appeal to item (2) and (3). Namely, the sequence\[
\big( \check{Z}_{m} o, \bar{\axes}_{\check{j}(\lceil 2m3/\rceil)}(\check{\w}), \ldots, \bar{\axes}_{\check{j}(i+1)}(\check{\w}),  \bar{\axes}_{\check{j}(i)}(\check{\w}), \axes_{j(i)}(\w), \ldots, \axes_{j(\lceil k/3K \rceil)}(\w), \axes_{j(\lceil k/3K \rceil + 1)}, \ldots \big).
\]
is $D_{0}$-semi-aligned. In particular,  $\big(\check{Z}_{m} o, \bar{\axes}_{j(\lceil k/3K \rceil )} (\w) \big)$ is $D_{0}$-semi-aligned.
\end{proof} 

\begin{comment}
Now Lemma \ref{lem:SchottkyAlign} tells us that  \[\begin{aligned}
(\check{Z}_{m} o, Z_{n} o)_{o}\le d(o, Z_{k} o)
\end{aligned}
\]
for all $m \ge 0$ and $n, k \ge \Devi(\check{\w}, \w)$. Hence, we deduce:

\begin{cor}\label{cor:GromPivot}
In the setting of Lemma \ref{lem:Devi}, we also have \[\Prob\left.\left(\sup_{m\ge 0, n \ge k} (\check{Z}_{m} o, Z_{n} o)_{o} \ge d(o, Z_{k} o) \,\right| \,g_{k+1}, \check{g}_{1}, \ldots, \check{g}_{k+1}\right)\le Ke^{-k/K}\]
for all $k$ and all choices of $g_{k+1}, \check{g}_{1}, \ldots, \check{g}_{k+1} \in G$.
\end{cor}
\end{comment}

Here is a corollary of Lemma \ref{lem:Devi} that we will use in Section \ref{section:LDP}.

\begin{cor}[{\cite[Lemma 4.14]{gouezel2022exponential}}]\label{cor:gouezelRW1Cor}
Let $\mu$ be a non-elementary probability measure on $G$ and let $(Z_{n})_{n}$ be the random walk generated by $\mu$. Then for each $\epsilon > 0$, there exists $C>0$ such that  \[
\Prob \big(d(o, g Z_{n}o) \ge d(o, go) - C\,\,\textrm{for all $n \ge 0$}\big) \ge 1-\epsilon/2 \quad (\forall \, g \in G).
\]
\end{cor}

\begin{proof}
Let us pick $K_{0} > 0$ and a long enough and large $K_{0}$-Schottky set $S$ for $\mu$. Let $K'$ be the constant as in Lemma \ref{lem:Devi}. Given $\epsilon>0$, we take $N>1$ large enough so that $K' e^{-N / K'} \le \epsilon /4$. Then, the definition of the RV $\Devi(\check{\w}, \w)$ and Lemma \ref{lem:Devi} tells us that  \[
\Prob \left( \begin{array}{c} \textrm{there exists $i < N$ such that $\axes_{i}$ is a Schottky axis and}\\ \textrm{ $(g^{-1} o, \axes_{i}, Z_{n} o)$ is $D_{0}$-semi-aligned for each $n \ge N$}\end{array} \, \Big| \, \check{g}_{1} = g^{-1}  \right) \ge 1-\epsilon/4.
\]

When $(g^{-1} o, \axes_{i}, Z_{n} o)$ is $D_{0}$-semi-aligned, the second inequality in Lemma \ref{lem:SchottkyAlign}(ii) implies \[
d(g^{-1}o, Z_{n} o) \ge d(g^{-1} o, Z_{i - M_{0}} o) + 50E_{0} N \ge d(g^{-1} o, o) - d(Z_{i-M_{0}} o, o) \ge d(o, go) - \sum_{j=1}^{N} d(o, g_{j}o).
\]
This bound also holds for $n \le N$: \[
d(g^{-1} o, Z_{n} o) \ge d(g^{-1} o, o) - d(Z_{n} o, o) \ge d(o, go) - \sum_{j=1}^{N} d(o, g_{j} o).
\]
Given these, the proof ends by taking large enough $C>0$ such that \[
\Prob \Big( \sum_{j=1}^{N} d(o, g_{j} o) \ge C\Big) \le \epsilon/4. \qedhere
\]
\end{proof}

\begin{cor}\label{cor:SLLNInfinite}
Let $\mu$ be a non-elementary probability measure on $G$ whose expectation is infinite. Then $\mu^{\ast m}$ has infinite expectation for each $m>0$. In particular, the drift $\lambda(\mu) := \lim_{m \rightarrow \infty} \frac{1}{m} \E_{\mu^{\ast m}} [d(o, go)]$ is infinity.
\end{cor}

\begin{proof}
Let $\epsilon =0.2$ and let $C = C(\mu, \epsilon)$ be as in Corollary \ref{cor:gouezelRW1Cor}. Let $(g_{1}, \ldots, g_{m})$ be distributed according to $\mu^{m}$. Then by Corollary \ref{cor:gouezelRW1Cor}, we have \[
\E\big[ d(o, g_{1} g_{2} \cdots g_{m} o) \, \big| \, g_{1} = g \big] \ge \E\big[( d(o, go) - C) \cdot 1_{\{ d(o, g g_{2} \cdots g_{m} o) \ge d(o, g o) \}}\, \big| \, g_{1}= g\big] \ge 0.9 \cdot (d(o, go) -C).
\]
Now integrating over $g_{1} \in \supp \mu$ with law $\mu$, we get \[
\E\big[d(o, g_{1}g_{2} \cdots g_{m} o) \big] \ge 0.9\E_{\mu} [d(o, go) - C] = +\infty. \qedhere
\]
\end{proof}

Similarly, fixing the Schottky set $S$ for $\mu$, we similarly define $\check{\Devi} = \check{\Devi}(\check{\w}, \w)$ as the minimal index $k$ that is associated with another index $i \le k$ such that: \begin{enumerate}
\item $(\check{g}_{i}^{-1}, \ldots, \check{g}_{i-M_{0} + 1}^{-1})$ is a Schottky sequence;
\item $(\check{Z}_{m}o, \bar{\axes}_{i}(\check{\w}), o)$ is $D_{0}$-semi-aligned for all $m \ge k$, and 
\item $(\bar{\axes}_{i}(\check{\w}), Z_{n} o)$ is $D_{0}$-semi-aligned for all $n\ge 0$.
\end{enumerate}
Then we similarly have \begin{equation}\label{eqn:DeviDecays2} \begin{aligned}
\Prob\left(\check{\Devi}(\check{\w}, \w) \ge k \, \Big|\, \check{g}_{k+1}, g_{1}, \ldots, g_{k+1}\right) \le K' e^{- k/K'}.
\end{aligned}
\end{equation}
%\Prob\left.\left(\sup_{m\ge 0, n \ge k} (\check{Z}_{m} o, Z_{n} o)_{o} \ge d(o, \check{Z}_{k} o) \right| \check{g}_{k+1}, g_{1}, \ldots, g_{k+1}\right) \le Ke^{-k/K}.

Thanks to these exponential bounds, we can establish the deviation inequality.

\begin{prop}\label{prop:deviation}
Let $p>0$ and let $((\check{Z}_{n})_{n}, (Z_{n})_{n})$ be the (bi-directional) random walk generated by a non-elementary probability measure $\mu$ on $G$ with finite $p$-th moment. Then the random variable $\sup_{n, m \ge 0} (\check{Z}_{m} o, Z_{n} o)_{o}$ has finite $2p$-th moment.
\end{prop}

Note the difference between this proposition and \cite[Proposition 5.6, 5.8]{choi2023central}; we are taking the global suprema, not the limit suprema.

\begin{proof}
Let $K'$ be the constant for $\mu$ as in Lemma \ref{lem:Devi} and let\[
D_{k} := \sum_{i=1}^{k} d(o, g_{i}o),  \quad \check{D}_{k} := \sum_{i=1}^{k} d(o, \check{g}_{i} o).
\] 
By triangle inequality, $d(o, Z_{k} o) < D_{l}$ and $d(o, \check{Z}_{k} o) \le \check{D}_{l}$ for all $k \le l$. We begin by claiming \begin{equation}\label{eqn:deviClaim}
\sup_{n, m \ge 0} (\check{Z}_{m} o, Z_{n} o)_{o}^{2p} \le \sum_{i=0}^{\infty} |\check{D}_{i+1}^{p} D_{i+1}^{p} - \check{D}_{i}^{p} D_{i}^{p}| \left(1_{\check{D}_{i} \ge D_{i}} 1_{i < \Devi} + 1_{\check{D}_{i} \le D_{i}} 1_{i < \check{\Devi}}\right) \quad \textrm{almost surely}.
\end{equation}
Since $\Prob(\max\{\Devi, \check{\Devi}\} \ge k)$ is summable by Inequality \ref{eqn:DeviDecays1} and \ref{eqn:DeviDecays2}, Borel-Cantelli implies that \[
l := \min \left\{ i : 1_{\check{D}_{i} \ge D_{i}} 1_{i < \Devi} + 1_{\check{D}_{i} \le D_{i}} 1_{i < \check{\Devi}}=0\right\} <+\infty \quad \textrm{almost surely}.
\]
Note that the RHS of Inequality \ref{eqn:deviClaim} is at least $ \check{D}_{l}^{p} D_{l}^{p}$.

Now at $i = l$, we have either $\check{D}_{l} \ge D_{l}$ or $\check{D}_{l} \le D_{l}$. In the first case $l \ge \Devi$ must hold. Then for $m \ge 0$ and $n \ge l$, we have \[\begin{aligned}
(\check{Z}_{m} o, Z_{n} o)_{o}^{2p} \le d(o, Z_{l} o)^{2p} \le D_{l}^{2p} \le\check{D}_{l}^{p}D_{l}^{p}
\end{aligned}
\]
by Lemma \ref{lem:DeviGromProd}. Moreover, for $m \ge 0$ and $n \le l$, we have \[
(\check{Z}_{m} o, Z_{n} o)_{o}^{2p} \le d(o, Z_{n} o)^{2p} \le D_{n}^{2p} \le D_{l}^{2p} \le \check{D}_{l}^{p}D_{l}^{p}.
\]
In the second case $l \ge \check{\Devi}$ must hold, and for a similar reason $(\check{Z}_{m} o, Z_{n} o)_{o}^{2p}$ is dominated by $\check{D}_{l}^{p} D_{l}^{p}$. Inequality \ref{eqn:deviClaim} now follows.

We now need a small observation:

\begin{fact}\label{fact:deviationRealIneq}
For $s_{1}, s_{2}, t_{1}, t_{2} \ge 0$, the following holds: \[\begin{aligned}
|t_{1}^{p} t_{2}^{p} - s_{1}^{p} s_{2}^{p}| &=| t_{1}^{p}(t_{2}^{p} - s_{2}^{p}) + (t_{1}^{p} - s_{1}^{p})s_{2}^{p}| \\
&\le 2^{2p} \left( |t_{1} - s_{1}|^{p} + s_{1}^{p-n_{p}} |t_{1} - s_{1}|^{n_{p}} +s_{1}^{p} \right) \cdot \left( |t_{2} - s_{2}|^{p} + s_{2}^{p-n_{p}} |t_{2} - s_{2}|^{n_{p}} \right) \\
& + 2^{p} \left( |t_{1} - s_{1}|^{p} + s_{1}^{p-n_{p}} |t_{1} - s_{1}|^{n_{p}}\right) s_{2}^{p}. \quad (n_{p} = p\,\, \textrm{if}\,\, 0 \le p \le 1, \,\, n_{p} = 1 \,\, \textrm{otherwise})
\end{aligned}
\]
\end{fact}

\begin{proof}[Proof of Fact \ref{fact:deviationRealIneq}]
The fact follows from the following inequality in \cite[Section 5.4]{benoist2016central}:\[
|t^{p} - s^{p} |\le  2^{p} \big( |t-s|^{p} + s^{p-n_{p}} |t-s|^{n_{p}} \big) \quad (n_{p} = p\,\, \textrm{if}\,\, 0 \le p \le 1, \,\, n_{p} = 1 \,\, \textrm{otherwise}).
\]
We give its proof for completeness. Assume $t \ge s$ without loss of generality. When $p \le 1$,  the concavity of $f(x) = x^{p}$ implies the inequality. When $p >1$, we divide the cases. If $s<t/2$, then\[
t^{p} - s^{p} < t^{p} < (2(t-s) )^{p} \le 2^{p} |t-s|^{p}.
\]
If $s \ge t/2$, then we have\[\begin{aligned}
t^{p} - s^{p} &= \int_{s}^{t} p x^{p-1} \, dx \le \int_{s}^{t} p \left(\frac{s}{t-s} (x-s) + s\right)^{p-1} \, dx & \left(\because \frac{s}{t-s} \ge 1\right) \\
&= (t-s) \cdot p s^{p-1} \int_{1}^{2} u^{p-1} \, du & \left(u = \frac{1}{t-s} (x-s) + 1\right) \\
&= (t-s) s^{p-1} (2^{p} - 1) \le 2^{p} s^{p-1} (t-s). \qedhere
\end{aligned}
\]
\end{proof}

By Fact \ref{fact:deviationRealIneq}, the expectations of $|\check{D}_{i+1}^{p} D_{i+1}^{p} - \check{D}_{i}^{p} D_{i}^{p}| \left(1_{\check{D}_{i} \ge D_{i}} 1_{i < \Devi} + 1_{\check{D}_{i} \le D_{i}} 1_{i < \check{\Devi}}\right)$ for $i \ge 0$ are summable as soon as there exists $K''>0$ such that \begin{equation}\label{eqn:deviationMiddle1}
\E\left[ d(o, \check{g}_{i+1})^{n_{1}} d(o, g_{i+1})^{n_{2}} \check{D}_{i}^{p-n_{1}} D_{i}^{p-n_{2}} \left(1_{\check{D}_{i} \ge D_{i} }1_{i<\Devi} + 1_{\check{D}_{i} \le D_{i}} 1_{i < \check{\Devi}}\right)\right] < K''i^{2p+2}e^{-i/K''}
\end{equation} for each $0 \le n_{1}, n_{2} \le p$ with $n_{1} + n_{2} \ge \min(p, 1)$. We discuss the case $n_{2} >0$; the other case $n_{1} >0$ can be handled in the same way.

We will take advantage of the fact that $\E[\check{D}_{i}^{p} D_{i}^{p}]$ is bounded. Namely, the expectation of $\check{D}_{i}^{p-n_{1}} D_{i}^{p-n_{2}}$ on the set $\{D_{i} > c\}$ is small for large $c$. Next, on the set $\{D_{i} \le c\}$, we will bound the expectation of $\check{D}_{i}^{p-n_{1}} D_{i}^{p-n_{2}} 1_{D_{i} < c}1_{i < \Devi}$ by using the exponential bound on $\Prob( i < \Devi)$ (that suppresses $D_{i}^{p-n_{2}} < c^{p-n_{2}}$) independent of the distribution of $\check{D}_{i}$. 

We first discuss the term $\E\left[ d(o, \check{g}_{i+1})^{n_{1}} d(o, g_{i+1})^{n_{2}} \check{D}_{i}^{p-n_{1}} D_{i}^{p-n_{2}}\cdot 1_{\check{D}_{i} \ge D_{i} }1_{i<\Devi} \right]$. 
Let us fix $\check{g}_{i+1}$ and $g_{i+1}$ for the moment, and let $c:= e^{i/2pK'}$. We then have a decomposition \begin{equation}\label{eqn:deviMiddle1stCase}\begin{aligned}
&\E\left.\left[ \check{D}_{i}^{p-n_{1}} D_{i}^{p-n_{2}} 1_{\check{D}_{i} \ge D_{i}} 1_{i<\Devi}\, \right| \, \check{g}_{i+1}, g_{i+1}\right] \\
&= \E\left.\left[ \check{D}_{i}^{p-n_{1}} D_{i}^{p-n_{2}}1_{D_{i} > c} 1_{\check{D}_{i} \ge D_{i}} 1_{i<\Devi} \, \right|\, \check{g}_{i+1}, g_{i+1}\right] + \E\left.\left[ \check{D}_{i}^{p-n_{1}} D_{i}^{p-n_{2}} 1_{D_{i} \le c} 1_{\check{D}_{i} \ge D_{i}} 1_{i<\Devi} \, \right|\, \check{g}_{i+1}, g_{i+1}\right].
\end{aligned}
\end{equation}
The first term is controlled as follows:\[
\begin{aligned}
& \E\left.\left[ \check{D}_{i}^{p-n_{1}} D_{i}^{p-n_{2}}1_{D_{i} > c} 1_{\check{D}_{i} \ge D_{i}} 1_{i<\Devi} \, \right|\, \check{g}_{i+1}, g_{i+1}\right] \\ &\le  \E\left.\left[ \check{D}_{i}^{p-n_{1}} D_{i}^{p-n_{2}}1_{D_{i} > c} \, \right|\, \check{g}_{i+1}, g_{i+1}\right] \\
&\le \E \left[ \check{D}_{i}^{p-n_{1}} D_{i}^{p} \cdot c^{-n_{2}} \right] \le \E[\check{D}_{i}^{p-n_{1}}] \cdot \E[D_{i}^{p}] \cdot c^{-n_{2}}& (\because D_{i}^{-n_{2}} \le c^{-n_{2}} ) \\
&\le i^{p-n_{1} + 1} \E_{\mu} [d(o, go)^{p-n_{1}}] \cdot i^{p+1} \E_{\mu}[d(o, go)^{p}] \cdot c^{-n_{2}}.
\end{aligned}
\]
In the final step, we used the following fact for each $r>0$ and $i > 0$:\begin{equation}\label{eqn:cuteExp}\begin{aligned}
 \E\bigg[ \bigg(\sum_{j=1}^{i} d(o, g_{j} o) \bigg)^{r}\bigg] &\le \E\Big[\Big(i \cdot \max_{1 \le j \le i} d(o, g_{j} o)\Big)^{r}\Big] \le \E \bigg[ i^{r}\cdot  \sum_{j=1}^{i} d(o, g_{j} o)^{r} \bigg] \le i^{r+1} \E_{\mu} [d(o, go)^{r}].
\end{aligned}
\end{equation}

Next, we apply Lemma \ref{lem:Devi} to the second term of the RHS of Equation \ref{eqn:deviMiddle1stCase} and observe:
\[
\begin{aligned}
\E\left.\left[ \check{D}_{i}^{p-n_{1}} D_{i}^{p-n_{2}} 1_{D_{i} \le c} 1_{\check{D}_{i} \ge D_{i}} 1_{i<\Devi} \, \right|\, \check{g}_{i+1}, g_{i+1}\right]&\le \E\Big[ \check{D}_{i}^{p-n_{1}} \cdot \E\Big[ c^{p-n_{2}} 1_{i<\Devi} \, \Big| \, \check{g}_{1}, \ldots, \check{g}_{i+1}, g_{i+1}\Big]\Big]\\
&\le \E\Big[\check{D}_{i}^{p-n_{1}} \cdot c^{p-n_{2}} \Prob\left[\Devi > i \, \big| \, \check{g}_{1}, \ldots, \check{g}_{i+1}, g_{i+1} \right]\Big] \\
&\le i^{p-n_{1} + 1} \E_{\mu} [d(o, go)^{p-n_{1}}] \cdot c^{p-n_{2}} \cdot K' e^{-i/K'}.
\end{aligned}
\]
Here, $c^{p-n_{2}}$ is dominated by $c^{p} = e^{i/2K'}$. Overall, we have \[
\E\left.\left[ \check{D}_{i}^{p-n_{1}} D_{i}^{p-n_{2}} 1_{\check{D}_{i} \ge D_{i}} 1_{i<\Devi} \right|  \check{g}_{i+1}, g_{i+1}\right] \le K' \E_{\mu}[d(o, go)^{p-n_{1}}] (1 + \E_{\mu} [d(o, go)^{p}]) \cdot i^{2p}  \max (e^{-i/2K'}, e^{-n_{2} i/2pK'}).
\]
We now multiply $ d(o, \check{g}_{i+1})^{n_{1}} d(o, g_{i+1})^{n_{2}} $ and integrate. As a result, we observe  \[
\begin{aligned}
&\E\left[ d(o, \check{g}_{i+1})^{n_{1}} d(o, g_{i+1})^{n_{2}} \check{D}_{i}^{p-n_{1}} D_{i}^{p-n_{2}}\cdot 1_{\check{D}_{i} \ge D_{i} }1_{i<\Devi} \right]\\
&= \E \Big[ d(o, g_{i+1} o)^{n_{1}} d(o, \check{g}_{i+1}o)^{n_{2}} \cdot \E\left.\left[ \check{D}_{i}^{p-n_{1}} D_{i}^{p-n_{2}} 1_{\check{D}_{i} \ge D_{i}} 1_{i<\Devi}\, \right| \, \check{g}_{i+1}, g_{i+1}\right]  \Big] \\
&\le \E \Big[ d(o, g_{i+1} o)^{n_{1}} d(o, \check{g}_{i+1}o)^{n_{2}} \cdot\E_{\mu}[d(o, go)^{p-n_{1}}] (1 + \E_{\mu} [d(o, go)^{p}]) \cdot i^{2p} K' e^{-\frac{n_{2}}{2(p+1)K'} i} \Big]\\
&\le C(\mu)\cdot  i^{2p} K' e^{- \frac{n_{2}}{2(p+1)K'}i}
\end{aligned}
\]
for some constant $C(\mu) < +\infty$ determined by the distribution of $\mu$, independent of $i$.
Note that $\mu$ has finite $q$-th moment for every $0 \le q \le p$ thanks to Jensen's inequality.

We similarly deal with the term $\E\left[ d(o, \check{g}_{i+1})^{n_{1}} d(o, g_{i+1})^{n_{2}} \check{D}_{i}^{p-n_{1}} D_{i}^{p-n_{2}}\cdot 1_{\check{D}_{i} \le D_{i} }1_{i<\check{\Devi}} \right]$. Fixing $g_{i+1}$ and $\check{g}_{i+1}$ first, we split the expectation based on the dichotomy for $\check{D}_{i}$:
\[\begin{aligned}
&\E\left.\left[ \check{D}_{i}^{p-n_{1}} D_{i}^{p-n_{2}} 1_{\check{D}_{i} \le D_{i}} 1_{i<\check{\Devi}}\, \right|\, \check{g}_{i+1}, g_{i+1}\right] \\
&= \E\left.\left[ \check{D}_{i}^{p-n_{1}} D_{i}^{p-n_{2}} 1_{\check{D}_{i} > c} 1_{\check{D}_{i} \le D_{i}} 1_{i<\check{\Devi}} \, \right|\, \check{g}_{i+1}, g_{i+1}\right] + \E\left.\left[ \check{D}_{i}^{p-n_{1}} D_{i}^{p-n_{2}} 1_{\check{D}_{i} \le c} 1_{\check{D}_{i} \le D_{i}} 1_{i<\check{\Devi}}\, \right|\, \check{g}_{i+1}, g_{i+1}\right].
\end{aligned}
\]
Here, a crucial observation is that $\check{D}_{i}^{p-n_{1}} D_{i}^{p-n_{2}} 1_{\check{D}_{i} > c} 1_{\check{D}_{i} \le D_{i}} 1_{i<\check{\Devi}}$ is dominated by $\check{D}_{i}^{p-n_{1}} D_{i}^{p-n_{2}} 1_{D_{i} > c} $. The remaining step is analogous to the previous computations:
\[
\begin{aligned}
&\E\left.\left[ \check{D}_{i}^{p-n_{1}} D_{i}^{p-n_{2}} 1_{\check{D}_{i} \le D_{i}} 1_{i<\check{\Devi}}\, \right|\, \check{g}_{i+1}, g_{i+1}\right] \\
&\le \E\left.\left[ \check{D}_{i}^{p-n_{1}} D_{i}^{p-n_{2}} 1_{D_{i} > c}\, \right|\, \check{g}_{i+1}, g_{i+1}\right] + \E\left.\left[ \check{D}_{i}^{p-n_{1}} D_{i}^{p-n_{2}} 1_{\check{D}_{i} \le c} 1_{i<\check{\Devi}}\, \right|\, \check{g}_{i+1}, g_{i+1}\right] \\
&\le \E\left.\left[ \check{D}_{i}^{p-n_{1}} D_{i}^{p} \cdot c^{-n_{2}} \, \right| \, \check{g}_{i+1}, g_{i+1} \right] + \E\left[ D_{i}^{p-n_{2}} \cdot \E\left.\left[ c^{p-n_{1}} 1_{i<\check{\Devi}} \, \right| \, \check{g}_{i+1}, g_{1}, \ldots, g_{i+1}\right]\right]\\
&\le \E[\check{D}_{i}^{p-n_{1}}] \cdot \E[D_{i}^{p}] \cdot c^{-n_{2}} + \E[D_{i}^{p - n_{2}}] \cdot c^{p-n_{1}} \Prob\left[\check{\Devi}> i \, \big| \, \check{g}_{i+1}, g_{1} \ldots, g_{i+1} \right] \\
&\le i^{p-n_{1}+1} \E_{\mu} [d(o, go)^{p-n_{1}}] \cdot i^{p+1} \E_{\mu}[d(o, go)^{p}] \cdot c^{-n_{2}} +i^{p-n_{2} + 1} \E_{\mu} [d(o, go)^{p-n_{2}}] \cdot c^{p-n_{1}} \cdot K'e^{-i/K'}.
\end{aligned}
\]
We then multiply $d(o, \check{g}_{i+1})^{n_{1}} d(o, g_{i+1})^{n_{2}}$ and integrate over $\check{g}_{i+1}$ and $g_{i+1}$ to obtain a summable bound. This concludes the Inequality \ref{eqn:deviationMiddle1}.
\end{proof}

The previous proof also yields the following corollary.

\begin{cor}\label{cor:minDevi}
Let $p>0$ and let $\big((\check{Z}_{n})_{n>0}, (Z_{n})_{n>0}\big)$ be the (bi-directional)  random walk generated by a non-elementary probability measure $\mu$  on $G$ with finite $p$-th moment.  Then there exists $K>0$ such that \[
\E\left[ \min \{ d(o, Z_{\Devi} o), d(o, \check{Z}_{\check{\Devi}} o)\}^{2p} \right] < K.
\]
\end{cor}

\begin{proof}
In view of the  previous proof, it suffices to check \[
\min \{ d(o, Z_{\Devi} o), d(o, \check{Z}_{\check{\Devi}} o)\}^{2p} \le \sum_{i=0}^{\infty} |\check{D}_{i+1}^{p} D_{i+1}^{p} - \check{D}_{i}^{p} D_{i}^{p}| \left(1_{\check{D}_{i} \ge D_{i}} 1_{i < \Devi} + 1_{\check{D}_{i} \le D_{i}} 1_{i < \check{\Devi}}\right). 
\]
The RHS is at least $\check{D}_{l}^{p} D_{l}^{p}$ for $l = \min \{ i : 1_{\check{D}_{i} \ge D_{i}} 1_{i < \Devi} + 1_{\check{D}_{i} \le D_{i}} 1_{i < \check{\Devi}}=0\}$. Note that either $\check{D}_{l} \ge D_{l}$ or $\check{D}_{l} \le D_{l}$ holds. In the first case, we are forced to have $l \ge \Devi$; then \[\begin{aligned}
\min \{ d(o, Z_{\Devi} o), d(o, \check{Z}_{\check{\Devi}} o)\}^{2p} \le d(o, Z_{\Devi} o)^{2p}\le D_{\Devi}^{2p} \le D_{l}^{2p} \le \check{D}_{l}^{p}D_{l}^{p}.
\end{aligned}
\]
In the second case, we are forced to have $l \ge \check{\Devi}$; then \[\begin{aligned}
\min \{ d(o, Z_{\Devi} o), d(o, \check{Z}_{\check{\Devi}} o)\}^{2p} \le d(o, \check{Z}_{\check{\Devi}} o)^{2p} \le \check{D}_{\check{\Devi}}^{2p} \le \check{D}_{l}^{2p} \le \check{D}_{l}^{p}D_{l}^{p}.\qedhere
\end{aligned}
\]
\end{proof}

We now discuss random walks with finite exponential moment.

\begin{cor}\label{cor:minDeviExp}
Let $\big((\check{Z}_{n})_{n>0}, (Z_{n})_{n>0}\big)$ be the (bi-directional)  random walk generated by a non-elementary probability measure $\mu$  on $G$ with finite exponential moment.  Then there exists $K>0$ such that \[
\E\left[ \operatorname{exp} \left( d(o, Z_{\Devi} o) /K\right)\right] < K.
\]
\end{cor}

\begin{proof}
Let $K'$ be as in Lemma \ref{lem:Devi} and $D_{i} = \sum_{k=1}^{i} d(o, g_{k} o)$. Then $e^{d(o, Z_{\Devi} o)/K}$ is dominated by $\sum_{i \le \Devi} e^{D_{i}/K}$. Hence, we need to show that $\E[e^{D_{i}/K} 1_{i < \Devi}]$ is summable. Let $K, c > 0$ and observe
\[\begin{aligned}
\E\left[ e^{ D_{i} / K} 1_{i < \Devi}\right] &\le \E[ e^{D_{i}/K} 1_{D_{i} < c} 1_{i < \Devi}] + \E[ e^{D_{i}/K} 1_{D_{i} \ge c} 1_{i < \Devi}]\\
& \le \E[ e^{c/K} 1_{i< \Devi}] + \E[ e^{2 D_{i} / K} e^{-c/K}] \\
&\le e^{c/K} K' e^{-i/K'} + e^{-c/K} \cdot \left( \E_{\mu}\left[\operatorname{exp} \left(2d(o, go)/K\right)\right] \right)^{i}.
\end{aligned}
\]
By taking $K$ large enough, we can make $\E_{\mu}[ \operatorname{exp}(2d(o, go)/K)] \le e^{1/4K'}$. Then we take $c = iK/2K'$ and conclude $\E[e^{D_{i}/K} 1_{i< \Devi}] < (K' + 1) e^{-i/4K'}$.
\end{proof}

\subsection{Limit theorems} \label{subsection:limitThm}

The second-moment deviation inequality implies the following CLT:

\begin{thm}\label{thm:CLTStrong}
Let $(X, G, o)$ be as in Convention \ref{conv:strong} and let $(Z_{n})_{n>0}$ be the random walk generated by a non-elementary probability measure $\mu$ on $G$ with finite second moment. Then the following limit (called the \emph{asymptotic variance of $\mu$}) exists: \[
\sigma^{2} (\mu) := \lim_{n \rightarrow \infty} \frac{1}{n} Var[d(o, Z_{n} o)],
\]
and the random variable $\frac{1}{\sqrt{n}} [d(o, Z_{n}o) - \lambda(\mu) n]$ converges in law to the Gaussian law $\mathcal{N}(0, \sigma(\mu))$ with zero mean and variance $\sigma^{2}(\mu)$.
\end{thm}

\begin{proof}
Since $\mu$ has finite second moment, Proposition \ref{prop:deviation} implies that $\sup_{n, m \ge 0} (\check{Z}_{m} o, Z_{n} o)_{o}$ has finite 4-th moment, and hence finite second moment. Now Theorem 4.1 and 4.2 of \cite{mathieu2020deviation} lead to the conclusion.
\end{proof}

\begin{remark}\label{rem:asymVar}
In fact, the following non-degeneracy statement holds:
\begin{fact}\label{fact:nonArith}
Let $(X, G, o)$ be as in Convention \ref{conv:strong} and let $(Z_{n})_{n}$ be the random walk generated by a non-elementary probability measure $\mu$ on $G$. Then the asymptotic variance $\sigma^{2}(\mu) := \lim_{n}\frac{1}{n} Var[d(o, Z_{n} o)] $ is nonzero if and only if $\mu$ is \emph{non-arithmetic}, i.e., there exists $N>0$ and two elements $g, h \in (\supp \mu^{\ast N})$ of $\supp \mu^{\ast N}$ with distinct translation lengths.
\end{fact}
The strict positivity of $\sigma^{2}(\mu)$ for non-arithmetic random walks on Gromov hyperbolic spaces and Teichm{\"u}ller space was discussed in \cite{choi2023central}; see Theorem B and Claim 6.2 of \cite{choi2023central}. Since the argument in \cite{choi2023central} also applies to the general case, we omit the proof here.
\end{remark}

We next discuss the law of the iterated logarithms.

\begin{thm}\label{thm:LILStrong}
Let $(X, G, o)$ be as in Convention \ref{conv:strong} and let $(Z_{n})_{n>0}$ be the random walk generated by a non-elementary probability measure $\mu$ on $G$ with finite second moment. Then for almost every sample path $(Z_{n})_{n}$ we have \[
\limsup_{n \rightarrow \infty} \frac{d(o, Z_{n} o) - \lambda(\mu) n}{\sqrt{2n \log \log n}} = \sigma(\mu),
\] 
where $\lambda(\mu)$ is the drift of $\mu$ and $\sigma^{2}(\mu)$ is the asymptotic variance of $\mu$.
\end{thm}
We proved the LIL based on the uniform $4$th order deviation inequality in \cite{choi2023central}. We give another argument because we will only have second-order deviation inequality in Part \ref{part:weak}.

\begin{proof}
In the proof of the LIL in \cite{choi2023central} (see \cite[Claim 7.1]{choi2023central}), the author proved:
\begin{lem}\label{lem:deAcosta}
Let $K>0$ and let $\{U_{k, i}\}_{i, k \in \Z_{>0}}$ be RVs such that for each $k$, $\{U_{k, i}\}_{i}$ are i.i.d.s with zero mean and variance at most $K$. Then for each $\epsilon >0$, there exists $M>0$ such that \[
\Prob \left( \limsup_{n} \frac{1}{\sqrt{2n \log \log n}} \Bigg| \sum_{k = M}^{\lfloor \log_{2} n \rfloor} \sum_{i=1}^{\lfloor n/2^{k+1} \rfloor} U_{k, i}  \Bigg| > \epsilon \right) \le \epsilon.
\]
\end{lem}
We now set \[\begin{aligned}
Y_{k, i} := d(Z_{2^{k} (i-1)} o, Z_{2^{k} i} o), \quad b_{k, i} := (Z_{2^{k} (2i-2)} o, Z_{2^{k} \cdot 2i}o)_{Z_{2^{k} (2i-1)}o}.
\end{aligned}
\]
Equivalently, we have $Y_{k+1, i} = Y_{k, 2i-1} + Y_{k, 2i} - 2b_{k, i}$. Note that $\{b_{k, i}\}_{k, i}$ have uniformly bounded variance by Lemma \ref{lem:Devi} and $\{b_{k, i} - \E[b_{k, i}]\}_{i}$ are i.i.d.s with zero mean for each $k$. We also set \[
b_{k; n} = \left\{ \begin{array}{cc} (Z_{2^{k+1} \lfloor n/2^{k+1} \rfloor} o, Z_{n} o)_{Z_{2^{k}(2 \lfloor n/2^{k+1} \rfloor +1)} o}  & \textrm{if $2^{k}(2 \lfloor n/2^{k+1} \rfloor +1)<n$} \\ 
0 & \textrm{otherwise}. \end{array} \right.
\]
We then observe the decomposition
 \begin{equation}\label{eqn:LILInt1}
d(o, Z_{n} o) =\sum_{i=1}^{\lfloor n/2^{M} \rfloor} Y_{M, i} + d\big( Z_{2^{M} \lfloor n/2^{M} \rfloor} o, Z_{n} o \big) - 2\sum_{k = M}^{\lfloor \log_{2} n \rfloor}\bigg( b_{k; n} +  \sum_{i=1}^{\lfloor n/2^{k+1} \rfloor} b_{k, i} \bigg). \quad( \forall n, M > 0)
\end{equation}
Indeed, the RHS is unchanged when $M$ increases by 1 and is equal to $d(o, Z_{n}o)$ at $M > \lfloor \log_{2} n  \rfloor $.

Now, fixing an $\epsilon > 0$, we take $M>0$ for $\{b_{k, i} - \E[b_{k, i}]\}_{i, k}$ using Lemma \ref{lem:deAcosta}. We balance each term in Display \ref{eqn:LILInt1} by subtracting its expectation, normalize with the denominator $\sqrt{2n \log \log n}$ and then examine the almost sure limit supremum. The classical LIL tells us that \[
\limsup_{n\rightarrow \infty} \frac{1}{\sqrt{2n \log \log n}} \sum_{i=1}^{ \lfloor n/2^{M} \rfloor} (Y_{M, i} - \E[Y_{M, i}]) = \frac{1}{\sqrt{2^{M}}} \sqrt{Var (Y_{M, 1})}.
\]

Regarding the second term, note that $d(Z_{2^{M} \lfloor n/2^{M} \rfloor} , Z_{n} o)$ is dominated by the sum of at most $2^{M}$ independent steps distributed according to $\mu$. This implies that \[
\Prob\big( d(Z_{2^{M} \lfloor n/2^{M} \rfloor} , Z_{n} o) > \epsilon \sqrt{n}\big) \le \Prob\bigg( \sum_{i=1}^{2^{M}} d(o, g_{i} o) > \epsilon \sqrt{n}\bigg),
\]
and RHS is summable in $n$ because $\mu$ has finite second moment. By Borel-Cantelli lemma, \[
\frac{1}{\sqrt{2n \log \log n}}  \big| d( Z_{2^{M} \lfloor n/2^{M} \rfloor} o, Z_{n} o) \big| = 0 \quad \textrm{almost surely.}
\]
Next, 
Lemma \ref{lem:deAcosta} implies that the term \[
\frac{1}{\sqrt{2n \log \log n}}\sum_{k=M}^{\lfloor\log_{2} n\rfloor} \sum_{i=1}^{\lfloor n/2^{k+1} \rfloor} (b_{k, i} - \E[b_{k, i}])
\] eventually falls into the interval $[-\epsilon, +\epsilon]$ outside a set of probability $\epsilon$.

It remains to deal with $\frac{1}{\sqrt{2n \log\log n}}\sum_{k=M}^{\lfloor \log_{2}n \rfloor} (b_{k; n} - \E[b_{k; n}])$. Let \[
b_{j} := \sup_{i, i' \ge 0} (Z_{j- i}o , Z_{j + i'} o)_{o}.
\]Then for each $k$ and $n$, $0 \le b_{k; n} \le b_{2^{k} (2 \lfloor n/ 2^{k+1} \rfloor + 1)}$ holds. Moreover, $b_{j}$'s are identically distributed with finite variance (and hence finite expectation). This implies that \[
0\le \frac{1}{\sqrt{2n \log \log n}} \sum_{k=1}^{\lfloor \log_{2} n \rfloor}\E[b_{k;n}]\le \frac{\log_{2} n}{\sqrt{2n \log \log n}} \E[b_{0}]
\]tends to 0 as $n$ goes to infinity.

We now estimate the summation \[
\sum_{k=0}^{\infty} \sum_{i=1}^{\infty} \Prob\left( b_{2^{k} (2i-1)}^{2} \ge \epsilon^{2}  \sqrt{2}^{k} (2i-2)  \right).
\]
To estimate this, for each $y>0$ let us count the number of pairs $(i, k) \in \Z_{\ge 0}^{2}$ such that $ \epsilon^{2}  \sqrt{2}^{k} (2i-2) < y$. For each $k \in \Z_{\ge 0}$, there exist at most $y / (\sqrt{2}^{k} \epsilon^{2})$ candidates for $i$. Summing them up, there are at most $C_{\epsilon} y$ such pairs $(i, k)$, where $C_{\epsilon}>0$ is a constant. This implies that \[\begin{aligned}
\sum_{k=0}^{\infty} \sum_{i=1}^{\infty} \Prob\left( b_{2^{k} (2i-2)}^{2} \ge \epsilon^{2}  \sqrt{2}^{k} (2i-2)  \right) &\le \sum_{k, i} \sum_{y=0}^{\infty}  \Prob\big(y-1 \le b_{2^{k}(2i-2)}^{2} < y\big) 1_{y \ge \epsilon^{2} \sqrt{2}^{k} (2i-2)}\\
&= \sum_{k, i} \sum_{y=0}^{\infty}  \Prob\big(y-1 \le b_{1}^{2} < y\big) 1_{y \ge \epsilon^{2} \sqrt{2}^{k} (2i-2)} \\
&\le \sum_{y=0}^{\infty}  \Prob\big(y-1 \le b_{1}^{2} < y\big) \cdot  \#\{(i, k) : y \ge \epsilon^{2} \sqrt{2}^{k} (2i-2)\} \\
&\le \sum_{y=0}^{\infty} \Prob\big(y-1 \le b_{1}^{2} < y\big) \cdot C_{\epsilon} y \le   \E[C_{\epsilon} b_{1}^{2}] < +\infty.
\end{aligned}
\]
By Borell-Cantelli, for almost every sample path $b_{2^{k}(2i-1)} < \epsilon \cdot 2^{k/4} \sqrt{2i-2}$ holds for all but finitely many $(i,k)$'s. In particular, for sufficiently large $n$, we have \[
b_{k; n} \le b_{2^{k} (2\lfloor n/2^{k+1} \rfloor + 1)} \le \epsilon \cdot 2^{k/4}\sqrt{2n/2^{k+1}} =\epsilon \sqrt{n} /2^{k/4} 
\]
for each $k = 1, \ldots, \lfloor \log_{2} n \rfloor$. Hence, we have \[
\frac{1}{\sqrt{2n \log \log n}}\sum_{k = M}^{\lfloor \log_{2} n \rfloor} b_{k; n} \le \epsilon \sum_{k=1}^{\infty} 1/2^{k/4} \le 10 \epsilon.
\]
Combining these estimates with Equation \ref{eqn:LILInt1}, we observe that for probability at least $1-\epsilon$, \[
\limsup_{n \rightarrow \infty} \frac{d(o, Z_{n} o) - \E[d(o, Z_{n} o)]}{\sqrt{2n \log \log n}} \in \left[\sqrt{\frac{Var[ d(o, Z_{2^{M}}o)]}{2^{M}}} - 20\epsilon,\sqrt{\frac{Var [d(o, Z_{2^{M}} o)]}{2^{M}}}+20\epsilon\right].
\]
By decreasing $\epsilon$ while increasing $M$, we arrive at the desired conclusion.
\end{proof}

We finally prove the geodesic tracking by random walks.

\begin{thm}\label{thm:geodTracking}
Let $(X, G, o)$ be as in Convention \ref{conv:strong}, let $p>0$ and let $(Z_{n})_{n}$ be the random walk generated by a non-elementary probability measure $\mu$ on $G$ with finite $p$-th moment. Then there exists $K>0$ such that, for almost every sample path $(Z_{n})_{n \ge 0}$, there exists a $K$-quasigeodesic $\gamma$ on $X$ satisfying \[
\lim_{n \rightarrow \infty} \frac{1}{n^{1/2p}} d(Z_{n} o, \gamma) = 0.
\]
\end{thm}

\begin{proof}
Recall Definition \ref{dfn:SchottkyLong}: Given $K_{0}>0$, we have defined:\begin{itemize}
\item $D_{0}= D(K_{0}, K_{0}) > K_{0}$ be as in Lemma \ref{lem:1segment}, 
\item $E_{0}= E(K_{0}, D_{0}) > D_{0}$, $L_{0} = L(K_{0}, D_{0})$ be as in Proposition \ref{prop:BGIPWitness}.
\end{itemize}
In addition to these, we define:\begin{itemize}
\item $E_{1}= E(K_{0}, D_{0})$, $L_{1} = L(K_{0}, D_{0})$ be as in Lemma \ref{lem:BGIPQuasi};
\item $E_{2}= E(K_{0}, E_{0} + 5K_{0})$, $L_{2} = L(K_{0}, E_{0}+5K_{0})$ be as in Proposition \ref{prop:BGIPWitness}.
\end{itemize}
Since $\mu$ is non-elementary, Proposition \ref{prop:Schottky} guarantees that there exist $K_{0}>0$, $M_{0}>L_{0} + L_{1}+L_{2}$ and a large enough $K_{0}$-Schottky set $S \subseteq (\supp \mu)^{M_{0}}$. We fix this $S$ from now on.

By Proposition \ref{prop:gouezelRW1}, there exists a probability space $(\Omega, \Prob)$ with RV $\diffPivot(\w) = \{j(1) < j(2) < \ldots \} \subseteq M_{0} \Z_{>0}$, the set of pivotal times, such that $(o, \axes_{j(1)}(\w), \axes_{j(2)}(\w), \ldots)$ is $D_{0}$-semi-aligned.

We now define $\Gamma(\w)$ as the concatenation of $[o, Z_{j(1) - M_{0}} o]$, $[Z_{j(1) - M_{0}} o, Z_{j(1)} o]$, $[Z_{j(1)} o, Z_{j(2) - M_{0}} o]$, $[Z_{j(2) - M_{0}} o, Z_{j(2)} o]$, $\ldots$. By Lemma \ref{lem:BGIPQuasi}, $\Gamma(\w)$ is an $E_{1}$-quasigeodesic for almost every $\w \in \Omega$. It remains to prove $\lim_{n} d(Z_{n}(\w) o, \Gamma(\w)) / n^{1/2p}=0$ almost everywhere.

By Corollary \ref{cor:minDevi}, $\min [ d(o, Z_{\Devi} o), d(o, \check{Z}_{\check{\Devi}} o)]^{2p}$ is dominated by an integrable RV. This implies  \begin{equation}\label{eqn:summableSum}
\sum_{k} \Prob\left(\min\big(d(o, Z_{\Devi} o), d(o, \check{Z}_{\check{\Devi}}o)\big)> g(k)\right) < +\infty
\end{equation}
for some function $g$ such that $\lim_{k} g(k)/k^{1/2p} = 0$. Also, Lemma \ref{lem:Devi} tells us that \begin{equation}\label{eqn:summableDeviEst}
\sum_{k} \Prob \big( \max(\Devi, \check{\Devi}) \ge k - M_{0} \big) < +\infty.
\end{equation}
Now, for each $k \in \Z_{>0}$, we consider the following sets: \[
A_{k} := \left\{ (\check{\w}, \w) : \begin{array}{c}\textrm{there exists $M_{0}\le i \le k-M_{0}$ such that $d(o, Z_{i} o) \le g(k)$ and} \\  \textrm{$(\check{Z}_{k} o, (Z_{i-M_{0}} o, \ldots, Z_{i} o), Z_{n} o)$ is $D_{0}$-semi-aligned for all $n \ge k$ }\end{array}\right\}.
\]
 \[
B_{k} := \left\{ (\check{\w}, \w) : \begin{array}{c}\textrm{there exists $M_{0}\le i \le k-M_{0}$ such that $d(o, \check{Z}_{i} o) \le g(k)$ and} \\  \textrm{$(\check{Z}_{k} o, (\check{Z}_{i} o, \ldots, \check{Z}_{i-M_{0}} o), Z_{n} o)$ is $D_{0}$-semi-aligned for all $n \ge k$ }\end{array}\right\}.
\]
Then the definition of the RV $\Devi(\check{\w}, \w)$ and $\check{\Devi}(\check{\w}, w)$, together with Inequality \ref{eqn:deviDomDist}, tells us that \[
 A_{k}^{c} \cap B_{k}^{c} \subseteq \Big\{ (\check{\w}, \w) : \min \big(d(o, Z_{\Devi} o),d(o, \check{Z}_{\check{\Devi}}o)\big)> g(k) \,\,\textrm{or} \,\,\max(\Devi, \check{\Devi}) \ge k - M_{0}\Big\} .
\]
Thanks to Display \ref{eqn:summableSum} and \ref{eqn:summableDeviEst}, we observe that $\Prob(A_{k}^{c} \cap B_{k}^{c})$ is also summable.

Finally, consider \[
C_{k} := \left\{ (\check{\w}, \w) : \begin{array}{c}\textrm{there exists $M_{0}\le i \le 2k-M_{0}$ such that $d(Z_{k}, Z_{i} o) \le g(k)$ and} \\  \textrm{$(o, (Z_{i - M_{0}} o, \ldots, Z_{i} o), Z_{n} o)$ is $D_{0}$-semi-aligned for all $n \ge 2k$ }\end{array}\right\}.
\]
Then $C_{k}$ contains $T^{k}(A_{k} \cup B_{k})$, where $T$ is the Bernoulli shift operator on the bi-infinite sample paths, which is measure preserving. Hence,  $\Prob(C_{k}^{c}) \le \Prob(A_{k}^{c} \cap B_{k}^{c})$ is summable. The Borel-Cantelli lemma implies that, for almost every sample path,  for each sufficiently large $k$ there exists $M_{0} \le j'(k) \le 2k$ such that $\diam(Z_{k} o \cup \axes_{j'(k)} o) \le d(Z_{k} o, Z_{j'(k)} o) + \diam (\axes_{j'(k)} o) \le  g(k) + K_{0}M_{0} + K_{0}$ and such that  $(o, \axes_{j'(k)}, Z_{n} o)$ is $D_{0}$-semi-aligned for $n \ge 2k$ ($\ast$),

Let us now pick a sample path satisfying $(\ast$), pick a sufficiently large $k$, and let $N$ be an index such that $j(N) \ge 2k$. Recall that $(o, \axes_{j(1)}(\w), \axes_{j(2)}(\w), \ldots)$ is $D_{0}$-semi-aligned. By Proposition \ref{prop:BGIPWitness}, $[o, Z_{j(N)} o]$ have subsegments $[x_{1}, y_{1}], \ldots, [x_{N}, y_{N}]$, in order from left to right, such that $[x_{i}, y_{i}]$ and $\axes_{j(i)}$ are $0.1E_{0}$-fellow traveling for $i = 1, \ldots, N$. Moreover, by Corollary \ref{cor:BGIPSelfNbd}, $\axes_{j(i)}$ and $[Z_{j(i) - M_{0}} o, Z_{j(i)} o]$ are $0.1E_{0}$-fellow traveling for $i=1, \ldots, N$. Finally, since $(o, \axes_{j'(k)}, Z_{j(N)} o)$ is $D_{0}$-semi-aligned, Proposition \ref{prop:BGIPWitness} tells us that $[o, Z_{j(N)} o]$ also contains a subsegment $[q_{1}, q_{2}]$ that $0.1E_{0}$-fellow travels with $\axes_{j'(k)}$. For convenience we let $y_{0} = o$ and $j(0) = 0$.

If $[q_{1}, q_{2}]$ overlaps with some $[x_{i}, y_{i}]$, this implies $d(\axes_{j'(k)}, [Z_{j(i) - M_{0}} o, Z_{j(i)} o]) \le E_{0}$ and hence $d(Z_{k} o, \Gamma(\w) )\le g(k) + E_{0} + K_{0} M_{0} + K_{0}$. If not, then $[q_{1}, q_{2}]$ is a subsegment of $[y_{i-1}, x_{i}]$ for some $i$. Lemma \ref{lem:fellowTravelAlign} then tells us that $(y_{i-1}, \axes_{j'(k)}, x_{i})$ is $0.4E_{0}$-aligned. Since $d(y_{i-1}, Z_{j(i-1)} o) \le 0.1E_{0}$ and $d(x_{i}, Z_{j(i) - M_{0}} o) \le 0.1E_{0}$, Lemma \ref{lem:projBdd} implies that $(Z_{j(i-1)} o, \axes_{j'(k)}, Z_{j(i)- M_{0}} o)$ is $(0.5E_{0} + 4K_{0})$-aligned. By Proposition \ref{prop:BGIPWitness}, $[Z_{j(i-1)} o, Z_{j(i) - M_{0}} o]$ passes through the $E_{2}$-neighborhood of $\axes_{j'(k)}$, and $d(Z_{k} o, \Gamma(\w)) \le g(k) + E_{2}$.

In summary, almost every sample path $(\check{\w}, \w)$ satisfies $(\ast)$, which leads to $d(Z_{k} o, \Gamma(\w)) \le g(k)+E_{0}+E_{2} + K_{0}M_{0} + K_{0}= o(k^{1/2p})$ eventually. This ends the proof.\end{proof}

Recall Corollary \ref{cor:minDeviExp}: if $\mu$ has finite exponential moment, then $\E[\operatorname{exp}(d(o, Z_{\Devi} o) / K)]$ is finite, i.e., $\Prob( d(o, Z_{\Devi} o) > K \log k)$ is summable for some $K>0$. By replacing $g(k)$ in the previous proof with $K \log k$, we obtain:

\begin{thm}\label{thm:geodTrackingExp}
Let $(X, G, o)$ be as in Convention \ref{conv:strong} and let $(Z_{n})_{n}$ be the random walk generated by a non-elementary probability measure $\mu$ on $G$ with finite exponential moment. Then there exists $K>0$ such that, for almost every sample path $(Z_{n})_{n \ge 0}$, there exists a $K$-quasigeodesic $\gamma$ satisfying \[
\limsup_{n \rightarrow \infty} \frac{1}{\log n} d(Z_{n} o, \gamma) \le K.
\]
\end{thm}

\section{Pivotal time construction} \label{section:pivotConst}

In this section we prove Proposition \ref{prop:gouezelRW1} by generalizing Gou{\"e}zel's theory in \cite[Section 4A]{gouezel2022exponential} to the setting of Convention \ref{conv:strong}. We first construct and study pivotal times in a discrete model and then realize them on random walks. This strategy is also employed for LDP in Section \ref{section:LDP}.

\subsection{Pivotal times: discrete model}\label{subsection:pivotDiscrete}
Throughout the subsection, we fix a long enough $K_{0}$-Schottky set $S$ with cardinality $N_{0}$. Given sequences of isometries $\mathbf{w} =(w_{i})_{i=0}^{\infty}$ and $\mathbf{v} =(v_{i})_{i=1}^{\infty}$ in $G$,  we draw a sequence of Schottky sequences \[\begin{aligned}
\mathbf{s} &= (\alpha_{1}, \beta_{1}, \gamma_{1}, \delta_{1}, \ldots, \alpha_{n}, \beta_{n}, \gamma_{n}, \delta_{n}) \in S^{4n},
\end{aligned}
\] 
with respect to the uniform measure on $S^{4n}$. We define isometries \begin{equation}\label{eqn:abcd}
a_{i} := \prodSeq(\alpha_{i}), \,\,b_{i} :=\prodSeq(\beta_{i}), \,\,c_{i} := \prodSeq(\gamma_{i}), \,\,d_{i} := \prodSeq(\delta_{i}),
\end{equation}
and study the word \[
w_{0} a_{1}b_{1}v_{1}c_{1}d_{1}w_{1} \cdots a_{k}b_{k}v_{k}c_{k}d_{k}w_{k} \cdots.
\]
With the base case $w_{0, 2}^{+} := id$, we define its subwords for $i>0$:\[\begin{array}{lll}
w_{i, 2}^{-} := w_{i-1, 2}^{+} w_{i-1}, & w_{i, 1}^{-} := w_{i, 2}^{-} a_{i}, & w_{i, 0}^{-} := w_{i, 2}^{-} a_{i}b_{i}, \\[5pt]
 w_{i, 0}^{+} := w_{i, 2}^{-} a_{i} b_{i} v_{i}, &  w_{i, 1}^{+} := w_{i, 2}^{-} a_{i} b_{i} v_{i} c_{i}, & w_{i, 2}^{+} := w_{i, 2}^{-} a_{i}b_{i}v_{i}c_{i}d_{i}.
\end{array}
\]
Let us also employ the notations \[\begin{aligned}
\varGam(\alpha_{i}) &:= w_{i, 2}^{-} \Gamma^{+}(\alpha_{i}), & \varGam(\beta_{i}) &:= w_{i, 1}^{-} \Gamma^{+}(\beta_{i}),\\
\varGam(\gamma_{i}) &:= w_{i, 0}^{+} \Gamma^{+}(\gamma_{i}), & \varGam(\delta_{i}) &:= w_{i, 1}^{+} \Gamma^{+}(\delta_{i}).\\
\end{aligned}
\]

We define the \emph{set of pivotal times} $P_{n} = P_{n}\left(\mathbf{s}; \mathbf{w}, \mathbf{v}\right)$ and an auxiliary moving point $z_{n} = z_{n}\left(\mathbf{s}; \mathbf{w}, \mathbf{v}\right)$ inductively. Let $P_{0} = \emptyset$ and $z_{0} = o$ as the base case. Given $P_{n-1} \subseteq \{1, \ldots, n-1\}$ and $z_{n-1} \in X$, the data $P_{n}$ and $z_{n}$ at step $n$ are determined by the following criteria. 
\begin{enumerate}[label=(\Alph*)]
\item When $\left(z_{n-1}, \varGam(\alpha_{n})\right)$, $\left(\varGam(\beta_{n}), w_{n, 1}^{+} o\right)$, $\left(w_{n, 0}^{-}o, \varGam(\gamma_{n})\right)$ and $\left(\varGam(\delta_{n}), w_{n+1, 2}^{-}o\right)$ are $K_{0}$-aligned, we set $P_{n} = P_{n-1} \cup \{n\}$ and $z_{n} = w_{n, 1}^{+}o$ (see Figure \ref{fig:0thCasePivot}).
\item Otherwise, we seek $i \in P_{n-1}$ and an integer $j \in \{i+1, \ldots, n-1\}$ such that $\big(\varGam(\delta_{i}), \varGam(\beta_{j})\big)$ is $D_{0}$-semi-aligned and such that  $\big(\varGam(\beta_{j}), w_{n+1, 2}^{-} o\big)$ is $K_{0}$-aligned. 

If such a pair $(i, j)$ exists, we pick the lexicographically maximal one and let $P_{n} := P_{n-1} \cap \{1, \ldots, i\}$ and $z_{n} = w_{j, 1}^{-}o$. If such a pair does not exist, then we let $P_{n} := \emptyset$ and $z_{n} := o$.
\end{enumerate}

We note that the set the set $P_{n}$ depends solely on $(w_{i})_{i=0}^{n}$, $(v_{i})_{i=1}^{n}$ and $(\alpha_{i}, \beta_{i}, \gamma_{i}, \delta_{i})_{i=1}^{n}$; it is independent from $\{w_{i}, v_{i}, \alpha_{i}, \beta_{i}, \gamma_{i}, \delta_{i} : i > n\}$.

$P_{n}$ records the Schottky axes aligned along $[o, w_{n+1, 2}^{-} o]$. More precisely:

\begin{prop}\label{prop:extremal}
Let $P_{n}= \{i(1) < \ldots < i(m)\}$. Then \[
\left(o, \varGam(\alpha_{i(1)}), \varGam(\beta_{i(1)}), \varGam(\gamma_{i(1)}), \varGam(\delta_{i(1)}), \ldots, \varGam(\alpha_{i(m)}), \varGam(\beta_{i(m)}), \varGam(\gamma_{i(m)}), \varGam(\delta_{i(m)}), w_{n+1, 2}^{-}o\right)
\]
is $D_{0}$-semi-aligned.
\end{prop}
On Gromov hyperbolic spaces, this corresponds to \cite[Lemma 5.3]{gouezel2022exponential}. Before proving the entire statement, let us prove two small parts of it.

\begin{lem}\label{lem:pivotJtn}
For any $\mathbf{s} \in S^{4n}$ and $1 \le i \le n$, $\left(\varGam(\alpha_{i}), \varGam(\beta_{i})\right)$ and $\left(\varGam(\gamma_{i}), \varGam(\delta_{i})\right)$ are $D_{0}$-aligned.
\end{lem}

\begin{proof}
Let us prove that $\left(\varGam(\alpha_{i}), \varGam(\beta_{i})\right) = \left( w_{i, 1}^{-} \bar{\Gamma}^{-} (\alpha_{i}), w_{i, 1}^{-} \Gamma^{+}(\beta_{i}) \right)$ is $D_{0}$-aligned, or equivalently, that $\left(\bar{\Gamma}^{-}(\alpha_{i}), \Gamma^{+}(\beta_{i})\right)$ is $D_{0}$-aligned. When $\alpha_{i} = \beta_{i}$, this is guaranteed by the definition of $K_{0}$-Schottky sets.

Now suppose $\alpha_{i} \neq \beta_{i}$. First, $(\bar{\Gamma}^{-}(\alpha_{i}), o)$ is $0$-aligned. Second, $\left(a_{i}^{-1} o, \Gamma^{-}(\alpha_{i})\right)$ is not $K_{0}$-aligned as $d(o, a_{i}^{-1} o) \ge 100E_{0} \ge K_{0}$. Then by the Schottky property of $S$, $\left(a_{i}^{-1} o, \Gamma^{+}(\beta_{i})\right)$ is $K_{0}$-aligned. Now Lemma \ref{lem:1segment} tells us that $\left(\bar{\Gamma}^{-}(\alpha_{i}), \Gamma^{+}(\beta_{i})\right)$ is $D_{0}$-aligned.

The alignment of $\left(\varGam(\gamma_{i}), \varGam(\delta_{i})\right)$ holds for the same reason.
\end{proof}

\begin{lem}[{\cite[Lemma 3.2]{choi2024pseudo-anosovs}}]\label{lem:intermediate}
Let $k \in \Z_{>0}$. Let $l< m$ be consecutive elements in $P_{k}$, i.e., $m \in P_{k}$ and $l = \max (P_{k} \cap \{1, \ldots, m-1\})$. Then $\left(\varGam(\delta_{l}), \varGam(\alpha_{m}) \right)$ is $D_{0}$-semi-aligned.
\end{lem}

\begin{proof}
$l, m \in P_{k}$ implies that $l \in P_{l}$ and $l, m \in P_{m}$. In particular, $l$ and $m$ are newly chosen at step $l$ and $m$, respectively, by fulfilling Criterion (A). Hence, $(\varGam(\delta_{l}), w_{l+1, 2}^{-}o)$ and $(z_{m-1}, \varGam(\alpha_{m}))$ are $K_{0}$-aligned ($\ast$), and $z_{l} = w_{l, 1}^{+}o$. Moreover, we have $P_{m} = P_{m-1} \cup \{m\}$ and $l = \max P_{m-1}$. 

If $l = m-1$ and $m$ was newly chosen at step $m = l+1$, then $z_{m-1} = z_{l} = w_{l,1}^{+}o$ holds. Lemma \ref{lem:1segment} and ($\ast$) imply that $\left(\varGam(\delta_{l}), \varGam(\alpha_{m})\right)$ is $D_{0}$-aligned.

If $l < m-1$, then $l=\max P_{m-1}$ has survived at step $m-1$ by fulfilling Criterion (B); there exist $j > l$ such that $\big( \varGam(\delta_{l}), \varGam(\beta_{j})\big)$ is $D_{0}$-semi-aligned and $\big(\varGam(\beta_{j}), w_{n+1, 2}^{-} o\big)$ is $K_{0}$-aligned. Furthermore, $z_{m-1}$ equals $w_{j, 1}^{-} o$, the beginning point of $\varGam(\beta_{j})$. 

Note that $(z_{m-1}, \varGam(\alpha_{m})\big)$ is $K_{0}$-aligned by ($\ast$).  Lemma \ref{lem:1segment} then asserts that $\left(\varGam(\beta_{j}), \varGam(\alpha_{m})\right)$ is $D_{0}$-aligned. Concatenating the two $D_{0}$-semi-aligned sequences, we conclude that $\left(\varGam(\delta_{l}), \varGam(\alpha_{m})\right)$ is $D_{0}$-semi-aligned.
\end{proof}

\begin{proof}[Proof of Proposition \ref{prop:extremal}]
Having established Lemma \ref{lem:intermediate}, it remains to prove that:
\begin{itemize}
\item $\left(o, \varGam(\alpha_{i(1)})\right)$ is $K_{0}$-aligned;
\item for $1 \le t \le m$, $\left(\varGam(\alpha_{i(t)}), \varGam(\beta_{i(t)}), \varGam(\gamma_{i(t)}), \varGam(\delta_{i(t)})\right)$ is $D_{0}$-aligned;
\item $\left(\varGam(\delta_{i(m)}), w_{n+1, 2}^{-}o\right)$ is $D_{0}$-semi-aligned.
%there exist finitely many Schottky axes $\varGam(\delta_{i(m)}) = \varGam_{1}, \ldots, \varGam_{M}$ such that $\left(\varGam_{1}, \ldots, \varGam_{M}, y_{n+1, 2}^{-}\right)$ is $D_{0}$-aligned.
\end{itemize}

Note that for each $t =1, \ldots, m$, $i(t)$ is newly chosen as a pivotal time at step $i(t)$ by fulfilling Criterion (A). In particular, we have that: \begin{itemize}
\item $\left(\varGam(\alpha_{i(t)}), \varGam(\beta_{i(t)})\right)$ is $D_{0}$-aligned (Lemma \ref{lem:pivotJtn});
\item $\left(\varGam(\beta_{i(t)}), \varGam(\gamma_{i(t)})\right)$ is $D_{0}$-aligned since $\left(\varGam(\beta_{i(t)}), w_{n, 1}^{+}o\right)$ and $\left(w_{i(t), 0}^{-}o, \varGam(\gamma_{i(t)})\right)$ are $K_{0}$-aligned (Lemma \ref{lem:1segment}), and
\item $\left(\varGam(\gamma_{i(t)}), \varGam(\delta_{i(t)})\right)$ is $D_{0}$-aligned (Lemma \ref{lem:pivotJtn}).
\end{itemize}
This guarantees the second item.

We also note that $P_{i(1)-1} = \emptyset$. Indeed, any $j$ in $P_{i(1) - 1}$ is smaller than $i(1)$ and would have survived in $P_{i(1)}$ (since what happened at step $i(1)$ was adding an element, not deleting some). Since $i(1)$ was not deleted at any later step, such $j$ would also not be deleted till the end and should have appeared in $P_{n}$. Since $i(1)$ is the earliest pivotal time in $P_{n}$, no such $j$ exists. Hence, $z_{i(1) -1} = o$ and Criterion (A) for $i(1)$ leads to the first item.

We now observe how $i(m)$ survived in $P_{n}$. If $i(m) = n$, then it was newly chosen at step $n$ by fulfilling Criterion (A). In particular, $(\varGam(\delta_{n}), w_{n+1, 2}^{-}o)$ is $K_{0}$-aligned as desired.

If $i(m) \neq n$, then it has survived at step $n$ as the last pivotal time by fulfilling Criterion (B). In particular, there exist $j > i(m)$ such that $\big(\varGam(\delta_{i(m)}), \varGam(\beta_{j}) \big)$ is $D_{0}$-semi-aligned and such that $(\varGam(\beta_{j}), w_{n+1, 2}^{-} o)$ is $K_{0}$-aligned. In particular,  $\big(\varGam(\delta_{i(m)}), \varGam(\beta_{j}) , w_{n+1, 2}^{-} o\big)$ is $D_{0}$-semi-aligned.
\end{proof}

Next, we study when $P_{n} = P_{n-1} \cup \{n\}$ happens, i.e., a new pivotal time is added to the set of pivotal times. This will guide us how to pivot the direction at a pivotal time without affecting the set of pivotal times. Recall that we draw $\alpha_{i}, \beta_{i}, \gamma_{i}, \delta_{i}$'s from $S$ with the uniform measure.

\begin{lem}\label{lem:0thCasePivot}
Let us fix $\mathbf{w} = (w_{i})_{i}$, $\mathbf{v} = (v_{i})_{i}$ and $\mathbf{s} \in S^{4(n-1)}$. Then  \[
\Prob\Big( \# P_{n}(\mathbf{s}, \alpha_{n}, \beta_{n}, \gamma_{n}, \delta_{n}) = \# P_{n-1}(\mathbf{s}) + 1\Big)\ge 1-4/N_{0}.
\]
\end{lem}

\begin{proof}
Recall Criterion (A) for $\#P_{n} = \# P_{n-1} + 1$. We will investigate the four required conditions one-by-one.

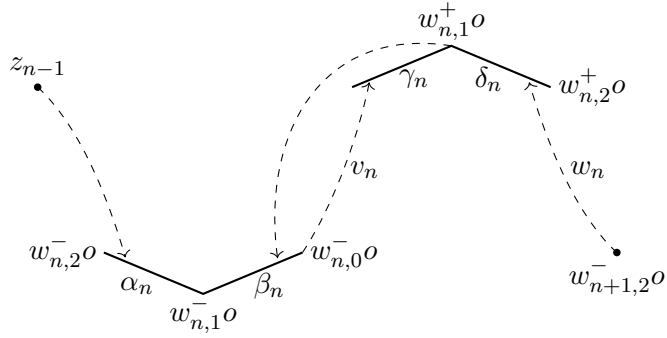
\begin{figure}
\centering
\begin{tikzpicture}
\def\c{1.1}
\fill (0, 0) circle (0.05);
\draw[thick] (0.8*\c, -2*\c) -- (2*\c, -2.5*\c) -- (3.2*\c, -2*\c);
\draw[thick] (3.8*\c, 0) -- (5*\c, 0.5*\c) -- (6.2*\c, 0);
\fill (7*\c, -2*\c) circle (0.05);

\draw[dashed, ->] (0, 0) .. controls (0.55*\c, -0.5*\c) and (0.9*\c, -1.5*\c) .. (1.06*\c, -2.06*\c);

\draw[shift={(7*\c, -2*\c)}, rotate=180, dashed, ->] (0, 0) .. controls (0.55*\c, -0.5*\c) and (0.9*\c, -1.5*\c) .. (1.06*\c, -2.06*\c);

\draw[shift={(3.2*\c, -2*\c)}, yscale=-1, dashed, ->] (0, 0) .. controls (0.4*\c, -0.6*\c) and (0.7*\c, -1.5*\c) .. (0.808*\c, -2.06*\c);

\draw[dashed, ->] (5*\c, 0.5*\c) .. controls (3.4*\c, 0.75*\c) and (2.7*\c, -0.6*\c) .. (2.9*\c, -2.07*\c);

\draw (0, 0) node[above] {$z_{n-1}$};
\draw (0.3*\c, -2*\c) node {$w_{n, 2}^{-}o$};

\draw (3.7*\c, -2*\c) node {$w_{n, 0}^{-}o$};
\draw (2*\c, -2.75*\c) node {$w_{n, 1}^{-}o$};

\draw (4.55*\c, 0.1*\c) node {$\gamma_{n}$};
\draw (5.45*\c, 0.1*\c) node {$\delta_{n}$};

\draw (5*\c, 0.8*\c) node {$w_{n, 1}^{+}o$};
\draw (6.7*\c, 0) node {$w_{n, 2}^{+}o$};

\draw (7*\c, -1.95*\c) node[below] {$w_{n+1, 2}^{-}o$};

\draw (3.95*\c, -1*\c) node {$v_{n}$}; 
\draw (6.65*\c, -1*\c) node {$w_{n}$};

\draw (1.2*\c, -2.4*\c) node {$\alpha_{n}$};
\draw (2.8*\c, -2.4*\c) node {$\beta_{n}$};

\end{tikzpicture}
\caption{Schematics for Criteria \ref{eqn:pivotingCond1}, \ref{eqn:pivotingCond2}, \ref{eqn:pivotingCond3} and \ref{eqn:pivotingCond4}.}
\label{fig:0thCasePivot}
\end{figure}

First, the condition \begin{equation}\label{eqn:pivotingCond1}
\diam\left(\pi_{\varGam(\gamma_{n})} (w_{n, 0}^{-}o) \cup w_{n, 0}^{+}o\right) = \diam\left(\pi_{\Gamma^{+}(\gamma_{n})} (v_{n}^{-1} o) \cup o\right) < K_{0}
\end{equation}
depends only on $\gamma_{n}$. This holds for at least $(\#S - 1)$ choices in $S$ by the $K_{0}$-Schottky-ness of $S$.

 Similarly, the condition \begin{equation}\label{eqn:pivotingCond2}
\diam\left(\pi_{\varGam(\delta_{n})} (w_{n+1, 2}^{-}o) \cup w_{n, 2}^{+}o\right) = \diam\left(\pi_{\Gamma^{-}(\delta_{n})} (w_{n} o) \cup o\right) < K_{0}
\end{equation}
depends only on $\delta_{n}$, and holds for at least $(\#S - 1)$ choices in $S$.

Fixing the choice of $\gamma_{n}$, the condition \begin{equation}\label{eqn:pivotingCond3}
\diam\left(\pi_{\varGam(\beta_{n})} (w_{n, 1}^{+}o) \cup w_{n, 0}^{-}o\right) = \diam\left(\pi_{\Gamma^{-}(\beta_{n})} (v_{n} c_{n} o) \cup o\right) < K_{0}
\end{equation}
depends only on $\beta_{n}$. This holds for at least $(\#S - 1)$ choices in $S$.

We now additionally fix the choice of $s = (\alpha_{1}, \beta_{1}, \gamma_{1}, \delta_{1}, \ldots, \alpha_{n-1}, \beta_{n-1}, \gamma_{n-1}, \delta_{n-1})$; in particular, $w_{n, 2}^{-}$ and $z_{n-1}$ are now determined. Then the condition\begin{equation}\label{eqn:pivotingCond4}
\diam\left(\pi_{\varGam(\alpha_{n})} (z_{n-1}) \cup w_{n, 2}^{-}o\right) = \diam\left(\pi_{\Gamma^{+}(\alpha_{n})} \left((w_{n, 2}^{-})^{-1} z_{n-1}\right) \cup o\right) < K_{0}
\end{equation}
depends on $\alpha_{n}$. This holds for at least $(\#S - 1)$ choices of $\alpha_{n}$.

In summary, the probability that Criterion (A) holds is at least \[
\frac{\# S - 1}{\#S} \cdot \frac{\# S - 1}{\#S} \cdot \frac{\# S - 1}{\#S} \cdot \frac{\# S - 1}{\#S} \ge 1 - \frac{4}{N_{0}} \qedhere
\]
\end{proof}

We now define the set $\pivotComplete$ of triples $(\beta, \gamma, v) \in S^{2} \times G$ that satisfy Condition \ref{eqn:pivotingCond1} and \ref{eqn:pivotingCond3}: \[\begin{aligned}
&\pivotComplete:= \left\{ (\beta, \gamma, v) \in S^{2} \times G  :  \big(\bar{\Gamma}^{-}(\beta),  v \Pi(\gamma) o\big), \big( v^{-1} o, \Gamma^{+}(\gamma)\big)\,\,\textrm{are $K_{0}$-aligned} \right\}.
\end{aligned}
\]
We also define its section for each $v \in G$:\[
\pivotComplete(v):= \left\{ (\beta, \gamma) \in S^{2}  :  \big(\bar{\Gamma}^{-}(\beta),  v \Pi(\gamma) o\big), \big( v^{-1} o, \Gamma^{+}(\gamma)\big)\,\,\textrm{are $K_{0}$-aligned} \right\}.
\]
While checking Display \ref{eqn:pivotingCond1} and \ref{eqn:pivotingCond3}, we observed that $\#\pivotComplete(v) \ge \#S^{2} - 2\#S$ for each $v\in G$. We now define pivoting.

\begin{lem} \label{lem:pivotEquiv}
Let $\mathbf{s} =  (\alpha_{1}, \beta_{1}, \gamma_{1}, \delta_{1}, \ldots, \alpha_{n}, \beta_{n}, \gamma_{n}, \delta_{n})$ be a choice drawn from $S^{4n}$ and let $\mathbf{w}, \mathbf{v}$ be auxiliary sequences in $G$.

Let $k \in P_{n}(\mathbf{s}; \mathbf{w}, \mathbf{v})$ and let ($\bar{\mathbf{s}}$; $\mathbf{w}, \bar{\mathbf{v}})$ be obtained from $(\mathbf{s}; \mathbf{w}, \mathbf{v})$ by replacing $(\beta_{k}, \gamma_{k}, v_{k})$ with some $(\bar{\beta}_{k}, \bar{\gamma}_{k}, \bar{v}_{k})$ chosen from $\tilde{S}$.

Then  $P_{l}(\mathbf{s}; \mathbf{w}, \mathbf{v}) = P_{l}(\bar{\mathbf{s}}; \mathbf{w}, \bar{\mathbf{v}})$ for any $1 \le l \le n$. 
\end{lem}
On Gromov hyperbolic spaces, this corresponds to \cite[Lemma 5.7]{gouezel2022exponential}.

\begin{proof}
Since $\alpha_{1}, \beta_{1}, \gamma_{1}, \delta_{1}, \ldots, \alpha_{k-1}, \beta_{k-1}, \gamma_{k-1}, \delta_{k-1}$ are intact, $P_{l}(\mathbf{s}) = P_{l}(\bar{\mathbf{s}})$ and $\pivotComplete_{l}(\mathbf{s}) = \pivotComplete_{l}(\bar{\mathbf{s}})$ hold for $l=0, \ldots, k-1$. At step $k$, $\alpha_{k}$ and $\delta_{k}$ satisfy Condition \ref{eqn:pivotingCond4} and Condition \ref{eqn:pivotingCond2} since $k \in P_{n}(\mathbf{s})$. Furthermore, $\bar{\beta}_{k}$ and $\bar{\gamma}_{k}$ satisfy Condition \ref{eqn:pivotingCond1} and \ref{eqn:pivotingCond3} for the new choice $\bar{v}_{k}$: \[
\diam\big(\pi_{\Gamma^{+}(\bar{\gamma}_{k})}(\bar{v}_{k}^{-1} o) \cup o\big) < K_{0} \quad \textrm{and}\quad \diam\big(\pi_{\Gamma^{-}(\bar{\beta}_{k})}(\bar{v}_{k}\bar{c}_{k} o) \cup o\big) < K_{0}, 
\]
since $(\bar{\beta}_{k}, \bar{\gamma}_{k}, \bar{v}_{k}) \in \pivotComplete$. Hence, $k$ is newly added in $P_{k}(\bar{\mathbf{s}})$ and \[
P_{k}(\bar{\mathbf{s}}) = P_{k-1}(\bar{\mathbf{s}}) \cup \{k\} = P_{k-1}(\mathbf{s}) \cup \{k\} = P_{k}(\mathbf{s}).
\] Meanwhile, $z_{k}$ is modified into $\bar{z}_{k} = \bar{w}_{k, 1}^{+}o = g w_{k, 1}^{+}o=gz_{k}$, where $g :=  w_{k, 2}^{-} a_{k}\bar{b}_{k} \bar{v}_{k}\bar{c}_{k} (w_{k, 2}^{-} a_{k}b_{k}v_{k}c_{k})^{-1}$. More generally, we have \begin{equation}\begin{aligned}\label{eqn:pivotedLoci}
\bar{w}_{l, t}^{-} = g w_{l, t}^{-} \quad (t \in \{0, 1, 2\},\, l > k), \\
\bar{w}_{l, 0}^{+} = g w_{l, 0}^{+} \quad\quad\quad\quad\quad\quad\,\,(l > k),\\
\bar{w}_{l, t}^{+} = gw_{l, t}^{+}\quad\quad (t \in \{1, 2\}, l \ge k).
\end{aligned}
\end{equation}

We now claim the following for $k < l \le n$: \begin{enumerate}
\item If $s$ fulfills Criterion (A) at step $l$, then so does $\bar{\mathbf{s}}$.
\item If not and if $(i, j)$ is the maximal pair of indices for $s$ in Criterion (B) at step $l$, then it is also the maximal one for $\bar{\mathbf{s}}$ at step $l$.
\item In both cases, we have $P_{l}(\mathbf{s}) = P_{l}(\bar{\mathbf{s}})$ and $\bar{z}_{l} = g z_{l}$.
\end{enumerate}
Assuming the third item for $l-1$: $P_{l-1}(\mathbf{s}) = P_{l-1}(\bar{\mathbf{s}})$ and $\bar{z}_{l-1} = g z_{l-1}$, Equality \ref{eqn:pivotedLoci} implies the first item. In this case we deduce $P_{l}(\mathbf{s}) = P_{l-1}(\mathbf{s}) \cup \{l\} = P_{l-1}(\bar{\mathbf{s}}) \cup \{l\} = P_{l}(\bar{\mathbf{s}})$ and $\bar{z}_{l} = \bar{w}_{l, 1}^{+}o = gw_{l, 1}^{+}o = gz_{l}$, the third item for $l$.

Furthermore, Equality \ref{eqn:pivotedLoci}  implies that $i$ in $P_{l-1}(\mathbf{s}) \cap \{k, \ldots, l-1\} = P_{l-1}(\bar{\mathbf{s}})\cap \{k, \ldots, l-1\}$ and $j>i$ work for $\mathbf{s}$ in Criterion (B) if and only if they work for $\bar{\mathbf{s}}$. Such $i$ can be found in $\{k, \ldots, l-1\}$, because $k$ survived in $P_{n}(\mathbf{s})$ and should not have been erased at step $l$. Hence, the maximal pair $(i, j)$ for $\mathbf{s}$ is also maximal for $\bar{\mathbf{s}}$. We then deduce $P_{l}(\mathbf{s}) = P_{l-1}(\mathbf{s}) \cap \{1, \ldots, i\} = P_{l-1}(\bar{\mathbf{s}}) \cap \{1, \ldots, i\} = P_{l}(\bar{\mathbf{s}})$ and $\bar{z}_{l} = \bar{w}_{j, 1}^{-}o  = gw_{j, 1}^{-}o = gz_{l}$ (using $j > i$), the third item for $l$.
\end{proof}

For $\mathbf{s}, \mathbf{s}' \in S^{4n}$ and sequences $\mathbf{w}, \mathbf{v}$, $\bar{\mathbf{v}}$ in $G$, we say that $(\bar{\mathbf{s}}; \mathbf{w}, \bar{\mathbf{v}})$ is \emph{pivoted from} $(\mathbf{s}; \mathbf{w}, \mathbf{v})$ if: \begin{itemize}
\item $\alpha_{i} = \bar{\alpha}_{i}$, $\delta_{i} = \bar{\delta}_{i}$ for all $i \in \{1, \ldots, n\}$;
\item $(\bar{\beta}_{i}, \bar{\gamma}_{i}, \bar{v}_{i}) \in \pivotComplete$ for each $i \in P_{n}(\mathbf{s}; \mathbf{w}, \mathbf{v})$, and 
\item $(\beta_{i}, \gamma_{i}, v_{i})= (\bar{\beta}_{i}, \bar{\gamma}_{i}, \bar{v}_{i})$ for each $i \in \{1, \ldots, n\} \setminus  P_{n}(\mathbf{s}; \mathbf{w}, \mathbf{v})$. 
\end{itemize}
By Lemma \ref{lem:pivotEquiv}, being pivoted from each other is an equivalence relation.

Fixing $\mathbf{w}$ and $\mathbf{v}$, for each $\mathbf{s} \in S^{4n}$ let $\mathcal{E}_{n}(\mathbf{s})$ be the equivalence class of $s$: \[
\mathcal{E}_{n}(\mathbf{s}) := \big\{ \bar{\mathbf{s}}  \in S^{4n} : ( \bar{\mathbf{s}}; \mathbf{w}, \mathbf{v})\,\,\textrm{is pivoted from $(\mathbf{s}; \mathbf{w}, \mathbf{v} )$}\big\}.
\]
We endow $\mathcal{E}_{n}(\mathbf{s})$ with the conditional probability of the uniform measure on $S^{4n}$. We now claim that $\#P_{n+1} - \#P_{n}$ conditioned on an equivalence class $\mathcal{E}_{n}(\mathbf{s})$ till step $n$ and the choice at step $n+1$ has uniform exponential tail.

\begin{prop} \label{prop:pivotCondition}
Fix $\mathbf{w} = (w_{i})_{i=0}^{\infty}$ and $\mathbf{v} = (v_{i})_{i=1}^{\infty}$. For each $j \ge0$ and $\mathbf{s} \in S^{4n}$, \[
\Prob\Big(\# P_{n+1}(\tilde{\mathbf{s}}, \alpha_{n+1}, \beta_{n+1}, \gamma_{n+1}, \delta_{n+1}) < \# P_{n}(\mathbf{s}) - j \, \Big| \, \tilde{\mathbf{s}} \in \mathcal{E}_{n}(\mathbf{s}), \,(\alpha_{n+1}, \beta_{n+1}, \gamma_{n+1}, \delta_{n+1}) \in S^{4}\Big)
\]
is less than $(4/N_{0})^{j+1}$.
\end{prop}

On Gromov hyperbolic spaces, this corresponds to \cite[Lemma 5.8]{gouezel2022exponential}.

\begin{proof}
An element $\tilde{\mathbf{s}} \in \mathcal{E}_{n}(\mathbf{s})$ is determined by its coordinates $(\tilde{\beta}_{k}, \tilde{\gamma}_{k})_{k \in P_{n}(\mathbf{s})}$ subject to the condition $(\tilde{\beta}_{k}, \tilde{\gamma}_{k}, v_{k}) \in \pivotComplete$. We consider a finer equivalence class by additionally fixing the coordinates $\gamma_{k}$'s: for $\tilde{\mathbf{s}} \in \mathcal{E}_{n}(\mathbf{s})$, let $\mathcal{E}_{n}'(\tilde{\mathbf{s}})$ be the set of $\bar{\mathbf{s}} \in \mathcal{E}_{n}(\mathbf{s})$ such that $\bar{\gamma}_{k} = \tilde{\gamma}_{k}$ for all $k$. Then $\mathcal{E}_{n}(\mathbf{s})$ is partitioned into $\{ \mathcal{E}_{n}'(\tilde{\mathbf{s}}): \tilde{\mathbf{s}} \in \mathcal{E}_{n}(\mathbf{s})\}$, and it suffices to establish the estimates on each $\mathcal{E}_{n}'(\tilde{\mathbf{s}})$. Henceforth, we will prove that\[
\Prob\Big(\# P_{n+1}(\tilde{\mathbf{s}}, \alpha_{n+1}, \beta_{n+1}, \gamma_{n+1}, \delta_{n+1}) < \# P_{n}(\mathbf{s}) - j \, \Big| \, \tilde{\mathbf{s}} \in \mathcal{E}_{n}'(\mathbf{s}), \,(\alpha_{n+1}, \beta_{n+1}, \gamma_{n+1}, \delta_{n+1}) \in S^{4}\Big)
\]
is less than $(4/N_{0})^{j+1}$ for each $\mathbf{s} = (\alpha_{1}, \beta_{1}, \gamma_{1}, \delta_{1}, \ldots, \alpha_{n}, \beta_{n}, \gamma_{n}, \delta_{n}) \in S^{4n}$ and $j \ge 0$.

Recall that we are fixing the sequences $\mathbf{w}$ and $\mathbf{v}$ throughout the proof. Let us define \[
\pivotComplete_{k} :=\left\{ \beta \in S : \big(\bar{\Gamma}^{-}(\beta), v_{k} \Pi(\gamma_{k}) o\big)\,\,\textrm{is $K_{0}$-aligned}\right\}.
\]
Then $\mathcal{E}_{n}'(\mathbf{s})$ is parametrized by $\prod_{i \in P_{n}(\mathbf{s})} \pivotComplete_{k}$ with the uniform measure. Let\[
\mathcal{A} := \Big\{(\alpha_{n+1}, \beta_{n+1}, \gamma_{n+1}, \delta_{n+1}) \in S^{4}: \#P_{n+1}(\mathbf{s}, \alpha_{n+1}, \beta_{n+1}, \gamma_{n+1}, \delta_{n+1}) = \#P_{n}(\mathbf{s})+1 \Big\}.
\] Lemma \ref{lem:0thCasePivot} implies that $\Prob(\mathcal{A}) \ge 1-4/N_{0}$ with respect to the uniform measure on $S^{4}$. Note that for each element $(\alpha_{n+1}, \beta_{n+1}, \gamma_{n+1}, \delta_{n+1})$ of $\mathcal{A}$, we have \[
P_{n}(\mathbf{s}) \subseteq P_{n}(\mathbf{s}) \cup \{n+1\} = P_{n+1}(\mathbf{s}, \alpha_{n+1}, \beta_{n+1}, \gamma_{n+1}, \delta_{n+1}).
\]
Hence, for each $\tilde{\mathbf{s}}\in \mathcal{E}_{n}'(\mathbf{s})$, $(\tilde{\mathbf{s}}, \alpha_{n+1}, \beta_{n+1}, \gamma_{n+1}, \delta_{n+1})$ (as a choice in $S^{4(n+1)}$) is pivoted from $(\mathbf{s}, \alpha_{n+1}, \beta_{n+1}, \gamma_{n+1}, \delta_{n+1})$ and $P_{n+1}(\tilde{\mathbf{s}}) = P_{n+1}(\mathbf{s}) = P_{n}(\mathbf{s}) \cup \{n+1\} = P_{n}(\tilde{\mathbf{s}}) \cup \{n+1\}$. Thanks to this, we have\[\begin{aligned}
&\Prob\Big( \#P_{n+1}(\tilde{\mathbf{s}}, \alpha_{n+1}, \beta_{n+1}, \gamma_{n+1}, \delta_{n+1}) < \#P_{n}(\tilde{\mathbf{s}}) \, \Big| \, \tilde{\mathbf{s}} \in \mathcal{E}_{n}'(\mathbf{s}), (\alpha_{n+1}, \beta_{n+1}, \gamma_{n+1}, \delta_{n+1}) \in S^{4} \Big) \\& \le 1 - \Prob(\mathcal{A}) \le 4/N_{0}.
\end{aligned}
\]
This settles the case $j=0$.

Now let $j = 1$. The event under discussion becomes void when $\# P_{n}(\mathbf{s}) < 2$. Excluding such cases, let $l< m$ be the last 2 elements of $P_{n}(\mathbf{s})$. We now freeze the coordinates $\beta_{k}$'s except for $k=m$. Namely, for $\tilde{\mathbf{s}} \in \mathcal{E}_{n}'(\mathbf{s})$, let $E^{(m)}(\tilde{\mathbf{s}})$ be the set of $\bar{\mathbf{s}} \in \mathcal{E}_{n}'(\mathbf{s})$ such that $\bar{\beta}_{k} = \tilde{\beta}_{k}$ for $k \in P_{n}(\mathbf{s}) \setminus \{m\}$. Then $\{E^{(m)}(\tilde{\mathbf{s}}) : \tilde{\mathbf{s}} \in \mathcal{E}_{n}'(\mathbf{s}) \}$ becomes a partition of $\mathcal{E}_{n}'(\mathbf{s})$, and $E^{(m)}(\tilde{\mathbf{s}})$ is parametrized by $\bar{\beta}_{m} \in \pivotComplete_{m}$ with the uniform measure. Note that $\pivotComplete_{m}$ has at least $\#S - 1$ elements. 

Fixing $(\alpha_{n+1}, \beta_{n+1}, \gamma_{n+1}, \delta_{n+1}) \in S^{4}$, and let $F^{(m)}(\tilde{\mathbf{s}})$ be the set of $\bar{\mathbf{s}} \in E^{(m)}(\tilde{\mathbf{s}})$ such that $\left(\varGam(\bar{\beta}_{m}), \bar{w}_{n+2, 2}^{-} o \right)$ $K_{0}$-aligned, or more precisely, \begin{equation}\label{eqn:addCondPivot}\begin{aligned}
&\diam\left(\pi_{\Gamma^{-1}(\bar{\beta}_{m})}((\tilde{w}_{m, 0}^{-})^{-1} \tilde{w}_{n, 2}^{-} a_{n+1} b_{n+1} v_{n+1} c_{n+1}d_{n+1} w_{n+1} o) \cup o\right) \\
&= \diam\left(o \cup \pi_{\Gamma^{-1}(\bar{\beta}_{m})} (v_{m} \tilde{c}_{m}\tilde{d}_{m} w_{m}  \cdots \tilde{c}_{n} \tilde{d}_{n} w_{n} \cdot a_{n+1}b_{n+1}v_{n+1} c_{n+1} d_{n+1} w_{n+1} o) \right) < K_{0}.
\end{aligned}
\end{equation}
This amounts to requiring a new Schottky condition to $\bar{\beta}_{m}$, in addition to the alignment of $\big( \bar{\Gamma}^{-1}(\bar{\beta}_{m}), v_{m} \Pi(\gamma_{m}) o\big)$; there are at least $\#S - 2$ choices that additionally satisfy this.

We now claim $\# P_{n+1}(\bar{\mathbf{s}}, \alpha_{n+1}, \beta_{n+1}, \gamma_{n+1}, \delta_{n+1}) \ge \# P_{n}(\mathbf{s}) - 1$ for $\bar{\mathbf{s}} \in F^{(m)}(\tilde{\mathbf{s}})$. First, since $l<m$ are consecutive elements in $P_{n}(\mathbf{s}) = P_{n}(\bar{\mathbf{s}})$, Lemma \ref{lem:intermediate} asserts that $\left(\varGam(\bar{\delta}_{l}),  \varGam(\bar{\alpha}_{m}) \right)$
is $D_{0}$-semi-aligned. Moreover, Lemma \ref{lem:pivotJtn} and Condition \ref{eqn:addCondPivot} imply that \[
\left( \varGam(\bar{\alpha}_{m}), \varGam(\bar{\beta}_{m})\right), \quad \left(\varGam(\bar{\beta}_{m}), \bar{w}_{n+2, 2}^{-} o\right)
\] are $D_{0}$-aligned and $K_{0}$-aligned, respectively. These together imply that \[
\left(\varGam(\bar{\delta}_{l}), \varGam(\bar{\beta}_{m})\right), \quad \left(\varGam(\bar{\beta}_{m}), \bar{w}_{n+2, 2}^{-} o\right)
\]
are $D_{0}$-semi-aligned and $K_{0}$-aligned, respectively: the pair $(l, m)$ qualifies Criterion (B) at step $n+1$. Hence, $P_{n+1}(\bar{\mathbf{s}}, \alpha_{n+1}, \beta_{n+1}, \gamma_{n+1}, \delta_{n+1}) \supseteq P_{n}(\bar{\mathbf{s}}) \cap \{1, \ldots, l\}$.

As a result, for each $\tilde{\mathbf{s}} \in \mathcal{E}_{n}(\mathbf{s})$ and $(\alpha_{n+1}, \ldots, \delta_{n+1}) \in S^{4}$ we have \[\begin{aligned}
&\Prob\Big(\#P_{n+1}(\bar{\mathbf{s}}, \alpha_{n+1}, \beta_{n+1}, \gamma_{n+1}, \delta_{n+1}) < \#P_{n}(\mathbf{s}) - 1 \, \Big| \, \bar{\mathbf{s}} \in E^{(m)}(\tilde{\mathbf{s}})\Big) \\
&\le  \frac{\# \left[ E^{(m)}(\tilde{\mathbf{s}}) \setminus F^{(m)}(\tilde{\mathbf{s}})\right]}{\# E^{(m)}(\tilde{\mathbf{s}})} \le \frac{2}{\# S - 1} \le \frac{3}{N_{0}}.
\end{aligned}
\]
Since $\{ E^{(m)}(\tilde{\mathbf{s}}): \tilde{\mathbf{s}} \in \mathcal{E}_{n}(\mathbf{s})\}$ partitions $\mathcal{E}_{n}(\mathbf{s})$, we deduce \[
\Prob\Big(\# P_{n+1}(\tilde{\mathbf{s}}, \alpha_{n+1}, \beta_{n+1}, \gamma_{n+1}, \delta_{n+1}) < \# P_{n}(\mathbf{s}) - 1 \, \Big| \, \tilde{\mathbf{s}} \in \mathcal{E}_{n}(\mathbf{s})\Big) \le \frac{3}{N_{0}}.
\]
for each $(\alpha_{n+1}, \beta_{n+1}, \gamma_{n+1}, \delta_{n+1}) \in S^{4}$. Moreover, this probability vanishes when $(\alpha_{n+1}, \ldots, \delta_{n+1}) \in \mathcal{A}$. Since $\Prob\big((\alpha_{n+1}, \ldots, \delta_{n+1}) \in \mathcal{A} \,\big|\, (\alpha_{n+1}, \ldots, \delta_{n+1}) \in S^{4}\big) \ge 1-4/N_{0}$, we deduce that \begin{equation}\label{eqn:firstPivotEst}\begin{aligned}
&\Prob\Big(\#P_{n+1}(\tilde{\mathbf{s}}, \alpha_{n+1}, \beta_{n+1}, \gamma_{n+1}, \delta_{n+1}) < \#P_{n}(\mathbf{s}) - 1 \, \Big| \, \tilde{\mathbf{s}} \in \mathcal{E}_{n}(\mathbf{s}), (\alpha_{n+1}, \beta_{n+1}, \gamma_{n+1}, \delta_{n+1}) \in S^{4}\Big) \\
&\le \frac{4}{N_{0}} \cdot \frac{4}{N_{0}} \le \left(\frac{4}{N_{0}}\right)^{2}.
\end{aligned}
\end{equation}

Now let $j =2$. Excluding the void case, we assume that $\#P_{n}(\mathbf{s}) \ge 3$; let $l' < l<m$ be the last 3 elements. To ease the notation, for $\beta \in S$ we define \[
\mathbf{s}'(\beta) := (\alpha_{1}, \beta_{1}, \gamma_{1}, \delta_{1}, \ldots, \alpha_{m}, \beta, \gamma_{m}, \delta_{m}, \ldots, \alpha_{n}, \beta_{n}, \gamma_{n}, \delta_{n}).
\]
In other words, $\mathbf{s}'(\beta)$ is obtained from $\mathbf{s}$ by replacing $\beta_{m}$ with $\beta$. Now let \[
\mathcal{A}_{1} := \left\{ \left(  \beta, \alpha_{n+1}, \beta_{n+1}, \gamma_{n+1}, \delta_{n+1} \right) \in \pivotComplete_{m} \times S^{4} : \#P_{n+1}\left(\mathbf{s}'(\beta), \alpha_{n+1}, \ldots, \delta_{n+1}\right) \ge \#P_{n}(\mathbf{s}) - 1 \right\}.
\]
Equivalently, we are requiring \[
P_{n}(\mathbf{s}) \cap \{1, \ldots, l\} \subseteq P_{n+1}\left(\mathbf{s}', \alpha_{n+1}, \beta_{n+1}, \gamma_{n+1}, \delta_{n+1}\right).
\]
(This equivalence relies on the fact $P_{n}(\mathbf{s}') = P_{n}(\mathbf{s})$ due to Lemma \ref{lem:pivotEquiv}.)

\begin{obs}\label{obs:pivotLater2nd}
For each
 \[
\tilde{\mathbf{s}} = (\tilde{\alpha}_{k}, \tilde{\beta}_{k}, \tilde{\gamma}_{k}, \tilde{\delta}_{k})_{i=1}^{n}  \in \mathcal{E}_{n}'(\mathbf{s}), \quad (\alpha_{n+1}, \beta_{n+1}, \gamma_{n+1}, \delta_{n+1}) \in S^{4},
\] 
 $(\tilde{\beta}_{m},  \alpha_{n+1}, \beta_{n+1}, \gamma_{n+1}, \delta_{n+1}) \in \mathcal{A}_{1}$ if and only if $\#P_{n+1}(\tilde{\mathbf{s}}, \alpha_{n+1}, \ldots, \delta_{n+1}) \ge \# P_{n}(\mathbf{s}) - 1$.
\end{obs}

To see this, suppose first that $(\tilde{\beta}_{m}, \alpha_{n+1}, \beta_{n+1}, \gamma_{n+1}, \delta_{n+1}) \in \mathcal{A}_{1}$. Then $(\tilde{\mathbf{s}}, \alpha_{n+1}, \beta_{n+1}, \gamma_{n+1}, \delta_{n+1})$ is pivoted from $\left(\mathbf{s}'( \tilde{\beta}_{m}), \alpha_{n+1}, \beta_{n+1}, \gamma_{n+1}, \delta_{n+1}\right)$, as they differ only at entries $\beta_{k}$'s for $k\in P_{n}(\mathbf{s}) \cap \{1, \ldots, l\} \subseteq P_{n+1}\left(\mathbf{s}', \alpha_{n+1}, \beta_{n+1}, \gamma_{n+1}, \delta_{n+1}\right)$. Lemma \ref{lem:pivotEquiv} then implies that \[
P_{n}(\mathbf{s}) \cap \{1, \ldots, l\} \subseteq P_{n+1}(\mathbf{s}', \alpha_{n+1}, \beta_{n+1}, \gamma_{n+1}, \delta_{n+1}) = P_{n+1}(\tilde{\mathbf{s}}, \alpha_{n+1}, \beta_{n+1}, \gamma_{n+1}, \delta_{n+1})
\]
and $\#P_{n+1}(\tilde{\mathbf{s}}, \alpha_{n+1}, \beta_{n+1}, \gamma_{n+1}, \delta_{n+1}) \ge \# P_{n}(\mathbf{s}) - 1$. 

Conversely, suppose $\#P_{n+1}(\tilde{\mathbf{s}}, \alpha_{n+1}, \beta_{n+1}, \gamma_{n+1}, \delta_{n+1}) \ge \# P_{n}(\mathbf{s}) - 1$. ($\ast$) Recall that $P_{n}(\tilde{\mathbf{s}}) = P_{n}(\mathbf{s})$, and recall that  $P_{n+1}(\tilde{\mathbf{s}}, \alpha_{n+1}, \ldots, \delta_{n+1})$ is either an initial section of $P_{n}(\tilde{s})$ or contains $P_{n}(\tilde{s})$. Considering these, the assumption ($\ast)$ implies\[
P_{n}(\mathbf{s}) \cap \{1, \ldots, l\} \subseteq P_{n+1}(\tilde{\mathbf{s}}, \alpha_{n+1}, \beta_{n+1}, \gamma_{n+1}, \delta_{n+1}).
\]
Then $\left(\mathbf{s}'( \tilde{\beta}_{m}), \alpha_{n+1}, \beta_{n+1}, \gamma_{n+1}, \delta_{n+1}\right)$ is pivoted from $(\tilde{\mathbf{s}}, \alpha_{n+1}, \beta_{n+1}, \gamma_{n+1}, \delta_{n+1})$, as the former choice differs from the latter choice only at entries $(\tilde{\alpha}_{k}, \tilde{\beta}_{k}, \tilde{\gamma}_{k}$)'s for $k \in P_{n}(\mathbf{s}) \cap \{1, \ldots, l\} \subseteq P_{n+1}\left(\tilde{\mathbf{s}}, \alpha_{n+1}, \beta_{n+1}, \gamma_{n+1}, \delta_{n+1}\right)$. Lemma \ref{lem:pivotEquiv} then implies that \[
P_{n}(\mathbf{s}) \cap \{1, \ldots, l\} \subseteq P_{n+1}(\tilde{\mathbf{s}}, \alpha_{n+1}, \ldots, \delta_{n+1}) = P_{n+1}(\mathbf{s}', \alpha_{n+1}, \ldots, \delta_{n+1})
\]
and $(\tilde{\beta}_{m},  \alpha_{n+1}, \beta_{n+1}, \gamma_{n+1}, \delta_{n+1}) \in \mathcal{A}_{1}$.

Combining Observation \ref{obs:pivotLater2nd} and Inequality \ref{eqn:firstPivotEst}, we deduce \[\begin{aligned}
&\Prob\big(\mathcal{A}_{1}\, \big| \, \pivotComplete_{m} \times S^{4}\big) \\
&= \Prob\Big(  (\tilde{\beta}_{m}, \alpha_{n+1}, \beta_{n+1}, \gamma_{n+1}, \delta_{n+1}) \in \mathcal{A}_{1}\, \Big| \, \tilde{\mathbf{s}} \in \mathcal{E}_{n}'(\mathbf{s}), (\alpha_{n+1}, \beta_{n+1}, \gamma_{n+1}, \delta_{n+1}) \in S^{4}\Big) \\
& =\Prob\Big(\#P_{n+1}(\tilde{\mathbf{s}}, \alpha_{n+1}, \beta_{n+1}, \gamma_{n+1}, \delta_{n+1}) \ge \#P_{n}(\mathbf{s}) - 1 \, \Big| \, \tilde{\mathbf{s}} \in \mathcal{E}_{n}'(\mathbf{s}), (\alpha_{n+1}, \beta_{n+1}, \gamma_{n+1}, \delta_{n+1}) \in S^{4}\Big) \\
&\ge 1 - \left(\frac{4}{N_{0}}\right)^{2}.
\end{aligned}
\]

This time, we freeze the coordinates $\beta_{k}$'s except for $k=l$: for $\tilde{\mathbf{s}} \in \mathcal{E}_{n}'(\mathbf{s})$, let $E^{(l)}(\tilde{\mathbf{s}})$ be the set of $\bar{\mathbf{s}} \in \mathcal{E}_{n}'(\mathbf{s})$ such that $\bar{\beta}_{k} = \tilde{\beta}_{k}$ for $k \in P_{n}(\mathbf{s}) \setminus \{l\}$.  Then $\{E^{(l)}(\tilde{\mathbf{s}}) : \tilde{\mathbf{s}} \in \mathcal{E}_{n}'(\mathbf{s})\}$ partitions $\mathcal{E}_{n}'(\mathbf{s})$ and  $E^{(l)}(\tilde{\mathbf{s}})$ is parametrized by $\bar{\beta}_{l} \in \pivotComplete_{l}$ with the uniform measure; note that $\# \pivotComplete_{l} \ge \#S - 1$.

Fixing $\tilde{\mathbf{s}}$, now pick $(\alpha_{n+1}, \beta_{n+1}, \gamma_{n+1}, \delta_{n+1}) \in S^{4}$ and let $F^{(l)}(\tilde{\mathbf{s}})$ be the set of $\bar{\mathbf{s}} \in E^{(m)}(\tilde{\mathbf{s}})$ such that $\left(\varGam(\bar{\beta}_{l}), \bar{w}_{n+2, 2}^{-} o \right)$ $K_{0}$-aligned, i.e., \begin{equation}\label{eqn:addCondPivot2}\begin{aligned}
&\diam\left(\pi_{\Gamma^{-1}(\bar{\beta}_{l})}((\tilde{w}_{l, 0}^{-})^{-1} \tilde{w}_{n, 2}^{-} a_{n+1} b_{n+1} v_{n+1} c_{n+1}d_{n+1} w_{n+1} o) \cup o\right)  < K_{0}.
\end{aligned}
\end{equation}
This amounts to requiring another Schottky condition to $\bar{\beta}_{l}$; there are at least $\#S - 2$ choices that additionally satisfy this.

We now claim that $\#P_{n+1}(\bar{\mathbf{s}}, \alpha_{n+1}, \beta_{n+1}, \gamma_{n+1}, \delta_{n+1}) \ge \# P_{n}(\mathbf{s}) - 2$ for $\bar{\mathbf{s}} \in F^{(l)}(\tilde{\mathbf{s}})$. First, since $l' < l$ are consecutive elements in $P_{n}(\bar{\mathbf{s}})$, Lemma \ref{lem:intermediate} asserts that\[
\left(\varGam(\bar{\delta}_{l'}),  \varGam(\bar{\alpha}_{l}) \right)
\]
is $D_{0}$-semi-aligned. Moreover, Lemma \ref{lem:pivotJtn} and Condition \ref{eqn:addCondPivot2} imply that \[
\left( \varGam(\bar{\alpha}_{l}), \varGam(\bar{\beta}_{l}) \right), \quad \left(\varGam(\bar{\beta}_{l}),  \bar{w}_{n+2, 2}^{-}o\right)
\] is $D_{0}$-aligned and $K_{0}$-aligned, respectively. Combining these, we observe that the pair $(l', l)$ qualifies Criterion (B) at step $n+1$. This implies $P_{n+1}(\bar{\mathbf{s}}, \alpha_{n+1}, \beta_{n+1}, \gamma_{n+1}, \delta_{n+1}) \supseteq P_{n}(\bar{\mathbf{s}}) \cap \{1, \ldots, l'\}$, hence the claim.

As a result, for each $\tilde{\mathbf{s}} \in \mathcal{E}_{n}(\mathbf{s})$ and $(\alpha_{n+1}, \ldots, \delta_{n+1}) \in S^{4}$ we have \[\begin{aligned}
&\Prob\Big(\# P_{n+1}(\bar{\mathbf{s}}, \alpha_{n+1}, \beta_{n+1}, \gamma_{n+1}, \delta_{n+1}) < \# P_{n}(\mathbf{s}) - 2 \, \Big| \,  \bar{\mathbf{s}}  \in E^{(l)}(\tilde{\mathbf{s}})\Big) \\
&\le \frac{\# \left[ E^{(l)}(\tilde{\mathbf{s}}) \setminus F^{(l)}(\tilde{\mathbf{s}})\right]}{\# E^{(l)}(\tilde{\mathbf{s}})}\le \frac{2}{\# S - 1}  \le \frac{3}{N_{0}}.
\end{aligned}
\]
Moreover, Observation \ref{obs:pivotLater2nd} asserts that the above probability vanishes for those equivalence classes $E^{(l)}(\tilde{\mathbf{s}})$ such that $(\tilde{\beta}_{m}, \alpha_{n+1}, \ldots, \delta_{n+1}) \in \mathcal{A}_{1}$. Since $\Prob[\mathcal{A}_{1} | \pivotComplete_{m} \times S^{4}] \le (4/N_{0})^{2}$, we conclude \begin{equation}\label{eqn:2ndPivotEst}\begin{aligned}
&\Prob\Big(\# P_{n+1}(\tilde{\mathbf{s}}, \alpha_{n+1}, \beta_{n+1}, \gamma_{n+1}, \delta_{n+1}) < \# P_{n}(\mathbf{s}) - 2 \, \Big| \, \tilde{\mathbf{s}} \in \mathcal{E}_{n}'(\mathbf{s}), (\alpha_{n+1}, \beta_{n+1}, \gamma_{n+1}, \delta_{n+1}) \in S^{4}\Big) \\
&\le \left(\frac{4}{N_{0}}\right)^{2} \times \frac{4}{N_{0}} \le \left(\frac{4}{N_{0}}\right)^{3}.
\end{aligned}
\end{equation}
We repeat this procedure for $j< \# P_{n}(\mathbf{s})$. The case $j \ge \# P_{n}(\mathbf{s})$ is void.
\end{proof}

\begin{cor}\label{cor:pivotDistbn}
Let us fix $\mathbf{w}$ and $\mathbf{v}$. When $\mathbf{s} = (\alpha_{i}, \beta_{i}, \gamma_{i}, \delta_{i})_{i=1}^{n}$ is chosen from $S^{4n}$ with the uniform measure, $\# P_{n}(\mathbf{s})$ is greater in distribution than the sum of $n$ i.i.d. $X_{i}$, whose distribution is given by \begin{equation}\label{eqn:expRV}
\Prob(X_{i}=j) = \left\{\begin{array}{cc} (N_{0} - 4)/N_{0} & \textrm{if}\,\, j=1,\\ (N_{0} - 4)4^{-j}/N_{0}^{-j+1}& \textrm{if}\,\, j < 0, \\ 0 & \textrm{otherwise.}\end{array}\right.
\end{equation}

Moreover, we have $\Prob(\#P_{n}(\mathbf{s}) \le (1-10/N_{0}) n) \le e^{-n/K}$ for some $K>0$.
\end{cor}

\begin{proof}
Let $\{X_{i}\}_{i}$ be the family of i.i.d. as in Equation \ref{eqn:expRV} that is also assumed to be independent from the choice $\mathbf{s}$. Lemma \ref{lem:0thCasePivot} and Proposition \ref{prop:pivotCondition} together imply the following for each $0\le k <n$:
\begin{equation}\label{eqn:probDistbn}
\Prob\left( \# P_{k+1}(\mathbf{s})\ge i + j \, \Big| \, \#P_{k}(\mathbf{s}) = i\right) \ge  \left\{\begin{array}{cc} 1 - \frac{4}{N_{0}} & \textrm{if}\,\, j=1,\\ 1 - \left(\frac{4}{N_{0}}\right)^{-j+1} & \textrm{if}\,\, j < 0.\end{array}\right. \quad (i=0, 1, 2, \ldots)
\end{equation}
Hence, there exists a nonnegative random variable $U_{k}$ such that $\# P_{k+1} - U_{k}$ and $\# P_{k} + X_{k+1}$ have the same distribution.

For each $1 \le k \le n$, we claim that $\Prob(\#P_{k} \ge i) \ge \Prob (X_{1} + \cdots + X_{k} \ge i)$ for each $i$.  For $k=1$, we have $\#P_{k-1}= 0$ and the claim follows from Inequality \ref{eqn:probDistbn}. Given the claim for $k$, we have \[\begin{aligned}
\Prob(\#P_{k+1}\ge i) &\ge \Prob(\#P_{k} + X_{k+1} \ge i) = \sum_{j} \Prob(\#P_{k}\ge j) \Prob(X_{k+1} = i - j)\\
 & \ge \sum_{j} \Prob(X_{1} + \cdots + X_{k} \ge j) \Prob(X_{k+1} = i-j) \\
 &= \Prob(X_{1} + \cdots + X_{k} + X_{k+1} \ge i).
 \end{aligned}
\]

The second claim holds since $X_{i}$'s have finite exponential moment and $\E[X_{i}] \ge 1-9/N_{0}$. 
\end{proof}

We now describe a simpler situation when $v_{i} = id$ for all $i$, i.e., we study \[
w_{0} a_{1}b_{1}c_{1}d_{1} w_{1} \cdots a_{n}b_{n} c_{n}d_{n} w_{n} \cdots.
\]
Before defining the pivoting, note that the conditions \[
\left(\bar{\Gamma}^{-} (\beta), \prodSeq(\gamma) o\right), \,\, \left(o, \Gamma^{+} (\gamma) \right) \,\,\textrm{are $K_{0}$-aligned}
\]
are satisfied by every pair of Schottky sequences $(\beta, \gamma) \in S^{2}$, as we proved in Lemma \ref{lem:pivotJtn}. Hence, Criterion (A) while defining the set of pivotal times is simplified as follows:  
\begin{enumerate}[label=(\Alph*')]
\item When $\left(z_{n-1}, \varGam(\alpha_{n})\right)$ and $\left(\varGam(\delta_{n}), w_{n+1, 2}^{-}o\right)$ are $K_{0}$-aligned, we set $P_{n} = P_{n-1} \cup \{n\}$ and $z_{n} = w_{n, 1}^{+}o$ (see Figure \ref{fig:0thCasePivot}).
\end{enumerate}
Moreover, $\pivotComplete$ contains all of $\{(\beta, \gamma, id) : \beta, \gamma \in S\}$. Hence, when $\mathbf{w}$ and $\mathbf{v}= (id)_{i=1}^{\infty}$ are fixed, the previous definition reads as follows: given a choice $\mathbf{s} = (\alpha_{1}, \beta_{1}, \ldots, \gamma_{n}, \delta_{n})$ in $S^{4n}$, we say that $\bar{\mathbf{s}} \in S^{4n}$ is \emph{pivoted from} $\mathbf{s}$ if: \begin{itemize}
\item $\alpha_{i} = \bar{\alpha}_{i}$, $\delta_{i} = \bar{\delta}_{i}$ for all $i \in \{1, \ldots, n\}$;
\item $(\beta_{i}, \gamma_{i})= (\bar{\beta}_{i}, \bar{\gamma}_{i})$ for each $i \in \{1, \ldots, n\} \setminus  P_{n}(\mathbf{s})$.
\end{itemize}
Therefore, for $\mathbf{s} \in S^{4n}$, $\bar{\mathbf{s}} \in \mathcal{E}_{n}(\mathbf{s})$ is parametrized by their choices $(\bar{\beta}_{i}, \bar{\gamma}_{i})_{i \in P_{n}(\mathbf{s})}$ distributed according to the uniform measure on $S^{2\#P_{n}(\mathbf{s})}$.

\subsection{Pivotal times in random walks} \label{subsection:pivotRW}

In this subsection, we define pivotal times for random walks and prove Proposition \ref{prop:gouezelRW1}. Let $\mu$ be a non-elementary probability measure on $G$ and $S \subseteq (\supp \mu)^{M_{0}}$ be a large enough $K_{0}$-Schottky set with cardinality $N_{0} \ge 400$. We also fix an integer $n\ge0$.

Let $\mu_{S}$ be the uniform measure on $S$. By taking suitably small $\alpha$, we can decompose $\mu^{4M_{0}}$ as \[
\mu^{4M_{0}} = \alpha \mu_{S}^{4} + (1-\alpha) \nu 
\]
for some probability measure $\nu$. We then consider Bernoulli RVs $(\rho_{i})_{i}$ with $\Prob(\rho_{i} = 1) = \alpha$ and $\Prob(\rho_{i} = 0) = 1-\alpha$, $(\eta_{i})_{i}$ with the law $\mu_{S}^{4}$ and $(\nu_{i})_{i}$ with the law $\nu$, all independent, and define \[
(g_{4M_{0}k+1}, \,\ldots, \,g_{4M_{0}k + 4M_{0}}) = \left\{\begin{array}{cc} \nu_{k} & \textrm{when}\,\, \rho_{k} = 0, \\ \eta_{k} & \textrm{when} \,\, \rho_{k} = 1. \end{array}\right.
\]
Then $(g_{i})_{i =1}^{\infty}$ has the law $\mu^{\infty}$. Since we need to prove Proposition \ref{prop:gouezelRW1} by fixing the choice of $g_{1}, \ldots, g_{\lfloor n/2\rfloor + 1}$ and $g_{n+1}$, we slightly modify $\rho_{i}$'s, namely, \[
\rho_{i}^{(n)} := \left\{\begin{array}{cc} 0 & \textrm{if $i \le n/8M_{0}$ or $i = \lfloor n/4M_{0} \rfloor$}, \\ \rho_{i} & \textrm{otherwise}.\end{array}\right.
\]
Let $\Omega$ be the ambient probability space on which the above RVs are all measurable. We denote by  $\sumRho(k) := \sum_{i=0}^{k} \rho_{i}^{(n)}$ the number of the Schottky slots till $k$ and by $\stopping(i) := \min \{j \ge 0 : \sumRho(j) = i\}$ the $i$-th Schottky slot. We also set $\vartheta(0) = -1$. Note that $\{j \ge 0 : \rho_{j}^{(n)} = 1\}= \{\vartheta(1) < \vartheta(2) < \ldots\}$.

For each $\w \in \Omega$ and $i \ge 1$ we define \[\begin{aligned}
w_{i-1} &:= g_{4M_{0}[\stopping(i-1) + 1] + 1} \cdots g_{4M_{0} \stopping(i)},\\
\alpha_{i} &:= (g_{4M_{0} \stopping(i) + 1}, \,\ldots, \,g_{4M_{0} \stopping(i) +M_{0}} ),\\
\beta_{i} &:= (g_{4M_{0} \stopping(i) + M_{0}+1}, \,\ldots, \,g_{4M_{0} \stopping(i) +2M_{0}} ),\\
\gamma_{i} &:= (g_{4M_{0} \stopping(i) + 2M_{0}+1}, \,\ldots, \,g_{4M_{0} \stopping(i) +3M_{0}} ),\\
\delta_{i} &:= (g_{4M_{0} \stopping(i) + 3M_{0}+1}, \,\ldots, \,g_{4M_{0} \stopping(i) +4M_{0}} ).
\end{aligned}
\]
In other words, $\eta_{\stopping(i)}$ becomes $(\alpha_{i}, \beta_{i}, \gamma_{i}, \delta_{i})$ (with $M_{0}$ steps each) and $w_{i}$ is the product of intermediate steps between $\eta_{\stopping(i-1)}$ and $\eta_{\stopping(i)}$. As in Subsection \ref{subsection:pivotDiscrete}, we write $a_{i} := \prodSeq(\alpha_{i})$, $b_{i} := \prodSeq(\beta_{i})$ and so on. We then have \begin{equation}\label{eqn:wnRepresent}
\w_{4M_{0} \stopping(l+1)} = w_{0} a_{1}b_{1}v_{1}c_{1}d_{1} w_{1} \cdots a_{l} b_{l} v_{l}c_{l} d_{l} w_{l}
\end{equation}
for each $l>0$. Following the discussion in Subsection \ref{subsection:pivotDiscrete}, we define \begin{equation}\label{eqn:pivotalTimeModel1}\begin{aligned}
P_{1}(\w) &= P_{1}\left((a_{1}, b_{1}, c_{1}, d_{1});(w_{i})_{i=0}^{1}, (id)_{i=1}^{1}\right), \\
P_{2}(\w) &= P_{2}\left((a_{i}, b_{i}, c_{i}, d_{i})_{i=1}^{2}; (w_{i})_{i=0}^{2},(id)_{i=1}^{2}\right), \\
&\vdots \\
\end{aligned}
\end{equation}
We finally define \[
\diffPivot(\w) := \left\{4M_{0} \stopping(i) + 2M_{0} : i \in \liminf_{k} P_{k}(\w)\right\}.
\]

Recall that $P_{k}$ is formed from $P_{k-1}$ by adjoining a new element $k$ or taking an initial section of $P_{k-1}$. Hence, any initial section $\{i(1) < \ldots < i(N)\}$ of $\liminf_{k} P_{k}(\w)$ is an initial section of some $P_{m}(\w)$ (in fact, for all sufficiently large $m$). Proposition \ref{prop:extremal} then tells us that: 
\begin{obs}\label{obs:pivotRandom1}
Let $\diffPivot(\w) = \{i(1) < i(2) < \ldots\}$. Then
\[
\left(o, \varGam(\alpha_{i(1)}), \varGam(\beta_{i(1)}), \varGam(\gamma_{i(1)}), \varGam(\delta_{i(1)}), \ldots, \varGam(\alpha_{i(k)}), \varGam(\beta_{i(k)}), \varGam(\gamma_{i(k)}), \varGam(\delta_{i(k)}), \ldots\right)
\]
is $D_{0}$-semi-aligned.
\end{obs}

Note that $(w_{i})_{i}$'s and $(\alpha_{i}, \beta_{i}, \gamma_{i}, \delta_{i})_{i >0}$ are independent, the latter being i.i.d. with the uniform distribution on $S^{4}$. By Corollary \ref{cor:pivotDistbn}, $P_{k}$ linearly increases: 

\begin{obs}\label{obs:pivotRandom2}There exists $K>1$ such that
\begin{equation}\label{eqn:pivotInnocentGrowth}
\Prob\Big(\# P_{k}(\w) \le k/K  \,\Big| \,w_{0}, w_{1}, \ldots \Big) \le K e^{-k/K}
\end{equation}
for every $k>0$ and every choice of $(w_{i})_{i}$.
\end{obs}
Here, the growth rate is independent of $g_{i}$'s that are not involved in $(\alpha_{i}, \beta_{i}, \gamma_{i}, \delta_{i})$'s. In particular, it is independent of $g_{1}, \ldots, g_{\lfloor n/2 \rfloor + 1}$ and $g_{n+1}$.

To couple the words $w_{k, 2}^{-}$'s and the actual random walk $Z_{n}$'s, we need to control $\stopping(i)$'s. For each $k \ge \lfloor n/4M_{0}\rfloor$ and $L>0$, we have \[
\begin{aligned}
\Prob( \sumRho(k) \le L) \cdot e^{-L} &\le \E[e^{-\sumRho(k)}]  = \prod_{i=1}^{k} \E[ \operatorname{exp}(-\rho_{i}^{(n)})] & (\because \textrm{Markov's inequality})\\
&= \prod_{n/8M_{0} < i \le k, i \neq \lfloor n/4M_{0} \rfloor} \E[ \operatorname{exp}(-\rho_{i}^{(n)})]  \\
&= \left( 1 - \alpha(1-e^{-\alpha}) \right)^{k - \lceil n/8M_{0} \rceil - 1}\\
&\le (1- \alpha^{2}/2)^{n/3M_{0}-4}. & (\because \textrm{$e^{-\alpha} \le \alpha/2$ for $0 \le \alpha \le 1$})
\end{aligned}
\]
By plugging in $L = \frac{\log(1+\alpha^{2}/2)}{3M_{0}} k$, we obtain \begin{equation}\label{eqn:gouezelSumRho1}
\Prob\left(\sumRho(k) \le k/K' \right) \le K' e^{-k/K'} \quad(k \ge n/4M_{0})
\end{equation}
for some $K' =K'(\alpha, M_{0})>1$ (independent of $\alpha$). 

Let us now combine the ingredients and prove Proposition \ref{prop:gouezelRW1}. Given the measure $\mu$, integers $k \ge n$, and the choices of $g_{1}, \ldots, g_{\lfloor n/2 \rfloor + 1}$ and $g_{n+1}$, we do the above construction. Then we have \[
\Prob\left( A_{k} := \big\{\sumRho(\lfloor k/4M_{0} \rfloor) \le \lfloor k/4K' M_{0}\rfloor \big\}\right) \le K' e^{-\lfloor k/4K'M_{0}\rfloor}.
\]
Let us fix a combination of the values of $\{\rho_{i}^{(n)} : i > 0\}$ in $A_{k}^{c}$, which determines $\vartheta(i)$'s. Furthermore, we fix a combination of the values of $\{\eta_{i} : i>0\}$. These choices determine
 \begin{equation}\label{eqn:condition1Model}
\big\{g_{j} : j \notin \cup_{i> 0} \{4M_{0} \vartheta(i)+1, \ldots, 4M_{0} \vartheta(i)+4M_{0} \} \big\}.
\end{equation}
and consequently $w_{i}$'s. Note also that $\{(\alpha_{i}, \beta_{i}, \gamma_{i}, \delta_{i}) = \eta_{\vartheta(i)} : i>0\}$ are i.i.d.s distributed according to $\mu_{S}^{4}$. Hence, conditioned on choices of $\{\rho_{i}^{(n)} : i > 0\} \in A_{k}^{c}$ and  $\{\eta_{i} : i>0\}$, we are now reduced to the combinatorial model. From Observation \ref{obs:pivotRandom2}, we deduce that \[\begin{aligned}
&\Prob\Big( \# P_{l}(\w) \ge \lfloor k/4KK'M_{0}\rfloor \,\textrm{ for all $l \ge \lfloor k/4K'M_{0} \rfloor$}\, \Big| \, w_{0}, w_{1}, \ldots \Big)\\
 &\ge 1 - K \sum_{l \ge \lfloor k / 4K'M_{0} \rfloor } e^{-l/K} \ge 1 - \frac{K}{1-e^{-1/K}} e^{-\lfloor k/4KK'M_{0} \rfloor}.
 \end{aligned}
\]
In other words, except for probability $\frac{K}{1-e^{-1/K}} e^{-\lfloor k/4KK'M_{0} \rfloor}$ (under the conditioning), the initial $\lfloor k/4KK'M_{0} \rfloor$-sections of $P_{\lfloor k/4K'M_{0} \rfloor}(\w)$ remains the same in $P_{l}(\w)$ for $l \ge \lfloor k/4M_{0} \rfloor$. Hence, it becomes an initial section of $\liminf_{l} P_{l}(\w)$. This means that \begin{equation}\label{eqn:whatObs}
\#\Big( \mathcal{P}(\w) \cap \{ 4M_{0} \vartheta(i) + 2M_{0} : i \in P_{\lfloor k/4K'M_{0} \rfloor}(\w) \} \Big) \ge \lfloor k/4KK'M_{0} \rfloor.
\end{equation}
Meanwhile, since $\{\rho_{i}^{(n)} : i > 0\}$ is determined in $A_{k}^{c}$, we have $\sumRho(\lfloor k/4M_{0} \rfloor) > \lfloor k/4K'M_{0}\rfloor$ and \[
P_{\lfloor k/4K' M_{0} \rfloor}(\w) \subseteq \{\vartheta(1), \ldots, \vartheta(\lfloor k/4K'M_{0} \rfloor)\} \subseteq \{1, \ldots,  \lfloor k/4M_{0} \rfloor -1\}.
\]
This implies $\{4M_{0} \vartheta(i) + 2M_{0} : i\in P_{\lfloor k/4K'M_{0} \rfloor}(\w) \}\subseteq \{1, \ldots, k - 2M_{0}\}$. Combines with Display \ref{eqn:whatObs}, this implies  \begin{equation}\label{eqn:whatObs2}
\#\Big( \mathcal{P}(\w) \cap \{ 1, \ldots, k\} \Big) \ge \lfloor k/4KK'M_{0} \rfloor.
\end{equation}
Summing up the conditional probabilities, we have \[
\Prob\Big(\#\big( \mathcal{P}(\w) \cap \{ 1, \ldots, k\} \big) \ge \lfloor k/4KK'M_{0} \rfloor \, \Big| A_{k}^{c} \Big) \ge 1 - \frac{K}{1-e^{-1/K}} e^{-\lfloor k/4KK'M_{0} \rfloor}.
\]
Since $\Prob(A_{k})$ decays exponentially, we conclude Inequality \ref{eqn:gouezelRWDecay}.

It remains to partition the probability space $\Omega$ into pivotal equivalence classes that satisfy Definition \ref{dfn:pivotalEquiv}, with $\mathcal{P}(\w)$ as the set of pivotal times.  We say that $\bar{\w} \in \Omega$ is pivoted from $\w$ if they only differ in the value of $\beta_{i}$'s for $i \in \liminf_{l} P_{l}(\w)$. Then being pivoted from each other is an equivalence relation. On an equivalence class $\mathcal{E}$, all random paths have the same set of pivotal times $\diffPivot(\mathcal{E}) = \{j(1) < j(2) < \ldots \} \subseteq M_{0}\Z$ that avoids $1, \ldots, \lfloor n/2 \rfloor$ and $n$. Moreover, the steps $g_{i}$'s are uniform across $\mathcal{E}$ except for \[
s_{k} := \big(g_{j(k)-M_{0}+1},\, g_{j(k) -M_{0}+ 2},\,\ldots, \,g_{j(k)}\big) \quad (k = 1, 2, \ldots),
\] which are i.i.d.s chosen from $S$ according to $\mu_{S}$. Lastly, observe that \[\begin{aligned}
\varGam(\beta_{i(k)}) &= (Z_{(4M_{0} + M') \stopping(i(k)) + M_{0}} o, \ldots, Z_{(4M_{0} + M') \stopping (i(k)) +  2M_{0}} o) \\
&= \big(Z_{j(k) - M_{0}} o, \ldots, Z_{j(k)} o\big) = \axes_{j(k)}.
\end{aligned}
\]
By Observation \ref{obs:pivotRandom1}, $(o, \mathbf{Y}_{j(1)}, \mathbf{Y}_{j(2)}, \ldots)$ is always $D_{0}$-semi-aligned. Proposition \ref{prop:gouezelRW1} is now proved.

\section{Large deviation principles} \label{section:LDP}

In this section, we consider a more delicate pivoting that leads to the large deviation principle. Definition \ref{dfn:pivotalEquivLDP} and Proposition \ref{prop:gouezelRWLDP} rephrases Gou{\"e}zel's result in \cite[Section 5A]{gouezel2022exponential} in terms of strongly contracting isometries.

\begin{definition}\label{dfn:pivotalEquivLDP}
Let $\mu$ and $\nu$ be non-elementary probability measures on $G$ and $(\Omega, \Prob)$ be a probability space for $\mu$. Let $0<\epsilon<1$, let $K_{0}, N > 0$ and let $S \subseteq (\supp \mu)^{M_{0}}$ be a long enough and large $K_{0}$-Schottky set for $\mu$.

A subset $\mathcal{E}$ of $\Omega$ is called an \emph{$(n, N, \epsilon, \nu)$-pivotal equivalence class for $\mu$}, associated with the \emph{set of pivotal times} \[
\diffPivot^{(n, N, \epsilon, \nu)}(\mathcal{E}) = \big\{j(1) < j'(1) <  \ldots < j(\#\diffPivot/2) < j'(\#\diffPivot/2)\big\} \subseteq M_{0} \Z_{>0},
\]  if the following hold: \begin{enumerate}
\item for each $\w \in \mathcal{E}$ and $k \ge 1$, \[\begin{aligned}
s_{k}(\w) &:= \big(g_{j(k) - M_{0} + 1}(\w), \,\,g_{j(k) - M_{0} + 2}(\w), \,\, \ldots, \,\, g_{j(k)}(\w) \big),\\
s_{k}'(\w) &:=  \big(g_{j'(k) - M_{0} + 1}(\w), \,\,g_{j'(k) - M_{0} + 2}(\w), \,\, \ldots, \,\, g_{j'(k)}(\w) \big)
\end{aligned}
\]
are Schottky sequences;
\item for each $\w \in \mathcal{E}$, $\big(o, \axes_{j(1)}, \axes_{j'(1)}, \ldots, \axes_{j(\#\diffPivot/2)}, \axes_{j'(\#\diffPivot/2)}, Z_{n} o\big)$ is $D_{0}$-semi-aligned;
\item for the RV defined as \[
r_{k} := g_{j(k) + 1} g_{j(k) + 2} \cdots g_{j'(k) - M_{0}},
\]
$(s_{k}, s_{k}', r_{k})_{k > 0}$ on $\mathcal{E}$ are i.i.d.s and $r_{k}$'s are distributed almost according to $\mu^{\ast2 M_{0} N} \ast \nu^{\ast \frac{j'(k) - j(k)}{2M_{0}} - N-0.5}$ in the sense that the following holds for every $g \in G$:\[
 (1-\epsilon) \big(\mu^{\ast 2M_{0} N} \ast \nu^{\ast \frac{j'(k) - j(k)}{2M_{0}} - N-0.5}\big)(g) \le  \Prob(r_{k} = g) \le (1+\epsilon) \big(\mu^{\ast 2M_{0} N} \ast \nu^{\ast\frac{ j'(k) - j(k)}{2M_{0}} -N-0.5}\big)(g)
\]
for each $g \in G$.
\end{enumerate}
\end{definition}

\begin{prop}\label{prop:gouezelRWLDP}
Let $M_{0}> 0$, $\mu$ be a non-elementary probability measure on $G$, let $0<\epsilon < 1$ and let $S \subseteq (\supp \mu)^{M_{0}}$ be a long enough and large Schottky set for $\mu$ with cardinality greater than $100/\epsilon$. Then there exists a non-elementary probability measure $\nu$ on $G$ such that the following holds.

For each sufficiently large integer $N$, there exists $K>0$ such that for each $n$ we have a probability space $(\Omega, \Prob)$ for $\mu$ and its measurable partition $\mathscr{P}_{n, N, \epsilon, \nu} = \{\mathcal{E}_{\alpha}\}_{\alpha}$ into $(n, N, \epsilon, \nu)$-pivotal equivalence classes that satisfies\begin{equation}\label{eqn:gouezelRWLDP}
\Prob\left( \w :\frac{1}{2}\# \diffPivot^{(n, N, \epsilon, \nu)}(\w) \le (1-\epsilon) \frac{n}{2M_{0} N} \right) \le K e^{-n/K}.
\end{equation}
\end{prop}

We will in fact prove a statement that is more explicit than Proposition \ref{prop:gouezelRWLDP}:

\begin{prop}\label{prop:gouezelRWLDPAdv}
Let $0<\epsilon < 1$, let  $K_{0}, M_{0} > 0$ and let $S \subseteq G^{M_{0}}$ be a long enough and large $K_{0}$-Schottky set with $\#S \ge 100/\epsilon$. Let $\mu$ be a probability measure on $G$ such that $m := \min \{ \mu^{M_{0}}(s) : s \in S\}$ is positive. Let $N > 40/m^{2} \epsilon$ and let $\nu$ be the measure defined by \[
\nu = \frac{1}{1-0.5m^{2}} \big( \mu^{\ast 2M_{0}} - 0.5m^{2} \cdot (\textrm{uniform measure on $\{\Pi(s) \Pi(s') : s, s' \in S\}$})\big).
\]
Then $\nu$ is a non-elementary probability measure. Moreover, there exists $K >0$ depending only on $S$, $m, N$ and $\epsilon$ (but not on $\mu$) such that, for each $n$, we have a probability space $(\Omega, \Prob)$ for $\mu$ and its measurable partition $\mathscr{P}_{n, N, \epsilon, \nu} = \{\mathcal{E}_{\alpha}\}_{\alpha}$ into $(n, N, \epsilon, \nu)$-pivotal equivalence classes, associated with the set of pivotal times $\mathcal{P}^{(n, N, \epsilon, \nu)}$, that satisfies\[
\Prob\left( \w :\frac{1}{2}\# \diffPivot^{(n, N, \epsilon, \nu)}(\w) \le (1-\epsilon) \frac{n}{2M_{0} N} \right) \le K e^{-n/K} \quad (\forall n \in \Z_{>0}).
\]
\end{prop}

Gou{\"e}zel proved Proposition \ref{prop:gouezelRWLDP} for random walks on a Gromov hyperbolic space in \cite[Section 5C]{gouezel2022exponential}. We adapt his proof to our setting here.

\begin{proof}
Let us denote the uniform measure on $S$ by $\mu_{S}$. In this proof, when a probability measure $\tau$ on $G^{k}$ is given, we denote by $\tau^{\ast}$ the pushforward measure by convolution: \[
\tau^{\ast}(g) := \sum_{(g_{1}, \ldots, g_{k}) \in G^{k}, \,\,g_{1} \cdots g_{k} = g} \tau(g_{1}, \ldots, g_{k}).
\]

Let $N_{0} = \#S$ be the cardinality of $S$. Note that $10/N_{0} \le \epsilon/10$. Consider the decomposition \begin{equation}\label{eqn:3rdDecomp}
\mu^{2M_{0}} = 0.5 m^{2} \mu_{S}^{2} + (1- 0.5m^{2}) \tau,
\end{equation}
where $\tau$ is a probability measure on $G^{2M_{0}}$ with $\tau^{\ast} = \nu$. Recall that $S$ is a long enough and large $K_{0}$-Schottky set, so there exists $a, b \in S$ such that $\Pi(a)$ and $\Pi(b)$ are independent strongly contracting isometries. Since $\tau$ has the same support with $\mu^{2M_{0}}$, $\nu$ puts nonzero weights on $a^{2}$ and $b^{2}$. Hence $\nu$ is non-elementary.

Given the decomposition as in Equation \ref{eqn:3rdDecomp}, we consider Bernoulli RVs $(\rho_{i})_{i\ge0}$ with $\Prob(\rho_{i} = 1) = 0.5m^{2}$ and $\Prob(\rho_{i} = 0) = 1-0.5m^{2}$, $(\eta_{i})_{i\ge0}$ with the law of $\mu_{S}^{2}$, $(\tau_{i})_{i}$ with the law of $\tau$ and $(\xi_{i})_{i\ge 0}$ with the law of $\mu^{2M_{0}}$, all independent. We define RVs $\{t_{j}, t_{j}'\}_{j =1}^{\infty}$. First, $t_{1}$ is the smallest $i>0$ with $\rho_{i} = 1$, and $t_{1}' := \min \{ i > t_{1} + N : \rho_{i} = 1\}$. Inductively, we define \[
t_{k} := \min \{i > t_{k-1}' : \rho_{i} = 1\}, \quad t_{k}' := \min\{i > t_{k} + N : \rho_{i} = 1\}.
\]
For convenience, we set $t_{0}' := 0$. We then define \[
(g_{2M_{0}(k-1) + 1}, \,\ldots,\, g_{2M_{0}(k-1) + 2M_{0}}) := \left\{ \begin{array}{cc} \eta_{k} & \textrm{when}\,\, k \in \{t_{j}, t_{j}'\}_{j=1}^{\infty} \\ \xi_{k} & \textrm{when}\,\, t_{j} +1 \le k \le t_{j} + N\,\, \textrm{for some} \,\, j \\
\tau_{k} & \textrm{otherwise}.\end{array}\right.
\]
Then $(g_{i})_{i=1}^{\infty}$ is distributed according to the product measure $\mu^{\infty}$ \cite[Claim 5.11]{gouezel2022exponential}. 
We let  $\sumRho(k) :=\#\{j \ge 1 : t_{j}' < k\}$. Now define \[\begin{aligned}
w_{i-1} &:= g_{2M_{0}t_{i-1}' + 1} \cdots g_{2M_{0}(t_{i}- 1)},\\
\alpha_{i} &:= (g_{2M_{0} t_{i} - 2M_{0} + 1}, \,\ldots, \,g_{2M_{0} t_{i}-M_{0}} ),\\
\beta_{i} &:= (g_{2M_{0} t_{i} - M_{0}+1}, \,\ldots, \,g_{2M_{0} t_{i}} ),\\
v_{i} &:= g_{2M_{0} t_{i} +1} \cdots g_{2M_{0} t_{i}' - 2M_{0}},\\
\gamma_{i} &:= (g_{2M_{0} t_{i}'  - 2M_{0}+1},\, \ldots,\, g_{2M_{0}t_{i}' -M_{0}}),\\
\delta_{i} &:= (g_{2M_{0} t_{i}'- M_{0}+1}, \,\ldots, \,g_{2M_{0}t_{i}' } )
\end{aligned}
\]
for $i=1, \ldots, \sumRho(\lfloor n/2M_{0} \rfloor)$ and define $w_{\sumRho(\lfloor n/2M_{0} \rfloor)} = g_{2M_{0}t_{\sumRho(\lfloor n/2M_{0} \rfloor)}' + 1} \cdots g_{n}$. Using these data, we define the set of pivotal times \[
P_{\sumRho(\lfloor n/2M_{0} \rfloor)}(\w) = P_{\sumRho(\lfloor n/2M_{0} \rfloor)}\left((\alpha_{i}, \beta_{i}, \gamma_{i}, \delta_{i})_{i=1}^{\sumRho(\lfloor n/2M_{0} \rfloor)}; (w_{i})_{i=0}^{\sumRho(\lfloor n/2M_{0} \rfloor)}, (v_{i})_{i=1}^{\sumRho(\lfloor n/2M_{0} \rfloor)}\right)
\]
as in Subsection \ref{subsection:pivotDiscrete}.

We first determine the values of $\rho_{j}$'s. Observe that $\sumRho(\lfloor n/2M_{0} \rfloor)$ and $\{t_{j}, t_{j}'\}_{j}$ depend solely on $\{\rho_{j}\}_{j}$ and counts the renewal times in $[0, n/2M_{0}]$ formed with a geometric distribution after a delay $N$. More explicitly, if we `omit' $\rho_{t_{k} + i}$'s for $k > 0$ and $i = 1, \ldots, N$ and define\[
(\rho_{1}', \rho_{2}', \rho_{3}', \ldots) := ( \rho_{1}, \ldots, \rho_{t_{1}}, \rho_{t_{1} + N + 1}, \rho_{t_{1} + N+ 2}, \ldots, \rho_{t_{2}}, \rho_{t_{2} + N+1}, \ldots),
\]
then $\{\rho_{i}'\}_{i}$ are i.i.d. Bernoulli RVs and $t_{k}' = kN + \min\{ j : \sum_{i=1}^{j} \rho_{i}' = 2k\}$. Hence, we have \[\begin{aligned}
&\Prob \left(\sumRho(\lfloor n/2M_{0} \rfloor) < (1-\epsilon/10) \frac{n}{2M_{0} N} \right)= \Prob \left(t_{ \lceil(1-\epsilon/10) \frac{n}{2M_{0} N}\rceil }' >\lfloor n/2M_{0}\rfloor \right) \\
&= \Prob \left( \left\lceil(1-\epsilon/10) \frac{n}{2M_{0} } \right\rceil+ \min\left\{ j : \sum_{i=1}^{j} \rho_{i}' = \left\lceil(1-\epsilon/10) \frac{n}{M_{0} N}\right\rceil\right\} \ge\left\lfloor \frac{n}{2M_{0}}\right\rfloor\right) \\
&\le \Prob\left(\sum_{i=1}^{\lceil\epsilon n/20M_{0}\rceil + 3} \rho_{i}' < (1-\epsilon/10) \frac{n}{M_{0} N}\right),
\end{aligned}
\]
which decays exponentially because  $\E[\rho_{i}'] = 0.5 m^{2}> 20/ \epsilon N$. Hence, there exists $K_{1}>0$ that depends on $m, \epsilon$ and $N$ such that:
 \begin{equation}\label{eqn:3rdModel1st}
\Prob\left(\sumRho(\lfloor n/2M_{0} \rfloor)< \left(1-\epsilon / 10 \right) \frac{n}{2M_{0}N} \right) \le K_{1}e^{-n/K_{1}}
\end{equation} 
Let us fix the choices of $(\rho_{i})_{i\ge 0}$. This determine $(t_{i}, t_{i}')_{i>0}$ and $\sumRho(\lfloor n/2M_{0} \rfloor)$. We then fix the data $(\tau_{i}, \xi_{i})_{i>0}$ and $\{\eta_{i}: i > t_{\sumRho(\lfloor n/2M_{0} \rfloor)}'\}$. These in turn determine   $(w_{i})_{i=0}^{\sumRho(\lfloor n/2M_{0} \rfloor)}$ and $(v_{i})_{i=1}^{\sumRho(\lfloor n/2M_{0} \rfloor)}$. Furthermore, \[
(\alpha_{i}, \beta_{i})_{i=1}^{\sumRho(\lfloor n/2M_{0} \rfloor)} = (\eta_{t_{i}})_{i=1}^{\sumRho(\lfloor n/2M_{0} \rfloor)}, \quad (\gamma_{i}, \delta_{i})_{i=1}^{\sumRho(\lfloor n/2M_{0} \rfloor)} = (\eta_{t'_{i}})_{i=1}^{\sumRho(\lfloor n/2M_{0} \rfloor)}
\]
are all independent and identically distributed according to $\mu_{S}^{2}$. Hence, the situation is reduced to the combinatorial model in Section \ref{section:pivoting}. Corollary \ref{cor:pivotDistbn} asserts the following for some $K_{2} >0$: \begin{equation}\label{eqn:3rdModel2nd}
\Prob\Big(\#P_{\sumRho(\lfloor n/2M_{0} \rfloor)} \le (1-10/N_{0}) \sumRho(\lfloor n/2M_{0} \rfloor)\, \Big| \, (\rho_{i}, \tau_{i}, \xi_{i})_{i\ge 0} \Big) \le K_{2}e^{- \sumRho(\lfloor n/ 2M_{0} \rfloor)/2M_{0} K_{2}}.
\end{equation}
Combining Inequality \ref{eqn:3rdModel1st} and \ref{eqn:3rdModel2nd}, we can conclude that $\Prob\left( \# P_{\sumRho(\lfloor n/2M_{0} \rfloor)} \le (1-\epsilon) \frac{n}{2M_{0} N} \right)$ decays exponentially.

Now, given $\w \in \Omega$ with $P_{\sumRho(\lfloor n/2M_{0} \rfloor)(\w) } (\w) = \{i(1) < i(2) < \ldots\}$, we define: \[
\begin{aligned}
\mathcal{P}^{(n, N, \epsilon, \nu)}(\w) &= \{j(1) < j'(1) < j(2) < j'(2) < \ldots\} \\
&:= \Big\{2M_{0} t_{i(1)}, \,2M_{0} t_{i(1)}' - M_{0}, \,2M_{0} t_{i(2)},\, 2M_{0} t_{i(2)}' - M_{0}, \ldots\Big\}.
\end{aligned}
\]
We just established the estimate is Display \ref{eqn:gouezelRWLDP} for this $\mathcal{P}^{(n, N, \epsilon, \nu)}$. Furthermore, note that \[\begin{aligned}
\varGam(\beta_{i(k)}) &= (Z_{2M_{0} t_{i(k)} - M_{0}} o, \ldots, Z_{2M_{0} t_{i(k)}} o) = (Z_{j(k) - M_{0}} o, \ldots, Z_{j(k)} o) = \mathbf{Y}_{j(k)},\\
\varGam(\gamma_{i(k)}) &= (Z_{2M_{0} t'_{i(k)} - 2M_{0}} o, \ldots, Z_{2M_{0} t'_{i(k)} - M_{0}}o) = (Z_{j'(k) - M_{0}} o, \ldots, Z_{j'(k)} o) = \mathbf{Y}_{j'(k)},
\end{aligned}
\]
are Schottky axes, and that $w_{\sumRho(\lfloor n/2M_{0} \rfloor)+1, 2}^{-} = Z_{n}$. Proposition \ref{prop:extremal} tells us that $(o, \mathbf{Y}_{j(1)}, \mathbf{Y}_{j'(1)},$ $ \mathbf{Y}_{j(2)}, \mathbf{Y}_{j'(2)}, \ldots, Z_{n} o)$ is always $D_{0}$-semi-aligned. This settles Item (i) and (ii) in Definition \ref{dfn:pivotalEquivLDP}.

It remains to realize the partition as in Definition \ref{dfn:pivotalEquivLDP} and check Item (iii) in Definition \ref{dfn:pivotalEquivLDP}. We declare the equivalence by pivoting. More precisely, given $\w \in \Omega$ with $P_{\sumRho(\lfloor n/2M_{0} \rfloor}(\w)= \{i(1) <i(2) < \ldots\}$, we declare that another element $\w' \in \Omega$ is equivalent to $\w$ if it has the same values of $(\rho_{i})_{i \ge 0}$ (hence the same values of $(t_{i}, t_{i}')_{i> 0}$) with $\w$, and if it has the same values of $(\eta_{i}, \tau_{i}, \xi_{i})_{i \ge 0}$ with $\w$, possibly except for \[
\Big\{ \eta_{i} : i  \in \cup_{k} \big\{ t_{i(k)}, t_{i(k)}' \big\}\Big\}, \,\, \Big\{ \xi_{i} : i \in \cup_{k} [t_{i(k)} + 1, t_{i(k)} + N] \Big\}, \,\, \Big\{ \tau_{i} : i \in \cup_{k} [t_{i(k)} + N, t_{i(k+1)}' - 1] \Big\}.
\]
Further, we require that $\big(\alpha_{i(k)}(\w'), \delta_{i(k)}(\w') \big) = \big(\alpha_{i(k)}(\w), \delta_{i(k)}(\w) \big)$ for each $k$ and \[
\big(\beta_{i(l)}(\w'), \gamma_{i(l)}(\w'), v_{i(l)}(\w') \big) \in \tilde{S}\quad (l = 1, \ldots, \#P_{\sumRho(\lfloor n/2M_{0} \rfloor}).
\]Note that under this requirement, $\w'$ has the same values of $(w_{i})_{i=0}^{\sumRho(\lfloor n/2M_{0} \rfloor)}$ and $\{v_{i}: i \neq i(1), \ldots, i(\#P_{\sumRho(\lfloor n/2M_{0} \rfloor)})\}$ with $\w$. By Lemma \ref{lem:pivotEquiv}, we have $P_{\sumRho(\lfloor n/2M_{0} \rfloor)}(\w) = P_{\sumRho(\lfloor n/2M_{0} \rfloor)}(\w')$, and the above relation becomes an equivalence relation.

Recall that conditioned on the data $(\rho_{i})_{i \ge0}$, $(\beta_{j}, \gamma_{j}, v_{j})$ is distributed according to $\mu_{S}^{2} \times \big( \mu^{\ast 2M_{0}} \ast (\tau^{\ast})^{\ast( t_{j}' - t_{j}- N - 0.5)} \big) = \mu_{S}^{2} \times \big( \mu^{\ast 2M_{0}} \ast \nu^{\ast( t_{j}' - t_{j}- N - 0.5)} \big)$. Now let $\mathcal{E}$ be a pivotal equivalence class that has pivotal times $P_{\sumRho(\lfloor n/2M_{0} \rfloor)} = \{i(1) < \ldots < i(m)\}$. Then $(\beta_{i(l)}, \gamma_{i(l)}, v_{i(l)})$'s are independent and distributed according to the restriction of $\mu_{S}^{2} \times \left(\mu^{\ast 2M_{0} N} \ast \nu^{\ast( t_{j}' - t_{j}- N - 1)} \right) $ onto the set of ``legitimate choices'' $\pivotComplete$. To describe this, let us define a (not necessarily probability) measure \[
\mu^{(1)}(s', s'', r) :=  \left\{ \begin{array}{cc} \mu_{S}(s') \mu_{S}(s'') (\mu^{\ast 2M_{0} N} \ast \nu^{\ast (t_{j}' - t_{j} - N-1)} ) (r) & \textrm{if $(s', s'', r) \in \pivotComplete$,} \\ 0 & \textrm{otherwise.}\end{array}\right.
\]
then $(\beta_{i(l)}, \gamma_{i(l)}, v_{i(l)})$ is distributed according to the normalized version $\mu^{(0)}$ of $\mu^{(1)}$, namely, $\mu^{(0)}(A) := \frac{1}{\mu^{(1)}(S^{2} \times G)} \mu^{(1)}(A)$ for each $A \subseteq S^{2} \times G$.

For each $r \in G$, among $N_{0}^{2}$ choices of $s'$ and $s''$ in $S$ at least $N_{0}^{2} - 2N_{0}$ choices qualify the criterion and make $(s', s'', r) \in \pivotComplete$ by the Schottky property. (See the discussion in Display \ref{eqn:pivotingCond2} and \ref{eqn:pivotingCond3}.) This implies the bound for each $r \in G$:\[
 \left(1 - \frac{2}{N_{0}} \right) (\mu^{\ast 2M_{0} N} \ast \nu^{\ast (t_{j}' - t_{j} - N-1)} ) (r) \le \mu^{(1)}(S^{2} \times \{r\}) \le (\mu^{\ast 2M_{0} N} \ast \nu^{\ast (t_{j}' - t_{j} - N-1)} ) (r).
\]
Summing this up for all $r \in G$, we obtain $1-2/N_{0} \le \mu^{(1)}(S^{2} \times G) \le 1$. Combining these two estimates, we conclude the following for every $g \in G$:\[\begin{aligned}
 \left(1 - \frac{2}{N_{0}} \right) (\mu^{\ast 2M_{0} N} \ast \nu^{\ast (t_{j}' - t_{j} - N-1)} ) (r)& \le \Prob( v_{i(l)} = r) = \mu^{(0)} (S^{2} \times \{r\}) \\
 &\le \left( 1 +\frac{3}{N_{0}} \right) (\mu^{\ast 2M_{0} N} \ast \nu^{\ast (t_{j}' - t_{j} - N-1)} ) (r).
 \end{aligned}
\]
This settles Item (iii) in Definition \ref{dfn:pivotalEquivLDP} as desired.
\end{proof}

We now establish the large deviation principle for random walks.

\begin{thm}\label{thm:LDPStrong}
Let $(X, G, o)$ be as in Convention \ref{conv:strong} and let $(Z_{n})_{n}$ be the random walk generated by a non-elementary probability measure $\mu$ on $G$. Let $\lambda(\mu) = \lim_{n} \frac{1}{n} \E[d(o, Z_{n} o)]$ be the drift of $\mu$. Then for each $0 < L <\lambda(\mu)$, the probability $\Prob(d(o, Z_{n} o) \le Ln)$ decays exponentially as $n$ goes to infinity.
\end{thm}

Recall that $\lambda(\mu) = +\infty$ when $\mu$ has infinite first moment, by Corollary \ref{cor:SLLNInfinite}.

\begin{proof}
Due to the subadditivity, we have $\E_{\mu^{\ast N}} [d(o, go)] \ge \lambda(\mu) N$ for each $N>0$. Since $L$ is smaller than $\lambda(\mu)$, there exists $\epsilon > 0$ such that \[
(1-\epsilon)^{3}  \lambda(\mu)> L+\epsilon.
\] For this $\epsilon>0$, let $S$ be a long enough Schottky set for $\mu$ with cardinality greater than $100/\epsilon$. By Proposition \ref{prop:gouezelRWLDP}, there exists a non-elementary probability measure $\nu$, and for each sufficiently large $N$,  a partition $\mathscr{P}_{n, N, \epsilon}$ into $(n, N, \epsilon, \nu)$-pivotal equivalence classes for each $n$ such that \[
\Prob\left( \w : \frac{1}{2} \# \diffPivot^{(n, N, \epsilon, \nu)}(\w) \le (1- \epsilon) \frac{n}{2M_{0} N} \right)
\]decays exponentially in $n$. Let $C > 0$ be a constant for $\nu$ provided by Corollary \ref{cor:gouezelRW1Cor}: we have \[
\Prob_{\nu^{\ast m}}( h: d(o, gho) \ge d(o, go) - C) \ge 1-\epsilon/2
\]
for each $g \in G$ and each $m > 0$. We now fix an $N$ such that $N > \frac{C}{2M_{0} \lambda(\mu) \epsilon}$. 

Let $\mathcal{E}$ be an equivalence class such that $\frac{1}{2} \# \diffPivot^{(n, N, \epsilon, \nu)}(\mathcal{E}) \ge (1-\epsilon) \frac{n}{2M_{0} N}$. Then for each $\w \in \mathcal{E}$, $(o, \axes_{j(1)}, \ldots, \axes_{j'(\#\diffPivot/2)}, Z_{n} o)$ is $D_{0}$-semi-aligned. The second inequality in Item (ii) of Lemma \ref{lem:SchottkyAlign} tells us that \[\begin{aligned}
d(o, Z_{n} o) \ge \sum_{i=1}^{\#\diffPivot / 2} d(Z_{j(i)}o, Z_{j'(i) - M_{0}} o) = \sum_{i=1}^{\#\diffPivot / 2} d(o, r_{i} o) \quad (r_{i} := g_{j(i) +1} \cdots g_{j'(i) - M_{0}}).
\end{aligned}
\]
Since $r_{i}$'s are non-negative i.i.d. with \[
\E[d(o, r_{i} o)] \ge (1-\epsilon) \E_{\mu^{\ast 2M_{0} N}}[d(o, go) - C] \ge (1-\epsilon)^{2}  \cdot 2M_{0}N\lambda(\mu),
\]
we can apply the classical theory of large deviation. As a result, there exists $K'>0$ such that \[
\Prob\left( \left.d(o, Z_{n}o) \le(1-\epsilon)^{3} \lambda(\mu) n  \, \right| \, \mathcal{E}\right)\le K' e^{-n/K'} \quad (\forall n > 0).
\]
Summing up this conditional probability, we obtain the desired exponential bound.
\end{proof}

We now connect Theorem \ref{thm:LDPStrong} with the large deviation principle. In \cite[Proposition 2.3, Theorem 2.8]{boulanger2022large}, Boulanger, Mathieu, Sert and Sisto presented a general theory of large deviation principles on metric spaces with Schottky sets. Combining their result with Theorem \ref{thm:expBdMod}, we establish the large deviation principle for random walks on the mapping class group.

\begin{cor}[Large deviation principle] \label{cor:LDPMod}
Let $(X, G, o)$ be as in Convention \ref{conv:strong} and let $(Z_{n})_{n\ge0}$ be the random walk generated by a non-elementary probability measure $\mu$ on $G$. Then there exists a proper convex function $I : \mathbb{R} \rightarrow [0, +\infty]$, vanishing only at the drift $\lambda(\mu)$, such that  \[\begin{aligned}
- \inf_{x \in \operatorname{int}(E)} I(x) &\le \liminf_{n \rightarrow \infty} \frac{1}{n} \log \Prob\left( \frac{1}{n}d(id, Z_{n}) \in E\right), \\
 -\inf_{x \in \bar{E}} I(x) & \ge\limsup_{n \rightarrow \infty} \frac{1}{n} \log \Prob\left(\frac{1}{n} d(id, Z_{n}) \in E\right) 
\end{aligned}
\]
holds for every measurable set $E \subseteq \mathbb{R}$.
\end{cor}
We note the work of Corso \cite{corso2021large}, who proved that the rate function exists and is proper for random walks involving strongly contracting isometries. Our Corollary \ref{cor:LDPMod} strengthens Corso's result by showing that $I(x) \neq 0$ for $x \in [0, \lambda(\mu))$, which is a consequence of Theorem \ref{thm:expBdMod}.

\part{Random walks with weakly contracting isometries} \label{part:weak}

In this part, we deal with groups acting on a space $X$ and another space $\tilde{X}$ equivariantly, where the action on $X$ involves strong contraction and the action on $\tilde{X}$ involves weak contraction: see Convention \ref{conv:weak}. After studying alignment of weakly contracting directions in Section \ref{section:weakAlign}, we establish limit theorems for mapping class groups in Section \ref{section:limitHHG}.

\section{Mapping class groups and HHGs} \label{section:HHG}

Let $\Sigma$ be a finite-type hyperbolic surface, let $(\tilde{X}, \tilde{d})$ be the Cayley graph of the mapping class group $G = \Mod(\Sigma)$ of $\Sigma$, and let $(X, d)$ be the curve complex of $\Sigma$ or the Teichm{\"u}ller space of $\Sigma$. The action of $G$ on $(X, d)$ satisfies Convention \ref{conv:strong}: $G$ contains independent pseudo-Anosov mapping classes that have strongly contracting orbits on $X$ (\cite[Contraction Theorem]{minsky1996quasi-projections}, \cite[Proposition 4.6]{masur1999curve}).

Let $\Proj : \tilde{X} \rightarrow X$ be the orbit map: $\Proj(g) = go$, where $o \in X$ is the basepoint. Since $G$ is finite generated and acts on $X$ by isometries, the map $\Proj$ is coarsely Lipschitz and is $G$-equivariant. We will denote by $\tilde{A}$ the object ``in the upper space'' corresponding to an object $A$ ``in the lower space". For example, we fix basepoints $\tilde{o} = id \in \tilde{X}$ and $o \in X$ that satisfy $\Proj(\tilde{o}) = o$. 

For each subset $\tilde{A} \subseteq \tilde{X}$, we define the projection $\tilde{\pi}_{\tilde{A}}$ from $\tilde{X}$ onto $\tilde{A}$ by referring to the closest point projection at the lower space $X$. Namely, for $\tilde{x} \in \tilde{X}$ and its projection $x := \Proj(\tilde{x})$, we define $\tilde{\pi}_{\tilde{A}}(\tilde{x}) := \Proj^{-1} \circ \pi_{A} \circ \Proj$ by $\tilde{a} \in \tilde{\pi}_{\tilde{A}}(\tilde{x}) \,\,\Leftrightarrow \,\, a \in \pi_{A}(x).$

\begin{lem}\label{lem:strongWeakCouple}
For each $C>1$ there exists $D>1$ such that if a $C$-quasigeodesic $\tilde{\gamma} : I \rightarrow \tilde{X}$ on $\tilde{X}$ has projection $\gamma$ onto $X$ that is a $C$-contracting axis, then  $\tilde{\gamma}$ is $D$-weakly contracting with respect to the map $\tilde{\pi}_{\tilde{\gamma}} := \Proj^{-1} \circ \pi_{\gamma} \circ \Proj$.
\end{lem}

\begin{proof}
Let us first consider the case that $\tilde{X}$ is the Cayley graph of $\Mod(\Sigma)$ and $(X, d)$ is the curve complex $\mathcal{C}(\Sigma)$ of $\Sigma$. Recall that there are coarsely Lipschitz projections $\Proj_{U} : \tilde{X} \rightarrow \{\textrm{uniformly bounded subsets of $\mathcal{C}U$}\}$ from $\tilde{X}$ to the curve complex $\mathcal{C}U$ of subsurfaces $U \subseteq \Sigma$, and $\rho_{U}^{V} : \mathcal{C}U \rightarrow \{\textrm{uniformly bounded subsets of $\mathcal{C}V$}\}$ for every pair of nested subsurfaces $V \subseteq U \subseteq \Sigma$. Further, $\Proj_{U}$ and $\rho_{\Sigma}^{U} \circ \Proj$ are uniformly coarsely equivalent.

Since $\tilde{\gamma}$ and $\gamma = \Pr \circ \tilde{\gamma}$ are $C$-quasigeodesics, $\{ \Proj_{U}(\tilde{\gamma}) : U \subsetneq \Sigma\}$ have uniformly bounded diameter (depending on $C$). This is due to the bounded geodesic image property. Namely, given a proper subsurface $U \subsetneq \Sigma$, there exists a uniformly bounded  neighborhood $N$ of $\partial U \subseteq \mathcal{C}(\Sigma)$ such that $\gamma \setminus N$ is uniformly close to a geodesic on $\mathcal{C}(\Sigma)$ that is disjoint from $\partial U$. By \cite[Theorem 3.1]{masur2000curve}, $\rho_{\Sigma}^{U} (\gamma \setminus N)$ has bounded diameter. Since $N$ is bounded, $\rho_{\Sigma}^{U}(\gamma \cap N)$ is also bounded.

Given the uniform boundedness of $\Proj_{U}(\tilde{\gamma})$'s, i.e., the coboundedness of $\tilde{\gamma}$, the weakly contracting property of $\tilde{\gamma}$ follows from \cite[Theorem 4.2]{duchin2009divergence} (cf. \cite[Lemma 5.6]{behrstock2006asymptotic}). More explicitly, \cite[Theorem 4.2]{duchin2009divergence} guarantees a constant $E = E(C)$ such that, for each $\tilde{x} \in \tilde{X}$, we have \[
\diam_{X} \Big( \pi_{\gamma}\circ \Proj \Big(\big\{ \tilde{p} \in \tilde{X}: d(\tilde{x}, \tilde{p}) < \frac{1}{E} \tilde{d}(\tilde{x}, \tilde{\gamma}) \Big\}\Big)\Big) < E.
\]
Since the $d$-diameter along $\gamma$ and $\tilde{d}$-diameter along $\tilde{\gamma}$ are coarsely equivalent (as $\gamma$ is a quasigeodesic), we conclude that $\tilde{\gamma}$ is weakly contracting with respect to $\tilde{\pi}_{\tilde{\gamma}}$.

When $(X, d)$ is the Teichm{\"u}ller space of $\Sigma$, the strongly contracting property of $\gamma$ implies that the Teichm{\"u}ller geodesics $[\gamma(t), \gamma(s)]$ for $t<s$ are contained in a uniform neighborhood of $\gamma$ (Corollary \ref{cor:BGIPSelfNbd}) and are hence uniformly thick. This in turn implies that $\Pr_{U}(\tilde{\gamma}(t), \tilde{\gamma}(s))$'s for $t<s$ and proper subsurfaces $U \subsetneq \Sigma$ are uniformly bounded (\cite[Theorem 1.1]{rafi2005short}, \cite[Theorem 5.5]{rafi2014hyperbolicity}, \cite[Theorem 4.1]{rafi2009covers}, \cite[Lemma 5.1]{durham2015convex}). Then we similarly deduce the weakly contracting property of $\tilde{\gamma}$ by \cite[Theorem 4.2]{duchin2009divergence}.
\end{proof}

In general, Lemma \ref{lem:strongWeakCouple} can be generalized to the setting where $G$ is a hierarchically hyperbolic group (HHG), $(\tilde{X}, \tilde{d})$ is its Cayley graph and $(X, d)$ is the top curve graph for $G$. This follows from \cite[Corollary 6.2]{abbott2021largest} ((3) $\Rightarrow$ (2)) and \cite[Theorem 4.4]{abbott2021largest}. Note that even though Corollary 6.2 and Theorem 4.4 assumes the unbounded products of the HHG structure for $G$, which is not granted in general, the directions we need do not require such an assumption.

In particular, the pseudo-Anosov axes on $\tilde{X}$ are weakly contracting. Hence, our setting is:

\begin{conv}\label{conv:weak}
We fix $B>0$ and assume that:
\begin{enumerate}
\item $(\tilde{X}, \tilde{d})$, $(X, d)$ are geodesic metric spaces;
\item $\Proj : \tilde{X} \rightarrow X$ is a coarsely Lipschitz map, i.e., 
for all $\tilde{x}, \tilde{y} \in \tilde{X}$\[
d(\Proj(\tilde{x}), \Proj(\tilde{y})) \le B\tilde{d}(\tilde{x}, \tilde{y})+ B;
\]
\item $G$ is a countable group of isometries acting on $\tilde{X}$ and $X$ equivariantly;
\item $\tilde{o} \in \tilde{X}$ and $o \in X$ are basepoints that satisfy $\Proj(\tilde{o}) = o$;
\item For each $C > 1$ there exists $D> 1$ such that a path $\tilde{\gamma}$ on $\tilde{X}$ is $D$-weakly contracting with respect to $\tilde{\pi}_{\tilde{\gamma}}$ whenever its projection $\Proj(\tilde{\gamma})$ is a $C$-contracting axis;
\item $G$ contains two independent strongly contracting isometries of $X$.
\end{enumerate}
For each object $\tilde{A} \subseteq \tilde{X}$, we denote by $A$ its projection $\Proj(\tilde{A}) \subseteq X$.
\end{conv}

When Item (iii) is replaced with the coarse equivariance condition, this setting also covers HHGs acting on the top curve graph. For simplicity, we denote the word norm of $g \in G$ by $|g|$. In the general case, one can replace $|g|$ with $\tilde{d}(\tilde{o}, g \tilde{o})$.
\begin{comment}
The following lemma will be useful: 

\begin{lem}\label{lem:strongWeakCompare}
For each $C>1$ there exists $K'> C$ such that for each path  $\tilde{\kappa}$ on $\tilde{X}$ whose projection $\kappa$ onto $X$ is a $C$-quasigeodesic, we have \[
\frac{1}{K'} \tilde{d}(\tilde{x}, \tilde{y}) - K' \le d(x, y) \le K'\tilde{d}(\tilde{x}, \tilde{y}) + K'
\]
for all points $\tilde{x}, \tilde{y}$ on $\tilde{\kappa}$.
\end{lem}

\begin{proof}
Let $\tilde{\kappa} : I \rightarrow \tilde{X}$ and $\kappa := \Proj \circ \tilde{\kappa} : I \rightarrow X$. If we denote by $F$ the coarse inverse of $\kappa$, $\Proj$ and $\tilde{\kappa} \circ F$ are coarsely Lipschitz maps between $\kappa$ and $\tilde{\kappa}$ which are coarse inverses of each other.
\end{proof} 
\end{comment}

\section{Alignment II: weakly contracting axes} \label{section:weakAlign}

Throughout, we adopt Convention \ref{conv:weak}. We define the alignment among paths $\tilde{\kappa}_{1}, \ldots, \tilde{\kappa}_{n}$ on $\tilde{X}$ based on Definition \ref{dfn:alignment} \emph{with respect to the projections $\tilde{\pi}_{\tilde{\kappa}_{i}} := \Proj^{-1} \circ \pi_{\kappa_{i}} \circ \Proj$}. 

\begin{lem}\label{lem:alignCoupled}
For each $K > 1$ there exists $K' > K$ such that the following hold. Let $\tilde{x}, \tilde{y} \in \tilde{X}$ and let $\kappa$ be a path on $\tilde{X}$ whose projection $\kappa$ on $X$ is a $K$-contracting axis. Then $\tilde{\kappa}$ is a $K'$-quasigeodesic that is $K'$-weakly contracting with respect to $\tilde{\pi}_{\tilde{\kappa}}$. Moreover, for each $C>1$ we have the implication \[\begin{aligned}
\textrm{$(\tilde{x}, \tilde{\kappa}, \tilde{y})$ is $C$-aligned} \Rightarrow \textrm{$(x, \kappa, y)$ is $K'C$-aligned}, \\
\textrm{$(x, \kappa, y)$ is $C$-aligned} \Rightarrow \textrm{$(\tilde{x}, \tilde{\kappa}, \tilde{y})$ is $K'C$-aligned}.
\end{aligned}
\]
\end{lem}

\begin{proof}
Let $\tilde{\kappa} : I \rightarrow \tilde{X}$ and $\kappa := \Proj \circ \tilde{\kappa} : I \rightarrow X$. The weakly contracting property of $\tilde{\kappa}$ is given by Lemma \ref{lem:strongWeakCouple}. If we denote by $F$ the coarse inverse of $\kappa$, $\Proj$ and $\tilde{\kappa} \circ F$ are maps between $\kappa$ and $\tilde{\kappa}$, and are coarse inverses of each other. This implies the coarse comparison\[
\frac{1}{K''} \tilde{d}(\tilde{p}, \tilde{q}) - K'' \le d(p, q) \le K''\tilde{d}(\tilde{p}, \tilde{q}) + K''
\]
for all points $\tilde{p}, \tilde{q}$ on $\tilde{\kappa}$, for some $K'' = K''(K)$. This implies the remaining items.
\end{proof}

We now prove the main proposition of this section.

\begin{prop}\label{prop:weakBGIPConcat}
For each $K, D > 1$, there exist $E, L'> K, D$ such that the following holds.

Let $L \ge L'$, let $\tilde{x}, \tilde{y} \in \tilde{X}$ and let $\tilde{\kappa}_{1}, \ldots, \tilde{\kappa}_{n}$ be paths on $\tilde{X}$ whose domains are longer than $L$ and such that their projections are $K$-contracting axes. Suppose that $(x, \kappa_{1}, \ldots, \kappa_{n}, y)$ is $D$-aligned. Then there exist points $\tilde{p}_{1}, \ldots, \tilde{p}_{n}$ on $[\tilde{x}, \tilde{y}]$, in order from left to right, such that \begin{equation}\label{eqn:weakBGIPQuasi}
\tilde{d}(\tilde{p}_{i}, \tilde{\kappa}_{i}) \le \sum_{j=1}^{n+1} e^{-|j-i-0.5| L/ E} \diam_{\tilde{X}}(\tilde{\kappa}_{j-1} \cup \tilde{\kappa}_{j}) + E.
\end{equation}
Here, we plug in $\tilde{\kappa}_{0} = \tilde{x}$ and $\tilde{\kappa}_{n+1} = \tilde{y}$.
\end{prop}

\begin{proof}
Let $B$ be the coarse Lipschitzness constant for $\Proj$, and define the constants: \begin{itemize}
%\item let $K_{1} = K'(K)$ be as in Lemma \ref{lem:strongWeakCompare};
\item let $K_{2} = K'(K)$ be as in Lemma \ref{lem:alignCoupled}, which is larger than $K>1$;
\item let $K_{4} = K'(K)$ be as in Lemma \ref{lem:weakContractingConvex}, which is larger than $K>1$;
\item let $E_{1}= E(K, D)$ be as in Lemma \ref{lem:1segmentVar}, which is larger than $K>1$;
\item let $E_{2} = E(K, D)$, $L_{0} = L(K, D)$ be as in Proposition \ref{prop:BGIPConcat}.
\end{itemize}
Now we define constants \[
\begin{aligned}
E&= 16KK_{2} K_{4} (1 + \log 2K_{4}) +E_{1} + E_{2} + B\\
L' &= L_{0} + 4KK_{2}(B+5K+K_{2} E).
\end{aligned} 
\]Then the following hold for all $L \ge L'$: \[
\begin{aligned}
\frac{L}{K} - E - K &\ge L/2K \ge  E_{1} + 2E_{2} + B + 4K, \\
L &\ge 4KK_{2} \left( \frac{B + 5K}{K_{2}} + K_{2} E \right), \\
\frac{1}{2} e^{-2L/E} &\ge K_{4}e^{-\frac{L}{4KK_{2} K_{4}}}.
 \end{aligned}
\]
Note that
 \begin{fact}\label{fact:weakBGIPConcat}
Let $C>0$, let $x \in X$ and let $\kappa$ be a $K$-contracting axis on $X$ with $L$-long domain.   If $(x, \kappa)$ is $C$-aligned, then $(\kappa, x)$ is not $(\frac{L}{K} - C-K)$-aligned.
\end{fact}

Let $L \ge L'$, let $\tilde{x}, \tilde{y} \in \tilde{X}$ and let $\tilde{\kappa}_{1}, \ldots, \tilde{\kappa}_{n}$ be paths on $\tilde{X}$ whose domains are longer than $L$ and whose projections $\kappa_{i}$'s onto $X$ are $K$-contracting axes. Recall our convention that, whenever we define $\tilde{A} \subseteq \tilde{X}$, we use the notation $A := \Proj(\tilde{A})$.

\emph{Step 1}. We prove the following for $n \ge 2$:\begin{quote}
if $(x, \kappa_{1})$ is $E$-aligned and $(\kappa_{1}, \ldots, \kappa_{n}, y)$ is $D$-aligned, then $[\tilde{x}, \tilde{y}]$ has points $\tilde{z}_{2}, \ldots, \tilde{z}_{n}$, in order from left to right, such that $(\kappa_{i-1}, z_{i})$ and $(z_{i}, \kappa_{i})$ are $E$-aligned ($z_{i} = \Proj \tilde{z}_{i}$).
\end{quote}
We induct on the number $n$ of the contracting axes. First, Proposition \ref{prop:BGIPConcat} implies that $(x, \kappa_{i}, y)$ is $E_{2}$-aligned for each $i$. In view of Fact \ref{fact:weakBGIPConcat}, $(\kappa_{1}, x)$ is not $(E_{1} + E_{2} +  B+4K)$-aligned but $(\kappa_{1}, y)$ is $E_{2}$-aligned. Now note that: \begin{itemize}
\item $\pi_{\kappa_{1}}$ is $(1, 4K)$-coarsely Lipschitz (Lemma \ref{lem:projBdd}),
\item $\Proj$ is $B$-coarsely Lipschitz and hence $\pi_{\kappa_{1}}\circ \Proj$ is $(B, B+4K)$-coarsely Lipschitz, and
\item the geodesic $[\tilde{x}, \tilde{y}]$ is connected.
\end{itemize}
Pick the rightmost point $\tilde{z}_{2} \in [\tilde{x}, \tilde{y}]$ such that $(\kappa_{1}, z_{2})$ is not $(E_{1} + E_{2})$-aligned. Then $(\kappa_{1}, z_{2})$ is $(E_{1} + E_{2} + B+4K)$-aligned, and hence $E$-aligned, since $\pi_{\kappa_{1}} \circ \Proj$ is $B$-coarsely Lipschitz. Since $(\kappa_{1}, \kappa_{2})$ is $D$-aligned and  $(\kappa_{1}, z_{2})$ is not $E_{1}$-aligned, Lemma \ref{lem:1segmentVar} implies that $(z_{2}, \kappa_{1})$ is $E_{1}$-aligned.

When $n = 2$, the proof ends here. Otherwise, note that $(z_{2}, \kappa_{2})$ is $E$-aligned and $(\kappa_{2}, \ldots, \kappa_{n}, y)$ is $D$-aligned. By the induction hypothesis, there exist $\tilde{z}_{3}, \ldots, \tilde{z}_{n}$ on $[\tilde{z}_{2}, \tilde{y}]$, in order from left to right, such that $(\kappa_{i-1}, z_{i})$ and $(z_{i}, \kappa_{i})$ are $E$-aligned for each $i\ge 3$. The claim now follows. 

\emph{Step 2: Construction of $\tilde{p}_{j}$'s}.
We now assume that $(x, \kappa_{1}, \ldots, \kappa_{n}, y)$ is $D$-aligned. By Step 1, we obtain points $\tilde{z}_{2}, \ldots, \tilde{z}_{n}$ on $[\tilde{x}, \tilde{y}]$, in order from left to right. We let $\tilde{z}_{1} := \tilde{x}$ and $\tilde{z}_{n+1} := \tilde{y}$.

Pick $j \in \{1, \ldots, n\}$. Then $(z_{j}, \kappa_{j})$ is $E$-aligned, and hence $L/2K$-aligned. Meanwhile, $(\kappa_{j}, z_{j+1})$ is $E$-aligned, so $(z_{j+1}, \kappa_{j})$ is not $(L/K - E-K)$-aligned, and hence not $L/2K$-aligned. Now let $\tilde{p}_{j}$ to be the rightmost point on $[\tilde{z}_{j}, \tilde{z}_{j+1}]$ such that $(p_{j}, \kappa_{j})$ is not $L/2K$-aligned. Then by the $(B, B+4K)$-Lipschitzness of $\pi_{\kappa_{j}} \circ \Proj$, we have that \[
L/2K - (B+4K) \le \diam_{X} \big( \textrm{beginning point of $\kappa_{j}$} \cup \pi_{\kappa_{j}}(p_{j}) \big) \le L/2K +(B+4K).
\]
In particular, $(p_{j}, \kappa_{j})$ is not  $(L/2K - (B+4K))$-aligned. Moreover, $(\kappa_{j} ,p_{j})$ is not $(L/2K - (B+5K))$-aligned by Fact \ref{fact:weakBGIPConcat}. Denoting the beginning point of $\tilde{\kappa}_{j}$ by $q_{j}$, Lemma \ref{lem:alignCoupled} implies \[
\begin{aligned}
\diam_{\tilde{X}}\Big(\tilde{\pi}_{\tilde{\kappa}_{j}}(\tilde{z}_{j}) \cup \tilde{\pi}_{\tilde{\kappa}_{j}}(\tilde{p}_{j}) \Big) &\ge \diam_{\tilde{X}}\Big(\tilde{\pi}_{\tilde{\kappa}_{j}}(\tilde{z}_{j}) \cup q_{j} \Big) - \diam_{\tilde{X}}\Big(q_{j}\cup \tilde{\pi}_{\tilde{\kappa}_{j}}(\tilde{p}_{j}) \Big) \\
&\ge \frac{1}{K_{2}} \left(\frac{L}{2K} -B-4K\right) - K_{2} E \ge \frac{L}{4KK_{2}}.
\end{aligned}
\]
For a similar reason, $\diam_{\tilde{X}}\Big(\tilde{\pi}_{\tilde{\kappa}_{j}}(\tilde{z}_{j+1}) \cup \tilde{\pi}_{\tilde{\kappa}_{j}}(\tilde{p}_{j}) \Big)$ is at least $ \frac{1}{K_{2}} \left(\frac{L}{2K} -B-5K\right) - K_{2} E \ge \frac{L}{4KK_{2}}$.
Now Lemma \ref{lem:weakContractingConvex} implies \begin{equation}\label{eqn:pjDistEstimate} \begin{aligned}
\tilde{d}(\tilde{p}_{j}, \tilde{\kappa}_{j}) &\le K_{4} e^{-\frac{L}{4KK_{2}K_{4}}} \tilde{d}(\tilde{z}_{j}, \tilde{\kappa}_{j}) + K_{4} e^{-\frac{L}{4KK_{2}K_{4}}} \tilde{d}(\tilde{z}_{j+1}, \tilde{\kappa}_{j})  + K_{4} \\
&\le \frac{1}{2} e^{-2L/E}  \tilde{d}(\tilde{z}_{j}, \tilde{\kappa}_{j}) + \frac{1}{2}e^{-2L/E}   \tilde{d}(\tilde{z}_{j+1}, \tilde{\kappa}_{j}) +K_{4}.
\end{aligned}
\end{equation}

\emph{Step 3: Estimating $\tilde{d}(\tilde{p}_{j}, \tilde{\kappa}_{j})$}.

Given Inequality \ref{eqn:pjDistEstimate}, it now suffices to prove: \begin{equation}
\label{eqn:weakBGIPQuasiZ}
\tilde{d}(\tilde{z}_{i}, \tilde{\kappa}_{i-1}) + \tilde{d}(\tilde{z}_{i}, \tilde{\kappa}_{i}) \le \sum_{j=1}^{n+1} 2 e^{-|j-i| L/ E} \diam_{\tilde{X}}(\tilde{\kappa}_{j-1} \cup \tilde{\kappa}_{j}) + E
\end{equation}
for $i = 2, \ldots, n$. To prove this, we collect indices $i$  that violates Inequality \ref{eqn:weakBGIPQuasiZ}. Let $I = \{m, m+1, \ldots, m'\}$ be a maximal 1-connected set of such indices. We aim to show that $I$ is empty.

Suppose to the contrary that $I$ is nonempty. Note first that $\tilde{z}_{1}$ and $\tilde{z}_{n+1}$ satisfy Inequality \ref{eqn:weakBGIPQuasiZ}, i.e., $1, n+1 \notin I$; hence $m \ge 2$ and $m' \le n$. We now compute $\tilde{d}(\tilde{z}_{m-1}, \tilde{z}_{m'+1})$ in two different ways. First, using Inequality \ref{eqn:pjDistEstimate} we deduce
\[\begin{aligned}
\tilde{d}(\tilde{z}_{m-1}, \tilde{z}_{m'+1}) &= \sum_{j=m}^{m'+1} \tilde{d}(\tilde{z}_{j-1}, \tilde{z}_{j}) \ge \sum_{j=m}^{m'+1}\Big(\tilde{d}(\tilde{z}_{j-1}, \tilde{p}_{j-1}) + \tilde{d}(\tilde{p}_{j-1}, \tilde{z}_{j}) \Big)\\
&\ge \sum_{j=m}^{m' + 1} \Big(\tilde{d}(\tilde{z}_{j-1}, \tilde{\kappa}_{j-1}) + \tilde{d}( \tilde{\kappa}_{j-1}, \tilde{z}_{j}) -2 \tilde{d}(\tilde{p}_{j-1}, \tilde{\kappa}_{j-1})\Big) \\
&\ge \sum_{j=m}^{m' + 1}\left( (1-e^{-2L/E})\left(\tilde{d}(\tilde{z}_{j-1}, \tilde{\kappa}_{j-1}) + \tilde{d}(\tilde{\kappa}_{j-1}, \tilde{z}_{j})  \right) - 2K_{4}\right).
\end{aligned}
\]
Recall that $\tilde{d}(\tilde{\kappa}_{j-1}, \tilde{z}_{j}) + \tilde{d}(\tilde{z}_{j}, \tilde{\kappa}_{j}) \ge \sum_{k} 2e^{-|k - j| L/E} \diam_{\tilde{X}} (\tilde{\kappa}_{k-1} \cup \tilde{\kappa}_{k}) + E$ holds for $m \le j \le m'$. Moreover, $E \cdot (\# I) \ge 2K_{4}  \cdot (\# I + 1) + 0.5E$ because $\# I \ge 1$ and  $E \ge 8K_{4}$. Hence, we obtain \[\begin{aligned}
\tilde{d}(\tilde{z}_{m-1}, \tilde{z}_{m'+1}) &\ge (1-e^{-2L/E})\Big(\tilde{d}(\tilde{z}_{m-1}, \tilde{\kappa}_{m-1}) + \tilde{d}(\tilde{\kappa}_{m'}, \tilde{z}_{m'+1})\Big) \\
& \quad + (1-e^{-2L/E}) \sum_{j=m}^{m'} \sum_{k=1}^{N+1} 2e^{-|k - j| L/E} \diam_{\tilde{X}} (\tilde{\kappa}_{k-1} \cup \tilde{\kappa}_{k}) + (1-e^{-2L/E}) \cdot 0.5E.
\end{aligned}
\]
If we rearrange the double summation with respect to $k$, the RHS is at least \[\begin{aligned}
&(1-e^{-2L/E})\Big( \tilde{d}(\tilde{z}_{m-1}, \tilde{\kappa}_{m-1}) + \tilde{d}(\tilde{z}_{m'+1}, \tilde{\kappa}_{m'}) + 0.5E\Big) + 2(1-e^{-2L/E})\sum_{j= m}^{m'}  \diam(\tilde{\kappa}_{j-1} \cup \tilde{\kappa}_{j}) \\
&+2(1-e^{-2L/E})\bigg( \sum_{1 \le k < m}  e^{ -(m - k)L/E} \diam_{\tilde{X}} (\tilde{\kappa}_{k-1} \cup \tilde{\kappa}_{k}) + \sum_{m'<k \le N+1} e^{ -(k - m')L/E} \diam_{\tilde{X}} (\tilde{\kappa}_{k-1} \cup \tilde{\kappa}_{k}) \bigg).
\end{aligned}
\]
Next, we will obtain an upper bound of $\tilde{d}(\tilde{z}_{m-1} , \tilde{z}_{m'+1})$: \[\begin{aligned}
\tilde{d}(\tilde{z}_{m-1}, \tilde{z}_{m' + 1} ) &\le \tilde{d}(\tilde{z}_{m-1}, \tilde{\kappa}_{m-1}) + \sum_{j=m}^{m'} \diam(\tilde{\kappa}_{j-1} \cup \tilde{\kappa}_{j}) + \tilde{d}(\tilde{z}_{m'+1}, \tilde{\kappa}_{m'}) \\
&\le (1-e^{-2L/E}) \Big(\tilde{d}(\tilde{z}_{m-1}, \tilde{\kappa}_{m-1}) + \tilde{d}(\tilde{z}_{m'+1}, \tilde{\kappa}_{m'}) \Big) + \sum_{j=m}^{m'} \diam(\tilde{\kappa}_{j-1} \cup \tilde{\kappa}_{j})  \\
&\quad + 2e^{-2L/E} \cdot E + e^{-2L/E}  \Big( \tilde{d}(\tilde{\kappa}_{m-2}, \tilde{z}_{m-1}) + \tilde{d}(\tilde{z}_{m-1}, \tilde{\kappa}_{m-1}) - E \Big)\\
&\quad + e^{-2L/E}  \Big( \tilde{d}(\tilde{\kappa}_{m'}, \tilde{z}_{m'+1}) + \tilde{d}(\tilde{z}_{m'+1}, \tilde{\kappa}_{m'+1}) - E \Big).
\end{aligned}
\]
We know that $6e^{-L/E} \le 1$ because $L >8E$. Having this in mind, we now make use of the fact that $m-1$ and $m'+1$ is not contained in $I$: $\tilde{d}(\tilde{z}_{m-1}, \tilde{z}_{m' + 1} )$ is bounded from above by \[\begin{aligned}
 &\,\, (1-e^{-2L/E}) \Big(\tilde{d}(\tilde{z}_{m-1}, \tilde{\kappa}_{m-1}) + \tilde{d}(\tilde{z}_{m'+1}, \tilde{\kappa}_{m'}) + 0.5E\Big)+ \sum_{j=m}^{m'} \diam_{\tilde{X}}(\tilde{\kappa}_{j-1} \cup \tilde{\kappa}_{j})  \\
&\,\,+ e^{-2L/E} \left( 2\sum_{k=1}^{N+1} \left(e^{-|k - m+1|L/E} + e^{-|k-m'-1|L/E}\right) \diam_{\tilde{X}} (\tilde{\kappa}_{k-1} \cup \tilde{\kappa}_{k}) \right) \\
&\le (1-e^{-2L/E})\Big( \tilde{d}(\tilde{z}_{m-1}, \tilde{\kappa}_{m-1}) + \tilde{d}(\tilde{z}_{m'+1}, \tilde{\kappa}_{m'}) + 0.5E\Big)  \\
&\,\,+ \sum_{j= m}^{m'} (1 + 4 e^{-2L/E}) \diam_{\tilde{X}}(\tilde{\kappa}_{j-1} \cup \tilde{\kappa}_{j})  + \sum_{1 \le k < m} 4e^{-L/E} \cdot  e^{ -(m - k)L/E} \diam_{\tilde{X}} (\tilde{\kappa}_{k-1} \cup \tilde{\kappa}_{k}) \\
&\,\,+ \sum_{m'<k \le N+1} 4 e^{-L/E} \cdot e^{ -(k - m')L/E} \diam_{\tilde{X}} (\tilde{\kappa}_{k-1} \cup \tilde{\kappa}_{k}),
\end{aligned}
\]
which is a contradiction. Hence, $I = \emptyset$ and Inequality \ref{eqn:weakBGIPQuasiZ} is established.
\end{proof}

\section{Limit laws for mapping class groups} \label{section:limitHHG}
\begin{comment}
In Subsection \ref{subsection:path}, given a sequence $s = (\phi_{1}, \ldots, \phi_{k}) \subseteq G^{k}$, we defined its axis $\Gamma^{+}(s) := (o, \phi_{1} o, \ldots, \phi_{1} \phi_{2} \cdots \phi_{k} o)$ and $\Gamma^{-}(s) := (o, \phi_{k}^{-1} o, \ldots, \phi_{k}^{-1} \cdots \phi_{1}^{-1} o)$ on $X$. We define corresponding paths on $\tilde{X}$ also: \[
\begin{aligned}
 \tilde{\Gamma}^{+}(s) &:= (\tilde{o}, \,\,\phi_{1} \tilde{o}, \quad\phi_{1} \phi_{2} \tilde{o},\quad \quad\ldots, \,\,\phi_{1}\phi_{2} \cdots \phi_{k} \tilde{o}), \\
\tilde{\Gamma}^{-}(s) &:= (\tilde{o}, \,\,\phi_{k}^{-1} \tilde{o},\,\, \phi_{k}^{-1} \phi_{k-1}^{-1} \tilde{o},\,\, \ldots, \,\,\phi_{k}^{-1} \cdots \phi_{1}^{-1} \tilde{o}).
\end{aligned}
\]
\end{comment}

We continue to employ the notion of Schottky sets defined in Definition \ref{dfn:Schottky}. Once a Schottky set $S$ and its element $s$ is understood, the translates of $\Gamma^{\pm}(s)$ are now called \emph{Schottky axes on $X$}, whereas the translates of $\tilde{\Gamma}^{\pm}(s)$ are called \emph{Schottky axes on $\tilde{X}$}. 

\begin{definition}\label{dfn:SchottkyLongWeak}
Given a constant $K_{0} > 0$, we define: \begin{itemize}
\item $K_{1} = K'(K_{0})$ be as in Lemma \ref{lem:alignCoupled},
\item $D_{0 } = D(K_{0}, K_{0})$ be as in Lemma \ref{lem:1segment},
\item $E_{0} = E(K_{0}, D_{0})$, $L_{0} = L(K_{0}, D_{0})$ be as in Proposition \ref{prop:BGIPWitness}.
\item $E_{1} = K'(K_{0}, E_{0})$, $L_{1} = L'(K_{0}, E_{0})$ be as in Proposition \ref{prop:weakBGIPConcat}.
\end{itemize}
Let $0 < \epsilon < 1$. If a $K_{0}$-Schottky set $S \subseteq G^{M_{0}}$ consists of sequences of length \[
M_{0} > \max\big(L_{0}, L_{1}, 2K_{1}E_{1}, (-\log (\epsilon^{2}/4)) \cdot E_{1}\big),
\]
then we call $S$ an \emph{$\epsilon$-constricting $K_{0}$-Schottky set}.
\end{definition}

Thanks to Proposition \ref{prop:Schottky}, for every non-elementary probability measure $\mu$ on $G$ and $N, \epsilon > 0$, there exists an $\epsilon$-long enough Schottky set for $\mu$ with cardinality $N$. We are ready to state:

\begin{prop}\label{prop:weakDeviDist}
Let $\mu$ be a non-elementary probability measure on the mapping class group $G$ and $((\check{Z}_{n})_{n}, (Z_{n})_{n})$ be the (bi-directional) random walk generated by $\mu$, with step sequences $((\check{g}_{n})_{n}, (g_{n})_{n})$. Then there exists $K'>0$ such that \[
\Prob\left(\textrm{$d( id, [\check{Z}_{m}, Z_{n}]) \le K'D_{k}$ for all $n, m \ge 0$} \, \Big| \, \check{g}_{k+1}, g_{k+1} \right) \le K' e^{-k/K'}
\]
holds for all $k$, where  \begin{equation}\label{eqn:weakDeviDistDK}
D_{k} :=\sum_{i=1}^{k} |g_{i}|+\sum_{i=1}^{k} |\check{g}_{i}|+ \sum_{i=1}^{\infty} e^{-i/K'}  |g_{i}|+\sum_{i=1}^{\infty} e^{-i/K'} |\check{g}_{i}|+1.
\end{equation}
\end{prop}

\begin{proof}
Let $S$ and $\check{S}$ be a long enough, large and $(2/e)$-constricting $K_{0}$-Schottky sets for $\mu$ and $\check{\mu}$, respectively, for some $K_{0}>0$. Proposition \ref{prop:gouezelRW1} determines a constant $K>0$ (not depending on $k$ but only on $\mu$), a probability space $(\Omega, \Prob)$ for $\mu$ and a partition of $\Omega$ into pivotal equivalence classes that is independent of the backward steps $(\check{g}_{n})_{n>0}$ and such that \[
\Prob(\#\diffPivot(\w) \cap \{1, \ldots, n\} \ge n/K \, | \, g_{k+1}) \ge 1 - K e^{-k/K}, \,\, (n \ge k)
\]
and also another partition into (backward) pivotal equivalence classes that is independent of the forward steps $(g_{n})_{n>0}$ and such that \[
\Prob(\#\diffPivot(\check{\w}) \cap \{1, \ldots, n\} \ge n/K \, | \, \check{g}_{k+1}) \ge 1 - K e^{-k/K}. \,\, (n \ge k)
\]
We enumerate $\diffPivot(\w)$ by $\{j(1) < j(2) < \ldots\}$ and $\diffPivot(\check{\w})$ by $\{\check{j}(1) < \check{j}(2) < \ldots\}$. Let us now define the event $B_{k}$ in $\Omega$; $(\check{\w}, \w) \in B_{k}$ if: \begin{enumerate}
\item $\#\diffPivot (\w)  \cap \{1, \ldots, n\}\ge n/K$ for all $n \ge k/3$;
\item $\#\diffPivot(\check{\w}) \cap \{1, \ldots, n\} \ge n/K$ for all $n \ge k/3$;
\item for each $n \ge k$ and $m \ge k$, the following are $D_{0}$-semi-aligned: \[\begin{aligned}
&\big(o, \,\axes_{j(1)}(\w), \, \axes_{j(2)}(\w),  \, \ldots, \, \axes_{j(\lfloor 2n/3K\rfloor)}(\w), \,Z_{n} o\big),\\
&\big(o, \,\axes_{\check{j}(1)}(\check{\w}), \, \axes_{\check{j}(2)}(\check{\w}),  \, \ldots, \, \axes_{\check{j}(\lfloor 2m/3K\rfloor )}(\check{\w}), \,\check{Z}_{m} o\big);
\end{aligned}
\]
\item $\big(\bar{\axes}_{\check{j}(i)}(\check{\w}),\, \axes_{j(i)}(\w) \big)$ is $D_{0}$-aligned for some $i \le k/3K$.
\end{enumerate}
In the proof of Lemma \ref{lem:Devi} we proved that $\Prob(B_{k})$ decays exponentially in $k$. It remains to prove that $d(id, [\check{Z}_{m}, Z_{n}]) \le K'D_{k}$ for any $n, m>0$ and $(\check{\w}, \w) \in B_{k}$, where we set $K' \ge 8+ 1.5K + E_{1}$. From now on, we fix $k$. When $n \le k$, we automatically have \[
d(id, [\check{Z}_{m}, Z_{n}]) \le d(id, Z_{n}) \le \sum_{i=0}^{k} |g_{i}| \le K'D_{k}.
\]
Similarly, the desired inequality holds when $m \le k$. Now assume $n, m \ge k$. The sequence \[
\big(\check{Z}_{m}o, \overline{\axes}_{\check{j}(\lfloor2m/3K\rfloor)}(\check{\w}), \ldots, \overline{\axes}_{\check{j}(\lfloor k/3K\rfloor)}(\check{\w}), \axes_{j(\lfloor k/3K\rfloor)}(\w), \ldots, \axes_{j(\lfloor2n/3K\rfloor)}(\w), Z_{n}o \big)
\]
is $D_{0}$-semi-aligned, and hence $E_{0}$-aligned by Proposition \ref{prop:BGIPConcat}. Here, recall that the involved Schottky set is $(2/e)$-long enough and that $-\log \left( \frac{4}{e^{2}} \cdot \frac{1}{4} \right) = 2$. Hence, $M_{0}/E_{1} \ge 2$. By Proposition \ref{prop:weakBGIPConcat}, there exists $\tilde{p} \in [\check{Z}_{m}, Z_{n} ]$ whose distance to $\tilde{\axes}_{j(\lfloor k/3K \rfloor)}(\w)$ is at most
 \[
\begin{aligned}
&\sum_{l=\lfloor k/3K \rfloor+1}^{\lfloor2n/3K\rfloor} e^{-l- \lfloor k/3K \rfloor} \diam\big(\tilde{\axes}_{j(l-1)} (\w)\cup \tilde{\axes}_{j(l)}(\w)\big) + \sum_{l=\lfloor k/3K \rfloor+1}^{\lfloor2m/3K\rfloor} e^{-l -\lfloor k/3K \rfloor}  \diam\big(\tilde{\axes}_{\check{j}(l-1)}(\check{\w}) \cup \tilde{\axes}_{\check{j}(l)}(\check{\w})\big) \\
&+ e^{-(\lfloor2n/3K\rfloor -   \lfloor k/3K\rfloor )} \diam \big(\tilde{\axes}_{j(\lfloor2n/3K\rfloor)}(\w) \cup Z_{n} ) + e^{-(\lfloor 2m/3K\rfloor- \lfloor k/3K\rfloor )} \diam \big(\tilde{\axes}_{j(\lfloor2m/3K\rfloor)}(\check{\w}) \cup  \check{Z}_{m} \big) \\
&+ e^{-1} \diam \big(\tilde{\axes}_{j(\lfloor k/3K \rfloor)}(\w) \cup \tilde{\axes}_{\check{j}(\lfloor k/3K \rfloor)}(\check{\w})\big)+ E_{1}.
\end{aligned}
\]

Here, note that \[\begin{aligned}
\diam\big(\tilde{\axes}_{j(k-1)}(\w) \cup \tilde{\axes}_{j(k)})(\w)\big) \le \sum_{i = j(k-1) - M_{0}+1}^{j(k)} | g_{i}|, \\
\diam\big(\tilde{\axes}_{j(\lfloor2n/3K\rfloor)}(\w) \cup Z_{n}\big) \le \sum_{i=j(\lfloor2n/3K\rfloor) - M_{0}+1}^{n} |g_{i}|.
\end{aligned}
\]
Note also that $l - \lfloor k/3K \rfloor \ge \frac{1}{2} l$ for $l > \lfloor2 k/3K \rfloor$. Using these, we deduce \[
\begin{aligned}
\tilde{d}\big(\tilde{p}, \tilde{\axes}_{j(\lfloor k/3K \rfloor)}(\w)\big) 
&\le  \sum_{i = j(\lfloor 2k/3K \rfloor)+1}^{j(\lfloor2n/3K\rfloor)} 2e^{- \frac{1}{2}\min \{l >0: j(l) \ge i\} } |g_{i}| + \sum_{i = \check{j}(\lfloor 2k/3K \rfloor)+1}^{\check{j}(\lfloor2m/3K\rfloor)} 2e^{- \frac{1}{2} \min \{l>0 : \check{j}(l) \ge i\} } |\check{g}_{i}| \\ 
&+\sum_{i = j(\lfloor2n/3K\rfloor)+1}^{n}2 e^{-\frac{1}{2}\lfloor 2n/3K\rfloor} |g_{i}|  + \sum_{i=j(\lfloor2m/3K\rfloor)+1}^{m} 2e^{-\frac{1}{2}\lfloor 2m/3K\rfloor} |\check{g}_{i}|\\
&+ \sum_{i=1}^{j(\lfloor 2k/3K \rfloor)} 2|g_{i}|  +  \sum_{i=1}^{\check{j}(\lfloor 2k/3K \rfloor)} 2|\check{g}_{i}| + E_{1}.
\end{aligned}
\]
Since we have $j(\lceil i/K\rceil) \le i$ for each $i \ge k/3$, this is bounded by \[
\begin{aligned}
&2 \sum_{i=1}^{k} |g_{i}| + 2\sum_{i=1}^{k} |\check{g}_{i}| + 2\sum_{i=k+1}^{\lfloor2n/3\rfloor} e^{-i/2K } |g_{i}| + 2e^{-\lfloor n/3K\rfloor} \sum_{i=\lfloor 2n/3\rfloor + 1}^{n} |g_{i}| \\
 &+2\sum_{i=k+1}^{\lfloor2m/3\rfloor} e^{-i/2 K } |\check{g}_{i}| + 2e^{-\lfloor m/3K\rfloor} \sum_{i=\lfloor m/3 \rfloor+ 1}^{m} |\check{g}_{i}| + E_{1}.
 \end{aligned}
\]
Moreover, since $\diam(id \cup \tilde{\axes}_{j(\lfloor k/ 3K \rfloor)})$ is bounded by $\sum_{i=1}^{j(\lfloor k/3K \rfloor)} |g_{i}| \le \sum_{i=1}^{k} |g_{i}|$, we conclude \[
\tilde{d}(id, \tilde{p}) \le 4 \sum_{i=1}^{k}  \left(|g_{i}| + |\check{g}|_{i} \right) +2 \sum_{i=1}^{\infty} e^{-\frac{1}{3K}i} \left( |g_{i}| + |\check{g}|_{i} \right) + E_{1}. \qedhere
\]
\end{proof}

\begin{prop}\label{prop:weakDeviation}
Let $p>0$ and let $((\check{Z}_{n})_{n}, (Z_{n})_{n})$ be the (bi-directional) random walk generated by a non-elementary probability measure $\mu$ on $G$ with finite $p$-th moment. Then there exists $K>0$ such that \[
\E_{\mu}\left[ \sup_{n, m} \check{d}(id, [\check{Z}_{m}, Z_{n}])^{p}\right] < K.
\]
In particular, for almost every sample path $(\check{\w}, \w)$, every geodesic in $\{[\check{Z}_{m}, Z_{n}] : m, n > 0\}$ intersects the same bounded metric ball centered at $id$.
\end{prop}

\begin{proof}
Let $K' >0$ be as in Proposition \ref{prop:weakDeviDist}. Let $D_{k}$ be as defined by Equation \ref{eqn:weakDeviDistDK} and let \[
A_{k} := \Big\{ (\check{\w}, \w) : \textrm{$d(id, [\check{Z}_{n}, Z_{m}]) \le D_{k}$ for all $n, m \ge 0$} \Big\}.
\]
Then we have \begin{equation}\label{eqn:weakDeviBound}
\frac{1}{K'} \sup_{n, m} \check{d}(id, [\check{Z}_{m}, Z_{n}] ) \le \sum_{k=1}^{\min\{ m : (\check{\w}, \w) \in A_{m}\}} (|g_{k}|+ |\check{g}_{k}|) + \sum_{k=1}^{\infty} e^{-k/K'} |g_{k}| + \sum_{k=1}^{\infty} e^{-k/K'} |\check{g}_{k}|+1.
\end{equation}
Noting that $|x + y|^{p} \le (2 \max(|x|, |y|))^{p} \le |2x|^{p} + |2y|^{p}$ for $x, y > 0$, it suffices to bound  $\E[I_{i}^{p}]$ for:\[
\begin{aligned}
I_{1} &:= \sum_{k=1}^{\min\{ m : (\check{\w}, \w) \in A_{m}\}} |g_{k}|, \quad &I_{2} &:= \sum_{k=1}^{\min\{ m : (\check{\w}, \w) \in A_{m}\}}  |\check{g}_{k}|, \\
I_{3} &:= \sum_{k=1}^{\infty} e^{-k/K'} |g_{k}|, \quad &I_{4} &:= \sum_{k=1}^{\infty} e^{-k/K'} |\check{g}_{k}|.
\end{aligned}
\]
Observe the following: when $|g_{k}| e^{-k/2K'}$ is bounded by $M$ for all $k$, we have \[
I_{3} = \sum_{k=1}^{\infty} e^{-k/K'} |g_{k}| \le M \sum_{k=1}^{\infty} e^{-k/2K'} \le MC
\]
for $C = (1-e^{-1/2K'})^{-1}$. This means
\[\begin{aligned}
\E[I_{3}^{p} ] &\le C^{p} \E\left[ \left(\max_{k} e^{-k/2K'} |g_{k}|\right)^{p}\right] \le C^{p} \E\left[\sum_{k} (e^{-k/2K'} |g_{k}|)^{p}\right] \\
&= C^{p} \E_{\mu} |g|^{p} \cdot \sum_{k} e^{-kp/2K'} < + \infty.
\end{aligned}
\]
For $I_{1}$, recall the inequality $|t^{p} - s^{p}|\le 2^{p} (|t-s|^{p} + s^{p- n_{p}} |t-s|^{n_{p}})$ for each $t, s \ge 0$ and $p > 0$, where $n_{p} = p$ for $0 \le p \le 1$ and $n_{p} = 1$ otherwise. From this, we have \[
\E[ I_{1}^{p}] \le 2^{p}\sum_{k=0}^{\infty} \E \bigg[\bigg(|g_{k+1}|^{p} + \bigg( \sum_{i=1}^{k} |g_{i}| \bigg)^{p-n_{p}} |g_{k+1}|^{n_{p}} \bigg) 1_{A_{k}^{c}}\bigg].
\]
Since $\Prob\big (A_{k}^{c} \, \big| \, g_{k+1}\big) \le K' e^{-k/K'}$ by Proposition \ref{prop:weakDeviDist}, $\E \left(|g_{k+1}|^{p} 1_{A_{k}^{c}} \right) \le\left( \E_{\mu}|g|^{p}\right) \cdot K'e^{-k/K'}$ is summable. Moreover,  \[\begin{aligned}
 \E \bigg[ \bigg( \sum_{i=1}^{k} |g_{i}| \bigg)^{p-n_{p}} |g_{k+1}|^{n_{p}}  1_{A_{k}^{c}} \bigg] &\le \E \bigg[\bigg( \bigg( \sum_{i=1}^{k} |g_{i}| \bigg)^{p} c^{-n_{p}} + c^{p - n_{p}} \bigg)|g_{k+1}|^{n_{p}}  1_{A_{k}^{c}}\bigg]\\
 &\le c^{-n_{p}} \cdot k^{p+1} (\E_{\mu}|g|^{p})^{2} + K'c^{p-n_{p}} e^{-k/K'} \E_{\mu}|g|^{p} 
 \end{aligned}
 \]
 holds for $c = e^{-k/2pK'}$, which is summable for $k$. Hence $\E[I_{1}^{p}]$ is finite. The remaining terms $\E[I_{2}^{p}]$ and $\E[I_{4}^{p}]$ can be handled in a similar way.
\end{proof}

We obtain an analogous estimate for random walks with finite exponential moment. Because it follows from the proof of Corollary \ref{cor:minDeviExp} given Inequality \ref{eqn:weakDeviBound}, we omit the proof.

\begin{prop}\label{prop:weakDeviationExp}
Let $((\check{Z}_{n})_{n>0}, (Z_{n})_{n>0})$ be the (bi-directional) random walk generated by a non-elementary probability measure $\mu$ on $G$ with finite exponential moment. Then there exists $K>0$ such that \[
\E_{\mu}\left[ \operatorname{exp} \left(\frac{1}{K} \sup_{n, m} \check{d}(id, [\check{Z}_{m}, Z_{n}])\right)\right] < K.
\]
\end{prop}

Using Proposition \ref{prop:weakDeviation}, we obtain the uniform second moment deviation inequality for non-elementary probability measures on the mapping class group. Now employing Theorem 4.1 and 4.2 of \cite{mathieu2020deviation} and the proof of Theorem \ref{thm:LILStrong}, we establish Theorem \ref{thm:CLT}.

To prove Theorem \ref{thm:tracking}, for each $k \ge 0$ and $(\check{\w}, \w) \in \check{\Omega} \times \Omega$ we define the infinite geodesic $\Gamma_{k}(\check{\w}, \w)$ to be a subsequential limit of $\{[\check{Z}_{n+k}, Z_{n-k}] : n=1, 2, \ldots\}$ (which exists by Arzela-Ascoli and the second result in Proposition \ref{prop:weakDeviation}). Note that $\tilde{d}(Z_{k}, \Gamma_{0}(\check{\w}, \w))$ are identically distributed with $\tilde{d}(id, \Gamma_{k}(\check{\w}, \w))$, which are all dominated by $\tilde{d}(id, \sup_{n, m} [\check{Z}_{n}, Z_{m}])$. It follows that \[
\Prob\Big(\tilde{d}(Z_{k}, \Gamma_{0}(\check{\w}, \w)) \ge g(k) \Big)
\]
is summable for some $o(k^{1/p})$-function $g(k)$ ($K \log k$ for some $K>0$, resp.) when the underlying measure has finite $p$-th moment (finite exponential moment, resp.). By Borel-Cantelli, we deduce \[
\lim_{n}\frac{1}{n^{1/p}} \tilde{d}(Z_{n}, \Gamma_{0})= 0 \quad \Big(\limsup_{n} \frac{1}{\log n} \tilde{d}(Z_{n}, \Gamma_{0}) < K\textrm{, resp.}\Big).
\]

We conclude this paper by establishing Theorem \ref{thm:expBdMod}. Recall that Proposition \ref{prop:gouezelRW1} guaranteed the alignment of Schottky axes along  $[o, Z_{n} o]$ \emph{on $X$}, which led to the reverse triangle inequality for distances on $X$ (Lemma \ref{lem:SchottkyAlign}). We  now record the corresponding result for distances on $\tilde{X}$.

\begin{lem}\label{lem:SchottkyAlignWeak}
Let $0< \epsilon<1$ and let $S$ be a long enough, $\epsilon$-constricting $K_{0}$-Schottky set. Let $\tilde{x}, \tilde{y} \in \tilde{X}$ and let $\tilde{\kappa}_{1}, \ldots, \tilde{\kappa}_{N}$ be Schottky axes on $\tilde{X}$. If $(x, \kappa_{1}, \ldots, \kappa_{N}, y)$ is $D_{0}$-semi-aligned, then we have \[
\tilde{d}(\tilde{x}, \tilde{y}) \ge (1-\epsilon) \left( \diam_{\tilde{X}}(\tilde{x}, \tilde{\kappa}_{1}) + \sum_{i=2}^{n} \diam_{\tilde{X}} (\tilde{\kappa}_{i-1}, \tilde{\kappa}_{i}) + \diam_{\tilde{X}}( \tilde{\kappa}_{n}, \tilde{y}) \right)- 4 \sum_{i=1}^{n} \diam_{\tilde{X}} (\tilde{\kappa}_{i}).
\]
\end{lem}

\begin{proof}
Let $M_{0}$ be such that $S \subseteq G^{M_{0}}$. Note that $4 e^{-M_{0} / 2E_{0}} \le \epsilon<1$, which implies \[
\sum_{j\in \Z} e^{-|j-0.5| M_{0}/E_{0}} \le \frac{\epsilon}{4} \cdot \frac{1}{1 - (\epsilon/4)^{2}} \le \frac{\epsilon}{2}.
\]
Now let $\tilde{\kappa}$ be an arbitrary Schottky axis on $\tilde{X}$. Because $M_{0} > 2K_{1}E_{1}$ for $E_{1}$ as in Definition \ref{dfn:SchottkyLongWeak} and $\tilde{\kappa}$ is a $K_{1}$-quasigeodesics by Lemma \ref{lem:alignCoupled}, we have $\diam_{\tilde{X}}(\tilde{\kappa}) > E_{1}$.

Since $(x, \kappa_{1}, \ldots, \ldots, \kappa_{N}, y)$ is $D_{0}$-semi-aligned, they are $E_{0}$-aligned by Proposition \ref{prop:BGIPWitness}. Consequently, $(\tilde{x}, \tilde{\kappa}_{1}, \ldots, \tilde{\kappa}_{N}, \tilde{y})$ is $K_{1}E_{0}$-aligned by Lemma \ref{lem:alignCoupled}. Since the domains of $\kappa_{i}$'s are longer than $M_{0} \ge L_{1}$, we can obtain the points $\tilde{p}_{i}$'s on $[\tilde{x}, \tilde{y}]$ as described in Proposition \ref{prop:weakBGIPConcat}. 

We now have \[\begin{aligned}
\tilde{d}(\tilde{x}, \tilde{y}) &= \tilde{d}(\tilde{x}, \tilde{p}_{1}) + \sum_{i=2}^{n} \tilde{d}(\tilde{p}_{i-1}, \tilde{p}_{i}) + \tilde{d}(\tilde{p}_{n}, \tilde{y})\\
&\ge \left(\diam_{\tilde{X}} (\tilde{x}, \tilde{\kappa}_{1}) -\diam_{\tilde{X}}(\tilde{\kappa}_{1}) - \tilde{d}(\tilde{\kappa}_{1}, \tilde{p}_{1}) \right)+\left(\diam_{\tilde{X}} (\tilde{y}, \tilde{\kappa}_{n}) -\diam_{\tilde{X}}(\tilde{\kappa}_{n}) - \tilde{d}(\tilde{\kappa}_{n}, \tilde{p}_{n})\right) \\
 &+ \sum_{i=2}^{n} \Big(\diam_{\tilde{X}} (\tilde{\kappa}_{i-1}\cup \tilde{\kappa}_{i}) -\diam_{\tilde{X}}(\tilde{\kappa}_{i-1}) - \diam_{\tilde{X}} (\tilde{\kappa}_{i}) - \tilde{d}(\tilde{\kappa}_{i-1}, \tilde{p}_{i-1}) -\tilde{d}(\tilde{\kappa}_{i}, \tilde{p}_{i}) \Big).
 \end{aligned}
 \]
 
 Here, Proposition \ref{prop:weakBGIPConcat} tells us that \[\begin{aligned}
 \sum_{i=1}^{n} \tilde{d}(\tilde{\kappa}_{i}, \tilde{p}_{i})& \le \Big( \diam_{\tilde{X}}(\tilde{x}, \tilde{\kappa}_{1}) + \sum_{i=2}^{n} \diam_{\tilde{X}} (\tilde{\kappa}_{i-1}\cup  \tilde{\kappa}_{i}) + \diam_{\tilde{X}}( \tilde{\kappa}_{n}, \tilde{y}) \Big) \cdot \sum_{j \in \Z} e^{-|j-0.5| M_{0}/E_{1}} + E_{1} n\\
 &\le \frac{\epsilon}{2}  \Big( \diam_{\tilde{X}}(\tilde{x}, \tilde{\kappa}_{1}) + \sum_{i=2}^{n} \diam_{\tilde{X}} (\tilde{\kappa}_{i-1}\cup  \tilde{\kappa}_{i}) + \diam_{\tilde{X}}( \tilde{\kappa}_{n}, \tilde{y}) \Big) + \sum_{i=1}^{n} \diam_{\tilde{X}} (\tilde{\kappa}_{i}).
 \end{aligned}
 \]
Using this, we conclude \[
\begin{aligned}
 \tilde{d}(\tilde{x}, \tilde{y})&\ge (1-\epsilon) \Big( \diam_{\tilde{X}}(\tilde{x}, \tilde{\kappa}_{1}) + \sum_{i=2}^{n} \diam_{\tilde{X}} (\tilde{\kappa}_{i-1} \cup \tilde{\kappa}_{i}) + \diam_{\tilde{X}}( \tilde{\kappa}_{n}, \tilde{y}) \Big)- 4 \sum_{i=1}^{n} \diam_{\tilde{X}} (\tilde{\kappa}_{i}).\qedhere
\end{aligned}
\]
\end{proof}

\begin{cor}[{\cite[Lemma 4.14]{gouezel2022exponential}}]\label{cor:gouezelRW1CorWeak}
Let $\nu$ be a non-elementary probability measure on $G$ and let $(Z_{n})_{n}$ be the random walk generated by $\nu$. Then for each $\epsilon > 0$, there exists $C>0$ such that  \[
\Prob\Big ( |g Z_{n}| \ge (1-\epsilon)|g| - C\,\,\textrm{for all $n\ge 0$}\Big) \ge 1-\epsilon/2 \quad( \forall g \in G).
\]
\end{cor}

\begin{proof}
Let $S$ be a large, long enough and $\epsilon$-constricting $K_{0}$-Schottky set for $\mu$ in $G^{M_{0}}$, for some suitable  $K_{0}, M_{0}>0$. (This determines the constants $K_{1}, D_{0}, \ldots$ as in Definition \ref{dfn:SchottkyLongWeak}.)

As in the proof of Corollary \ref{cor:gouezelRW1Cor}, there exists $N>0$ independent of $g$ such that  \[
\Prob \left( \begin{array}{c} \textrm{there exists $i < N$ such that $\axes_{i}$ is a Schottky axis and}\\ \textrm{ $(g^{-1}o, \axes_{i}(\w), Z_{n}o)$ is $D_{0}$-semi-aligned for each $n \ge N$}\end{array} \right) \ge 1-\epsilon/4.
\]Also, when $(g^{-1} o, \axes_{i}, Z_{n}o)$ is $D_{0}$-semi-aligned, Lemma \ref{lem:SchottkyAlignWeak} implies that \[
\begin{aligned}
\tilde{d}(g^{-1}, Z_{n}) &\ge (1-\epsilon)\left( \diam_{\tilde{X}}\left(g^{-1} \cup  \tilde{\axes}_{i}\right) + \diam_{\tilde{X}} \left( \tilde{\axes}_{i} \cup Z_{n} \right) \right) - 4 \diam_{\tilde{X}}(\tilde{\axes}_{i})\\
&\ge (1-\epsilon) \tilde{d}\big(g^{-1}, \tilde{Z}_{i-M_{0}} \big)- 4 \cdot (K_{1}M_{0} +K_{1}) \\
&\ge (1-\epsilon) |g| - (1-\epsilon)\sum_{j=1}^{i-M_{0}} |g_{j}| - C'' \ge (1-\epsilon)|g| - \sum_{j=1}^{N} |g_{j}| - C'',
\end{aligned}
\]
where $C'' = 4( K_{1} M_{0} + K_{1})$. 
This bound also holds for $n \le N$: \[
\tilde{d}\left(g^{-1}, Z_{n}\right) \ge |g^{-1}| - |Z_{n}| \ge |g| - \sum_{j=1}^{N} |g_{j}|.
\]
Given these, the proof ends by taking large enough $C'>0$ such that \[
\Prob \bigg( \sum_{j=1}^{N} |g_{j}| \ge C' - C''\bigg) \le \epsilon/4. \qedhere
\]
\end{proof}

\begin{proof}[Proof of Theorem \ref{thm:expBdMod}]
As in the proof of Theorem \ref{thm:LDPStrong}, we first take $\epsilon>0$ such that \[
(1-\epsilon)^{4} \lambda(\mu) > L+\epsilon.
\] Let $S$ be a large, long enough and $\epsilon$-constricting Schottky set for $\mu$ with cardinality greater than $10/\epsilon$, and let \[
C'' = 4 \max_{s \in S} \diam_{\tilde{X}} \tilde{\Gamma}^{+}(s).
\] By Proposition \ref{prop:gouezelRWLDP}, there exists a non-elementary probability measure $\nu$, and for each sufficiently large $N$,  a partition $\mathscr{P}_{n, N, \epsilon}$ into $(n, N, \epsilon, \nu)$-pivotal equivalence classes for each $n$ such that \[
\Prob\left( \w : \frac{1}{2} \# \diffPivot^{(n, N, \epsilon)}(\w) \le (1- \epsilon) \frac{n}{2M_{0} N} \right)
\]decays exponentially in $n$. Let $C > 0$ be a constant for $\nu$ provided by Corollary \ref{cor:gouezelRW1CorWeak}: we have \[
\Prob_{\nu^{\ast m}}( h: |gh| \ge (1-\epsilon) |g| - C) \ge 1-\epsilon/2
\]
for each $g \in G$ and each $m > 0$. We now fix an $N$ such that $N > \frac{C+C''}{M_{0} \lambda(\mu) \epsilon}$. 

Let $\mathcal{E}$ be an equivalence class such that $\frac{1}{2} \# \diffPivot^{(n, N, \epsilon)}(\mathcal{E}) \ge (1-\epsilon) \frac{n}{2M_{0} N}$. For each $\w \in \mathcal{E}$ $(o, \axes_{j(1)}, \ldots, \axes_{j'(\#\diffPivot/2)}, Z_{n} o)$ is $D_{0}$-semi-aligned. Lemma \ref{lem:SchottkyAlignWeak} tells us that \[\begin{aligned}
|Z_{n}| &\ge \sum_{i=1}^{\#\diffPivot / 2} \big( (1-\epsilon) \tilde{d} \left(Z_{j(i)}, Z_{j'(i) - M_{0}}\right) - C'' \Big) = \sum_{i=1}^{\#\diffPivot / 2} \Big(|r_{i}| - C'' \Big) \quad (r_{i} := g_{j(i) +1} \cdots g_{j'(i) - M_{0}}).
\end{aligned}
\]
Since $r_{i}$'s are i.i.d. with \[
\E\big[|r_{i}| - M\big] \ge (1-\epsilon) \E_{\mu^{\ast 2M_{0} N}}\big[(1-\epsilon)|g| - C\big]  - C''\ge (1-\epsilon)^{3}  \cdot 2M_{0}N\lambda(\mu),
\]
we can find $K'>0$ not depending on $n$ such that \[
\Prob\left( \left.|Z_{n}|\le(1-\epsilon)^{4} \lambda n  \, \right| \, \mathcal{E}\right)\le K' e^{-n/K'} \quad(\forall n > 0).
\]
Summing up this conditional probability, we obtain the desired exponential bound.
\end{proof}

\medskip
\bibliographystyle{alpha}
\bibliography{random1}

\end{document}